\documentclass[12pt,a4paper]{article}
\usepackage{amsmath,amssymb,amsthm,euscript}
\usepackage{xcolor,graphicx,epsfig,caption,subcaption}
\usepackage{fancyhdr}
\usepackage[backgroundcolor=white]{todonotes}
\usepackage[colorlinks = true, urlcolor = blue, pdfborder= 0 0 0.5, urlbordercolor = blue ]{hyperref}
\usepackage{longtable,xspace,wrapfig,tabularx,lscape}
\usepackage{multirow}
\usepackage{rotating}
\usepackage{remreset}
\usepackage{hhline,caption}

\newtheorem{theorem}{Theorem}
\numberwithin{theorem}{section}
\newtheorem{lemma}[theorem]{Lemma}

\theoremstyle{definition}
\newtheorem{definition}[theorem]{Definition}

\newtheorem{remark}[theorem]{Remark}
\newtheorem{example}[theorem]{Example}
\newtheorem{proposition}[theorem]{Proposition}
\newtheorem{corollary}[theorem]{Corollary}
\newtheorem{notation}[theorem]{Notation}
\newcommand{\emr}[1]{\textcolor{black}{#1}}

\renewenvironment{proof}[1][\unskip]{
\par\noindent{\it Proof #1. }} { \mbox{}\hfill
$\blacksquare$ \par }

\normallineskip=1pt \normalbaselines \setbox\strutbox\hbox{\vrule
width 0pt height 6.4pt depth 1.6pt}

\makeatletter
\renewcommand{\@oddfoot}{\hfill ---~\thepage~--- \hfill}

\def\ds@whichfont{dsrom}
\relax

\DeclareMathAlphabet{\mathds}{U}{\ds@whichfont}{m}{n}

\makeatother

\newcommand\const{\mathord{\rm const}}

\newcommand\supp{\mathop{\rm supp}\nolimits }

\newcommand\eqdef{\buildrel{\rm def}\over=}

\newcommand{\bbR}{\mathbb{R}}

\newcommand{\gro}{G_1}
\newcommand{\gfi}{G_2}

\mathchardef\R"71B3 \mathchardef\Z"71B4 \mathchardef\C"71B2
\mathchardef\Q"71B1 \mathchardef\N"71B0 \mathchardef\k"71B9

\renewcommand\matrix[1]{\null\,\vcenter{\fns\normalbaselines\mathsurround0pt
    \ialign{\hfil$##$\hfil&&\enspace\hfil$##$\hfil\crcr
      \mathstrut\crcr\noalign{\kern-\baselineskip}
      #1\crcr\mathstrut\crcr\noalign{\kern-\baselineskip}}}\,}

\newcommand\dfill{\cleaders\hbox to 10pt{\hss.\hss}\hfill }

\newcommand\LR[3]{\setbox0\hbox{$#3$}\setbox1\hbox{$\left#1\vcenter{\copy0}
\right#2$}\dimen200\ht1\advance\dimen200 by -\ht0
\dimen201\ht1\advance\dimen201 by \dp1
\mathord{\mathopen{\lower\dimen200\hbox{$\left#1\vcenter to
\dimen201{}\right.$}}\copy0\mathclose{\lower\dimen200
\hbox{$\left#2\vcenter to \dimen201{}\right.$}}}}

\newcommand{\eps}{\varepsilon}
\newcommand{\Ds}{\mathcal{D}^{(2)}}
\newcommand{\Da}{\mathcal{D}^{(1)}}
\newcommand{\mS}{\mathcal{S}}
\newcommand\hatop[1]{ {\widehat #1}_\vartheta}
\begin{document}
\bibliographystyle{plain} 
%\title{xxxx}
\date{ }

\title{Uniform lower bounds on the dimension of Bernoulli convolutions}

\author{V. Kleptsyn, M. Pollicott and P. Vytnova\thanks{
The authors are very grateful to the anonymous referee for his helpful suggestions and comments. 
%\\
The first author is partly supported by ANR Gromeov (ANR-19-CE40-0007), by Centre
Henri Lebesgue (ANR-11-LABX-0020-01) and by the Laboratory of Dynamical Systems
and Applications NRU HSE, of the Ministry of science and higher education of the
RF grant ag. No. 075-15-2019-1931. The second author is partly
supported by ERC-Advanced Grant 833802-Resonances and EPSRC grant
EP/T001674/1. 
The third author is partly supported by
EPSRC grant EP/T001674/1. }}

\maketitle

\begin{abstract}
    In this note we present an algorithm to obtain a uniform lower bound on Hausdorff
    dimension of the stationary measure of an affine iterated function
    scheme with similarities, the best known example of which is Bernoulli
    convolution. The Bernoulli convolution measure~$\mu_\lambda$ is the probability measure
    corresponding to the law of the random variable 
    %For any  $\frac12<\lambda<1$ we consider the law of the
    %\emph{Bernoulli convolution} $\mu_{\lambda}$ given by the random variable of 
    $$
    \xi = \sum_{k=0}^\infty \xi_k\lambda^k, 
    $$
    where $\xi_k$ are i.i.d. random variables assuming values~$-1$ and~$1$ with
    equal probability and $\frac12 < \lambda < 1$.
    %The purpose of this note is to present an approach to obtain lower bounds on their dimensions, 
    %as well as lower bounds of dimensions for more general self-similar measures. 
    In particular, for Bernoulli convolutions we give a uniform lower bound
    $\dim_H(\mu_\lambda) \geq 0.96399$ for all $\frac12<\lambda<1$.
\end{abstract}

\tableofcontents

\section{Introduction}

In this note we will study stationary measures for certain types of iterated
function schemes. We begin with an important example.

\subsection{Bernoulli convolutions}

The study of the properties of Bernoulli convolutions  was greatly advanced  by
two influential papers of Paul Erd\"os from 1939 \cite{erdos1939} and 1940
\cite{erdos1940} and has remained an active area of research ever since.  We
briefly recall the definition: given $0 < \lambda <1$ we can associate the
Bernoulli convolution measure $\mu_\lambda$ on the real line 
corresponding to the distribution of the series
\begin{equation}
    \label{eq:rv}
    \xi = \sum_{k=0}^\infty \xi_k\lambda^k
\end{equation}
where $(\xi_k)_{k=1}^\infty$ are independent random variables assuming values
$\pm 1$ with equal probability.
Equivalently, this is the probability measure given by the weak-star limit  of the measures 
$$
\mu_\lambda  = \lim_{n\to +\infty} \frac{1}{2^n} \sum_{i_1, \cdots, i_n \in
\{0,1\}} \delta\left(\sum_{j=1}^n (-1)^{i_j} \lambda^j\right), 
$$
where $\delta(y)$ is the Dirac delta probability measure supported on $y$.
The properties of these measures have been studied in great detail. 
We refer the reader to recent surveys by Gou\"ezel~\cite{G18} and
Hochman~\cite{H18} for an overview of existing results. 

The properties of the measure $\mu_\lambda$ are very sensitive to the choice of $\lambda$.  
For example, if $0 < \lambda < \frac{1}{2}$ then $\mu_\lambda$ is supported on a
Cantor set and is singular with respect to Lebesgue measure, but if  $\lambda =
\frac{1}{2}$ then the measure $\mu_{1/2}$ equals  the normalized Lebesgue
measure on $[-2,2]$.
For  $\frac{1}{2} < \lambda < 1$ the situation is more subtle. In this case the measure
$\mu_\lambda$ is supported on the closed interval $\left[-\frac{1}{1-\lambda},
\frac{1}{1-\lambda}  \right]$.   
It was conjectured by Erd\"os in 1940~\cite{erdos1940}, and proved by Solomyak
in 1995~\cite{solomyak}, that for almost all $\lambda \in (\frac{1}{2},1)$ (with
respect to Lebesgue measure) the measure  $\mu_\lambda$ is absolutely
continuous. Recently Shmerkin~\cite{S14},~\cite{S15}, developing the method of
Hochman~\cite{hochman}, improved this result to show that   
the set of $\frac{1}{2} < \lambda < 1$ for which 
$\mu_\lambda$ is {\it not } absolutely continuous has zero Hausdorff dimension.
On the other hand, it  was shown  by Erd\"os in~\cite{erdos1939} that this
exceptional  set of values is non-empty. 

We will be concerned with another, though related, aspect of the Bernoulli
convolutions $\mu_\lambda$, namely their Hausdorff dimension. 
%\emr{(we recall the
%definition below, in Section~\ref{subsection:dimH}, see
%Definition~\ref{Hdim}).}
\begin{definition}
    \label{def:HDmeasure}
    The \emph{Hausdorff dimension} of a probability measure $\mu$ is defined by
    \begin{equation}\label{eq:dimHmu}
        \dim_H(\mu) : = \inf \{\dim_H(X) \mid \mbox{ $X$ is a Borel set with
        } \mu(X)=1\},
    \end{equation}
    where $\dim_H(X)$ stands for the Hausdorff dimension of a set $X$,
    see Section~\ref{subsection:dimH} for definition. 
\end{definition}
\noindent
Any measure $\mu$ which is  absolutely continuous with respect to Lebesgue
measure
automatically satisfies $\dim_H(\mu) =1$, and therefore the result of Shmerkin
implies that $\dim_H(\mu_\lambda)=1$ for all but an exceptional set of
parameters $\lambda$ of zero Hausdorff dimension.   
Furthermore, Varj\'u~\cite{V} recently proved a stronger result that
 $\dim_H(\mu_\lambda) =1$ for all transcendental $\lambda$.  
 
Therefore, it remains  to consider the set of algebraic parameter values.
It turned out that for certain class of algebraic numbers, namely, for
the reciprocals of Pisot numbers, it is possible to
compute Hausdorff dimension~$\dim_H(\mu_\lambda)$ explicitly, subject to
computer resources. 
We briefly recall the definition. 
\begin{definition}
Pisot number~$\beta$ is an algebraic number strictly greater than one all of
whose (Galois) conjugates, excluding itself, lie strictly inside the unit circle. 
\end{definition}
The Pisot numbers form a closed subset of $\mathbb R$ and have Hausdorff dimension strictly less than~$1$. 
The smallest Pisot number is $\beta_{min} = 1.3247\ldots$ (a root of $x^3 -
x - 1=0$).

The first progress on dimension of Bernoulli convolutions was made by Erd\"os in~\cite{erdos1939}, where he  
showed that if $\lambda$ is the reciprocal of a Pisot number, then  
$\mu_\lambda$ is not absolutely continuous. 
Garsia~\cite{G63} improved on the Erd\"os result by showing that
$\dim_H(\mu_\lambda) < 1$ whenever $\lambda$ is the reciprocal of a  Pisot
number. 
This phenomenon is called \emph{dimension drop} and it remains unknown whether Pisot numbers are the only numbers with this
property.

Alexander and Zagier estimated $\dim_H(\mu_\lambda)$ in the case that
$\lambda=\frac2{1+\sqrt5}$ was the reciprocal of the Golden mean and 
Grabner, Kirschenhofer and Tichy~\cite{GKT02} gave examples of explicit
algebraic numbers~$\lambda$, the so-called ``multinacci'' numbers for which the dimension
drop takes place. The values they computed are amongst the smallest known values
for the dimension of Bernoulli convolutions. For example, they estimated that 
\begin{align*}
\mbox{when } &\lambda^3 - \lambda^2 - \lambda - 1 = 0, && \mbox{ then
}\dim_H(\mu_\lambda) = 0.980409319534731\ldots \\
\mbox{when } &\lambda^4 - \lambda^3 - \lambda^2 -\lambda -1 = 0,  && \mbox{ then
} \dim_H(\mu_\lambda) = 0.986926474333800\ldots.
\end{align*}
%and these values are the smallest known dimensions.
The technique developed in~\cite{GKT02} has been subsequently extended to a
wider class of algebraic parameter values cf.~\cite{AFKP} and~\cite{HKPS19},  
however, the limitation of this method is that it requires studying each 
parameter value independently.

It is therefore a  basic  problem to get a \emph{uniform} lower bound on $\dim_H
\mu_\lambda$ for $\frac12<\lambda<1$ and to identify possible dimension drops.
A simplifying observation is that
%for algebraic~$\lambda$ we have that 
\begin{equation}
    \label{eq:sqrt}
\dim_H(\mu_\lambda) \geq \dim_H(\mu_{\lambda^2})
\end{equation}
 and thus if suffices to get a lower bound for $\frac{1}{2}< \lambda <
 \frac{1}{\sqrt{2}}$. 
 \begin{remark}
     The inequality~\eqref{eq:sqrt} is established, in
 particular, in~\cite[Proposition 2.1]{HS18} for algebraic parameter
 values~$\lambda$, but it is easy to see that it holds for all $\frac12\le
 \lambda \le 1$, since for any probability measure $\mu$ we have that
 $\dim_H \mu*\mu \ge \dim_H \mu$. 
 \end{remark}
Our first main result on the dimension of Bernoulli convolutions is a collection of
piecewise-constant uniform lower bounds over increasingly finer partitions of
the parameter space.
\begin{theorem}\label{t:main} 
\begin{enumerate}
\item\label{i:1} The dimension of Bernoulli convolutions $\mu_\lambda$ for any
    $\frac12 < \lambda < 1$ satisfies 
\[
%\dim_H \mu_{\lambda} \ge D_2(\mu_{\lambda}) \ge 0.9604.
\dim_H \mu_{\lambda} \ge G_0 := 0.96399.
\] 
\item\label{i:2} Moreover, the dimension of Bernoulli convolutions $\mu_\lambda$
    is roughly bounded from below by a piecewise-constant function $\gro$ with $8$
    intervals of continuity, %\forall \lambda \in (1/2,1) \quad \dim_H \mu_{\lambda} \ge D_2(\mu_{\lambda}) \ge G_{10}(\lambda),
%$\forall \lambda \in (1/2,1)$ 
$\dim_H \mu_{\lambda} \ge \gro(\lambda)$, where the values of~$\gro$ are given
in Table~\ref{tab:g1-bernoulli} for $0.5 \le \lambda \le 0.8$. 
\item\label{i:3} The previous bound can be further refined. The dimension of Bernoulli convolutions $\mu_\lambda$
    is bounded from below by a piecewise-constant function $\gfi$ corresponding
    to approximately~$10000$ intervals
%\forall \lambda \in (1/2,1) \quad \dim_H \mu_{\lambda} \ge D_2(\mu_{\lambda}) \ge G_{5000}(\lambda),
%$\forall \lambda \in (1/2,1)$ 
$\dim_H \mu_{\lambda} \ge \gfi(\lambda)$, where the graph of the function~$\gfi$
is presented in Figure~\ref{fig:plotBC}, with the particularly interesting
region $0.5 < \lambda < 0.575$ presented in Figure~\ref{fig:g2-bernoulli}.
\end{enumerate}
\end{theorem}

In the proof, we derive~\ref{i:2} from~\ref{i:3} and~\ref{i:1}
    from~\ref{i:2}, rather than establishing each estimate independently. 
    We choose to give the statement in three parts for the clarity of
    exposition. 
\begin{remark}
Our proof is computer assisted and these bounds are not sharp, at least in the
following sense: using a finer partition of the parameter space one
could obtain even better lower bounds. This of course requires more computer
time. %Our proof is computer assisted,
%and with further computer time the estimates could be refined.  
\end{remark}

\parbox[c][60mm][t]{70mm}{
%\begin{table}

\bigskip

\begin{tabular}{|c|c|}
  \hline
   Interval & $\gro$  \\
  \hline
    $[0.5000 , 0.5037 )$ &   $0.9900 $ \\%.500000      0.503750        0.9900 
    $[0.5037 , 0.5181 )$ &   $0.9800 $ \\%.503700      0.518148        0.9800 
    $[0.5181 , 0.5200 )$ &   $0.9700 $ \\%.518100      0.520000        0.9700
    $[0.5200 , 0.5430 )$ &   $0.9785 $ \\%.520000      0.543000        0.9785 
    $[0.5430 , 0.5451 )$ &   $0.9639 $ \\%.543210      0.545140        0.9600 
    $[0.5451 , 0.5527 )$ &   $0.9785 $ \\%.545140      0.552772        0.9785 
    $[0.5527 , 0.5703 )$ &   $0.9850 $ \\%.552770      0.570390        0.9850 
    $[0.5703 , 0.8000 )$ &   $0.9900 $ \\%.570300      0.580000        0.9900 
  \hline
\end{tabular}
\captionof{table}{Values of $\gro$.}
\label{tab:g1-bernoulli}
%\end{table}
}%
\parbox[c][60mm][t]{85mm}{
%\begin{figure}
\includegraphics[width=85mm,height=50mm]{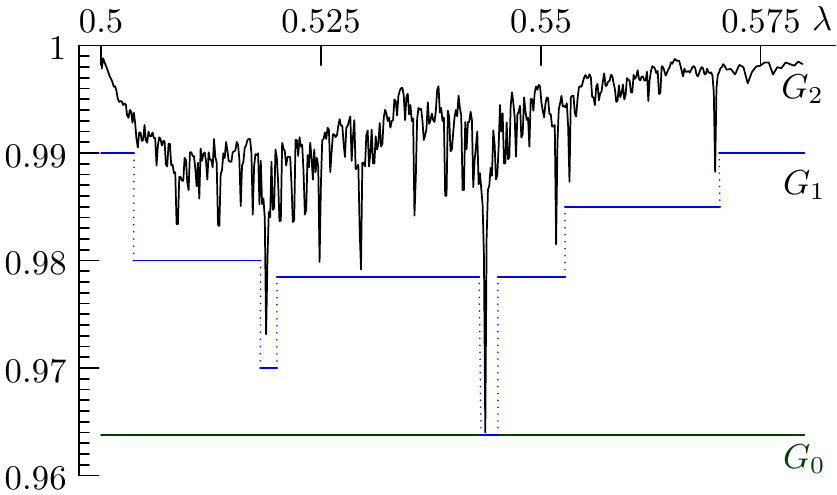}
\captionof{figure}{Plots of $G_0$, $\gro$ and~$\gfi$.}
\label{fig:g2-bernoulli}
%\end{figure}
}

The behaviour of the lower bound function~$\gfi$ appears to be
    quite intriguing, in particular, the largest dimension drops seem to correspond to
    the reciprocals of the limit points of the set of Pisot numbers,
    see Section~\ref{s:algebraicLambda} for further discussion and 
    Figure~\ref{fig:plotBC}, for detailed plots.
%    \footnote{\emr{The numerical data corresponding to this plot is available online in the supplementary materials to this paper, as well as at 
%    \emb{https://users.mccme.ru/polly/files/rawplotBCdata.dat}
%    }}

    To the best of our knowledge, the best result to date is due 
    Feng and Feng~\cite{FF21}; they obtained a global lower bound of
    $\dim_H(\mu_\lambda) \geq 0.9804085$.    
    They give an alternative approach for computing a lower bound for
    $\dim_H(\mu_\lambda)$, which uses the conditional entropy. 
    %which was announced at a conference in Cambridge in March~2019. 
    Three years earlier Hare and Sidorov~\cite{HS18} showed that $\dim_H(\mu_\lambda) \geq 0.82$.
    Their method depends on a result of Hochman and uses the fact that the dimension
    of~$\mu_\lambda$ can be expressed in terms of the Garsia entropy and most
    advances on this problem are based on this idea.

Our approach is different to both and is rooted in connection between iterated function schemes and random processes. 
In addition to uniform estimates, it allows us to compute good lower bounds on
$\dim_H(\mu_\lambda)$  for individual values~$\lambda$.

The following set of algebraic numbers, intimately related to Pisot numbers, is also
extensively studied.
\begin{definition}
    A Salem number is an algebraic integer $\sigma >1$ of degree at
    least 4, conjugate to $\sigma^{-1}$, all of whose conjugates, excluding
    $\sigma$ and $\sigma^{-1}$, lie on the unit circle.
\end{definition}
We refer to a survey by Smyth~\cite{Sm15} for an introduction to
the topic. 
The set of limit points of Salem numbers contains the Pisot numbers. 
We have computed the lower bound for the reciprocals of Salem numbers,
thus providing a partial supporting evidence that there is no dimension drop for
these parameter values.
\begin{theorem}
\label{thm:salemnum}
\begin{enumerate}
    \item For every one of the $99$ values  $\frac{1}{2} < \lambda < 1$ which is the reciprocal of  a Salem number of degree at most $10$
one has that  $\dim_H(\mu_\lambda) \geq 0.98546875$. Detailed estimates are
tabled in Appendix~\ref{ap:salem}. 

\item One can also consider the $47$ known so called small Salem numbers $\frac{10}{13}<  \lambda < 1$ and show that
$\dim_H(\mu_\lambda) \geq 0.999453125$.  Lower bounds on the dimensions of the
Bernoulli convolutions for the reciprocals of small Salem numbers are presented
in Appendix~\ref{ap:smallsalem}.
\end{enumerate}
\end{theorem}

Another conjecture suggests that there exists~$\varepsilon>0$ such that for any
$\lambda\in (1-\varepsilon,1)$ the dimension of the
measure~$\mu_\lambda$ equals~$1$.
In particular,  Breulliard---Varj\'u~\cite{BV19} showed  that
there exists $\varepsilon > 0$ so that 
$\dim_H(\mu_\lambda) =1$ for $1-\varepsilon < \lambda < 1$ \emph{under the 
assumption that the Lehmer's conjecture holds}.
The Lehmer's conjecture states that the Mahler measure
of any nonzero noncyclotomic irreducible polynomial with integer coefficients
is bounded below by some constant $c > 1$. It implies, in particular, that there
exists a smallest Salem number.  

As another application of our method, we give an asymptotic for the lower bound
of $\dim_H \mu_\lambda$ as $\lambda \to 1$ in Section~\ref{s:abounds}. More precisely,
we establish the following result. 
\begin{theorem}   
    \label{thm:near1}
There exist $c > 0$ and $\varepsilon  > 0$ so that 
$\dim_H(\mu_\lambda) \geq 1 - c (1-\lambda)$ for $1-\varepsilon < \lambda < 1.$
\end{theorem}

The Bernoulli convolutions are a special case of a far more general construction of self-similar measures, 
which we describe next. % subsection. 
\subsection{Iterated function schemes with similarities}
\label{subsection:ifs}

Let $k$ be fixed. Given $0 < \lambda < 1$ and $\bar c \in \mathbb R^n$, consider a collection $\mS =
\{f_j, j = 1,\ldots,k\}$ of $k$ contraction similarities  defined by 
$$
f_j \colon \bbR \to \bbR; \qquad f_j(x) = \lambda x \emr{\, +\,} c_j, \mbox{ for } 1\leq j \leq k.
$$
Let $\bar p = (p_1, \cdots, p_k)$ be a probability vector where $0 < p_j < 1$ and
$\sum\limits_{j=1}^k p_j=1$.

\begin{definition}
    \label{def:IFS}
    We call a triple $\mS(\lambda, \bar c, \bar p)$ an iterated function scheme of
    similarities. We will omit the dependence on $\bar c$ and $\bar p$ in the sequel
    when it leads to no confusion. 
\end{definition}

\begin{definition}
A  probability measure $\mu$ is called a {\it stationary measure} for the
contractions $f_1, \cdots, f_k$ and the probability vector $\bar p$
if it satisfies
$$
\mu = \sum_{j=1}^k p_j (f_j)_*\mu,
$$
i.e., $\int F(x) d\mu(x) = \sum\limits_{j=1}^k p_j \int F(f_jx) d\mu(x)$, for
all bounded continuous functions~$F$. % \in C(\mathbb R)$.
\end{definition}
\noindent The existence and the uniqueness of stationary measures in this seting follows from the
work of Hutchinson~\cite{hutchinson}. 

In this note we are particularly concerned with the following two systems.
The first one has Bernoulli convolution as the stationary measure.  
\begin{example}[Function scheme for Bernoulli convolutions]
    \label{ex:BCsystem}
%This general setup covers the particular case of Bernoulli convolutions.
Given a real number $\frac12<\lambda<1$ consider the iterated function scheme
 of two maps $f_0$, $f_1$ given by $f_j(x) = \lambda x + j$, $j=0,1$ and 
probability vector $\bar p = \left(\frac12,\frac12\right)$.
Then the stationary measure $\mu = \mu_\lambda$ corresponds to the distribution of the random variable 
$$
\sum_{k=0}^\infty \eta_k \lambda^k,
$$
where $\eta_k$ are i.i.d. assuming values $0$ and $1$ with equal probability. 
This agrees with formula~\eqref{eq:rv} up to the change of variables $\xi_k = 2\eta_k-1$.
\end{example}

\begin{example}[$\{0,1,3\}$-system]\label{ex:013}
We can consider  the  contractions  $f_1, f_2, f_3: \mathbb R \to \mathbb R$ defined by 
$$
f_1(x) = \lambda x, \quad
f_2(x) = \lambda x+1, \quad
f_3(x) = \lambda x+3,
$$
and the probability vector $p = (\frac{1}{3}, \frac{1}{3}, \frac{1}{3})$.
For the corresponding stationary measure $\mu_{\lambda}^{0,1,3}$ it is known
that for almost all $\frac{1}{4} \leq \lambda \leq \frac{1}{3}$ with respect to
Lebesgue measure we have 
\[
\dim_{H}(\mu_{\lambda}^{0,1,3}) = \frac{\log 3}{\log \lambda^{-1}}
\]
(this equality also holds for all $\lambda<\frac{1}{4}$) and for almost all
$\frac{1}{3} \leq \lambda \leq \frac{2}{5}$  with respect to Lebesgue measure
we have  $\dim_{H}(\mu^{0,1,3}_\lambda) = 1$; see~\cite{KSS95},~\cite{PS95}.
\end{example}

The next theorem provides a lower bound for~$\dim_H(\mu_\lambda^{0,1,3})$.

\begin{theorem}\label{thm:dim013}
The dimension of the stationary measure $\mu_{\lambda}^{0,1,3}$ for the
$\{0,1,3\}$-system has the lower bounds
\begin{enumerate}
    \item\label{013i} For any $\lambda \in [\frac{1}{4},\frac{2}{5}]$ we have that
%\[
 %\dim_H(\mu_{\lambda}^{0,1,3}) \geq D_2(\mu_{\lambda}^{0,1,3}) \ge \frac{\log 3}{\log(1/\lambda)} - 0.2.
 $ \dim_H(\mu_{\lambda}^{0,1,3}) \ge G^{0,1,3}_0 := \min\left\{\frac{\log
 3}{\log \lambda^{-1}},1\right\} - 0.2.$
%\]
\item\label{013ii} Moreover, $\dim_H(\mu_\lambda)$ is bounded from below by a piecewise-continuous
    function 
    $$
    \gro^{0,1,3}(\lambda)|_{I_k} = \min\left\{\frac{\log 3}{\log
    \lambda^{-1}},1\right\}-c_k
    $$
    with $11$ intervals of continuity $I_k$, $k = 1, \ldots 11$ which are given
    in Table~\ref{tab:013}, together with the corresponding values $c_k$. In
    other words, for any $\lambda \in [\frac14, \frac25]$
    $$
     \dim_H(\mu_{\lambda}^{0,1,3}) \ge \gro^{0,1,3}(\lambda).
    $$ 
 \item\label{013iii} The estimate from part~\ref{013ii} can be refined further. The dimension
    $\dim_H(\mu_\lambda^{0,1,3})$ is bounded from below by a piecewise-continuous
    function $\gfi^{0,1,3}$ with approximately~$10000$ intervals of continuity,
$\dim_H(\mu_{\lambda}^{0,1,3}) \ge \gfi^{0,1,3}(\lambda)$.
%\[ \dim_H(\mu_{\lambda}^{0,1,3}) \geq D_2(\mu_{\lambda}^{0,1,3}) \ge G^{0,1,3}_{10000}. \]
The graph of the function~$\gfi^{0,1,3}$ is presented in Figure~\ref{fig:013}.
\end{enumerate}
\end{theorem}

\parbox[c][70mm][t]{55mm}{

\smallskip

\begin{tabular}{|c|c|}
    \hline
    Interval $I_k$ & $c_k$ \\
    \hline
$[0.2500,     0.2630]$&$  0.0350 $ \\
$[0.2630,     0.2650]$&$  0.0550 $ \\
$[0.2650,     0.2800]$&$  0.0350 $ \\
$[0.2800,     0.2820]$&$  0.0650 $ \\
$[0.2820,     0.2980]$&$  0.0350 $ \\
$[0.2980,     0.3210]$&$  0.0850 $ \\
$[0.3210,     0.3320]$&$  0.1100 $ \\
$[0.3320,     0.3350]$&$  0.2000 $ \\
$[0.3350,     0.3450]$&$  0.1100 $ \\
$[0.3450,     0.3670]$&$  0.0800 $ \\
$[0.3670,     0.4045]$&$  0.0400 $ \\
    \hline
\end{tabular}   
\captionof{table}{Table of values of $c_k$.}
\label{tab:013}
}%
\parbox[c][75mm][t]{105mm}{
\includegraphics[width=105mm,height=65mm]{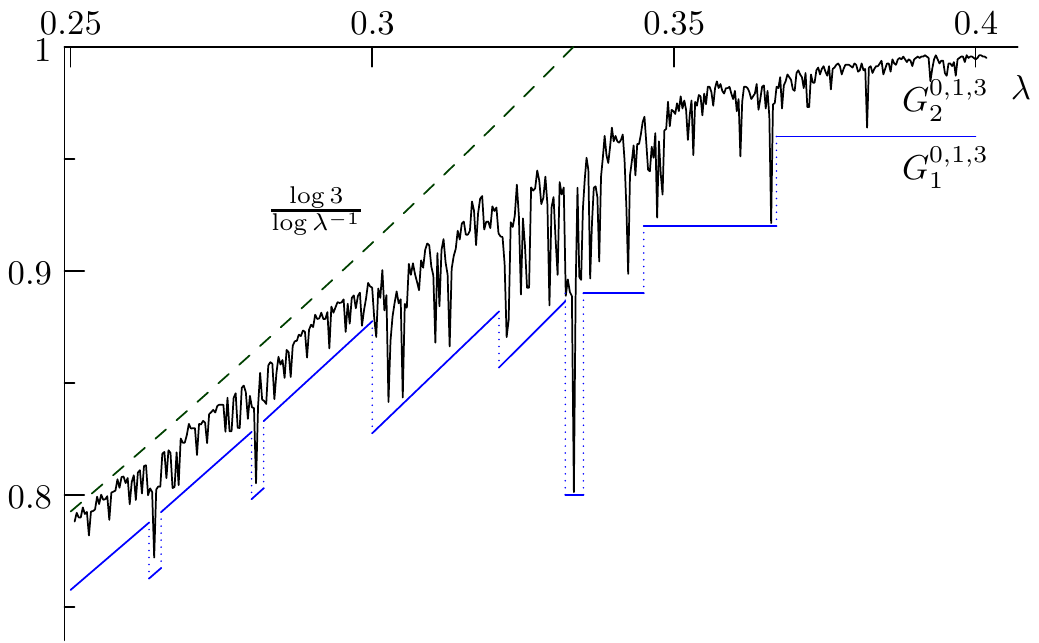}
\captionof{figure}{Plots of $G^{0,1,3}_1$ and~$G^{0,1,3}_{2}$.}
\label{fig:013}
}
\begin{remark}
An alternative version for (i) could be:
For any $0.25<\lambda<0.4$ we have that $\dim_H \mu_\lambda^{0,1,3} \ge
G_0^{0,1,3}(\lambda)$, where 
$$
G_0^{0,1,3}(\lambda) := 
\begin{cases} 
    \frac{\log 3}{\log \lambda^{-1}}-0.11, & \mbox{ if }
    0.25<\lambda<0.3210; \\ %\mbox{ or } 0.3250<\lambda < 0.4; \mbox{ and } \\
    \frac{\log 3}{\log \lambda^{-1}}-0.2, &\mbox{ if } 0.3210<\lambda < 0.3250; \\
    0.89, &\mbox{ otherwise. }
\end{cases}
$$
In particular, we see that the largest dimension drop seems to take place at $\lambda =
\frac13$. For this parameter value the dimension can be computed explicitly~\cite{KV2021} following the method
of~\cite{GKT02}, more precisely, 
$$
\dim_H\mu_{1/3}^{0,1,3} =0.83703915049\pm10^{-10}.
$$
\end{remark}

As in the case of Bernoulli convolutions, the biggest dimension drops
appear to correspond to the reciprocals of the limit points of hyperbolic
numbers\footnote{An algebraic number is called hyperbolic, if all its Galois
conjugates lie inside the unit circle}. However,
in contrast to the Pisot numbers in the interval $(1,2)$, the limit set of
hyperbolic numbers in the interval $\left(\frac52,4\right)$ is not very well studied. 
We give detailed plot of $G_2^{0,1,3}$ in Figure~\ref{fig:plot013} and discuss its feautures
in Section~\ref{s:algebraicLambda}. 

We obtain lower bounds for the Hausdorff dimension of Bernoulli convolutions and
for the stationary measures of the $\{0,1,3\}$-system using the same method,
which we outline in the next section. 

\subsection{Approach to lower bounds for Hausdorff dimension}
\label{ss:approach}
The Hausdorff dimension of a measure is an important characteristic which is
generally difficult to estimate, both numerically and analytically. 
We introduce two alternative characteristics of dimension type, namely, the
correlation dimension and the Frostman dimension, which are easier to
estimate and give a lower bound on the Hausdorff dimension. Whilst the numerical
results suggest that in the case of iterated function schemes with similarities 
the Frostman dimension and Hausdorff dimension behave very different, the
correlation dimension appear to exhibit the same dependence on parameter values 
as expected from the Hausdorff dimension. 

To sum up, our approach is the following: 

\indent\parbox{0.8\textwidth}{
\begin{enumerate}
    \item[Step 1:] Replace the Hausdorff dimension with the correlation
        dimension or the Frostman dimension; 
    \item[Step 2:] Compute the lower bound for the correlation dimension or the
        Frostman dimension. 
\end{enumerate}
}

\subsubsection{Affine iterated function schemes with similarities}

We begin by defining the correlation dimension which bounds the Hausdorff
dimension from below (cf. Lemma~\ref{l:inequality}). It has been introduced
in~\cite{PGH83} as a characteristic of dimension type. 
The notion was subsequently formalised by Pesin in~\cite{P93}, see
also~\cite{CHY97} and~\cite{SS98}. We will give a formal definition later in
Section~\ref{subsection:dimC}.

We proceed by introducing one of our main tools, a \emph{symmetric
diffusion operator} associated to an iterated function scheme. 

Let $\mS(\lambda,\bar c, \bar p)$ be an iterated function scheme of
similarities. Assume that $J_\lambda \subset \mathbb R$ is an interval
such that $f_i(J_\lambda)\subset J_\lambda$ for all $i=1,\dots,k$. 
For the invariant measures $\mu_\lambda$ we have that $\supp \mu_\lambda \subset
J_\lambda$.  

We say that an interval $J$ is $\lambda$-\emph{admissible} for an iterated function
scheme~$\mS(\lambda,\bar c,\bar p)$ if 
\begin{equation}
\label{eq:admit}
\mbox{ Interior}(J) \supset \overline{ \{x - y \mid x,y \in J_\lambda \} }
\end{equation}
This is illustrated in Figure~\ref{fig:admissible}.
Given a (possibly infinite) set of parameter values~$\Lambda\subset[0,1]$ we say that the interval $J$ is
{\it $\Lambda$-admissible} if it is an admissible interval for all $\lambda\in\Lambda$. 

\begin{figure}
    \centering
    \includegraphics{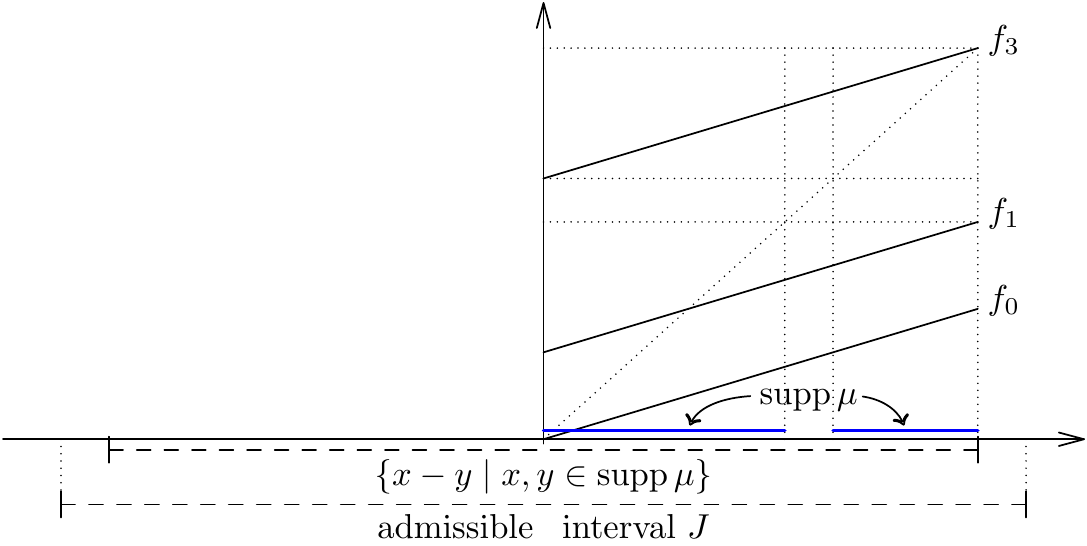}
    \caption{An iterated function scheme of three similarities $f_0(x) =
    \frac{x}{3}$, $f_1(x)=\frac{x}3+1$, and $f_3(x)=\frac{x}{3}+3$
    and an admissible interval~$J$.}
    \label{fig:admissible}
\end{figure}

\begin{definition}
    \label{def:symdif}
Given an iterated function scheme of similarities~$\mS(\lambda,\bar c, \bar p)$ for any $\alpha\in (0,1)$ we define the \emph{symmetric diffusion operator} $\Ds_{\alpha,\mS}$ by
\begin{equation}
    \label{eq:dif-twoway}
[\Ds_{\alpha,\mS}\psi](x) := \lambda^{-\alpha} \cdot \sum_{i,j=1}^k p_i p_j \cdot \psi\left(\frac{x -c_i+c_j}{\lambda}\right).
\end{equation}
We consider this operator to be acting on the space of all functions on the
real line, however the subset of nonnegative functions 
$$
\left\{\emr{\psi}: \mathbb R \to \mathbb R^{+} \mid \supp \emr{\psi} \subseteq 
\overline{\{ x-y \mid x, y \in J_\lambda \}} \right\}
$$ 
is invariant with respect to~$\Ds_{\alpha,\mS}$. 
\end{definition}

\begin{remark}
Although the difference between two operators $\Ds_{\alpha_1,\mS}$ and
$\Ds_{\alpha_2,\mS}$ is in scaling factor only, we prefer to keep this factor
as a part of the definition. 
\end{remark}
We are now ready to state a key result, which is the basis for our numerical
method.
\begin{theorem}\label{t:certificate-D2}
Let~$\mS(\lambda,\bar c, \bar p)$ be an iterated function scheme of
similarities.
Assume that for some $\alpha>0$ there exists an admissible compact interval $J \subset
\mathbb R$, a function $\psi:\bbR\to\bbR^+$ with $\supp \psi \subset J$ which is
positive and bounded away from~$0$ and from infinity on~$J$, such that for any
$x \in J$
\begin{equation}
    \label{eq:D2-test}
    [\Ds_{\alpha,\mS}\psi](x) < \psi(x).
\end{equation}
Then the correlation, and hence the Hausdorff, dimension of the $\mS$-stationary
measure $\mu$ is bounded from below by~$\alpha$: 
\[
\dim_H (\mu) \ge D_2(\mu)\ge \alpha.
\]
\end{theorem}
Theorem~\ref{t:certificate-D2} allows us to obtain rigorous lower
estimates for the correlation dimension~$D_2(\mu)$  of the stationary measure
$\mu$ for a single parameter value $\lambda$ (and thus for the Hausdorff
dimension $\dim_H\mu$), once a suitable test function~$\psi$ is found. 
This also provides us with a way to find an asymptotic lower bound and to prove
Theorem~\ref{thm:near1}. 

\begin{example}
    To illustrate the way Theorem~\ref{t:certificate-D2} is applied, we may choose
    $\lambda=0.75$, a function $\psi(x) = 1-0.2|x|$ and to apply the operator
    $\mathcal D^{(2)}_{0.2,\mathcal S}$. It is clear that we may choose $J =
    [-4,4]$. 
    Then Figure~\ref{fig:thm5ex} shows that 
    $\mathcal D^{(2)}_{0.2,\mathcal S} \psi (x) < \psi(x)$ and therefore $\dim_H \mu \ge 0.2$. 
    \begin{figure}
        \centerline{ \includegraphics{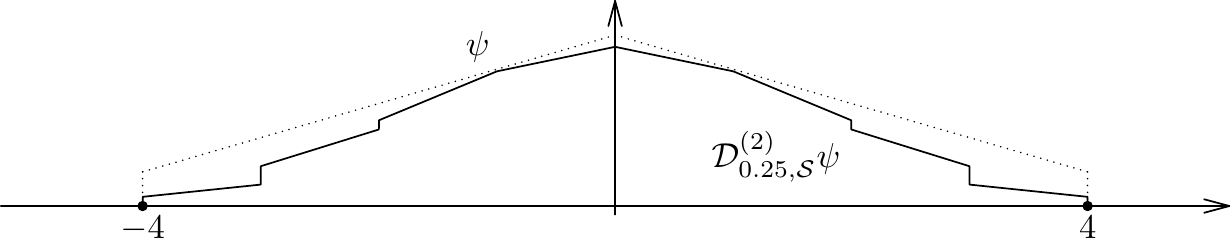} }
        \caption{Image of the function $\psi = 1 - 0.2|x|$ is strictly smaller
        than $\psi$.} 
        \label{fig:thm5ex}
    \end{figure}
\end{example}

We next want to adapt Theorem~\ref{t:certificate-D2} to prepare for a computer-assisted proof of
Theorems~\ref{t:main},~\ref{thm:salemnum} and~\ref{thm:dim013}. In Section~\ref{ss:efbounds-2} we modify
the operator $\Ds_{\alpha,\mS}$ to obtain an operator~${\mathcal
D}_{\alpha,\Lambda,\mathcal J}$ which preserves a subspace of piecewise constant functions,
and amend Theorem~\ref{t:certificate-D2} so that 
a common test function can be used for an open set of parameter
values~$\Lambda=(\lambda-\varepsilon,\lambda+\varepsilon)$.
This adaptation allows us to choose the test function to be piecewise constant on
intervals with rational endpoints and to verify the hypothesis of
Theorem~\ref{t:certificate-D2} numerically,
%Afterwards in~\S\ref{ss:efbounds-2}, we modify the operator $\widehat \Ds_{\alpha,\mS}$ further, and
%rework Theorem~\ref{t:certificate-D2} again, so that the same
%test function can be used for an open set of parameters $\lambda$. 
thus providing us with a means to obtain a uniform lower bound for the (correlation, and hence Hausdorff) 
dimension of the corresponding stationary measures.  

Afterwards, in Section~\ref{ss:efbounds-1} we give an iterative procedure to
construct a test function for the operators~$\Ds_{\alpha,\mS}$ and
${\mathcal D}_{\alpha,\Lambda,\mathcal J}$. 

A natural question arises. Assume that $D_2(\mu_\lambda)>\alpha$. Does there
exist a test function~$\psi$ so that~\eqref{eq:D2-test} holds? 
The next result gives an affirmative answer.
\begin{theorem}\label{t:finding-D2}
Let $\mu_\lambda$ be the unique stationary measure of a scheme of contraction
similarities~$\mS(\lambda)$. Then for any $\alpha<D_2(\mu_\lambda)$ the
hypothesis of Theorem~\ref{t:certificate-D2} holds. In other words, %for a sufficiently small $\eps>0$
there exists an admissible interval $J$, a piecewise constant function $\psi$
with $\supp \psi \subset J$ which is positive and bounded away from 0 and from
infinity on $J$ and such that for any $x \in J$
\begin{equation*}
%    \label{eq:D2-test}
%\forall x\in J \quad 
[\Ds_{\alpha,\mS}\psi](x) < \psi(x).
\end{equation*}
\end{theorem}
%\begin{remark}
As the reader will see, the technique  in the proof of Theorem~\ref{t:certificate-D2} 
exploits the fact that the maps are similarities with the same scaling coefficient.
In the next section~\ref{ss:ingen} we generalise the method to study other types of iterated function schemes at the expense
of weaker estimates. 
%\end{remark}

\subsubsection{General uniformly contracting schemes}
\label{ss:ingen}
%estimates for the global dimension~$D_1$}

Let us denote by~$B_r(x)$ a neighbourhood of a point~$x$ of radius~$r$.   

We will be concerned with iterated function schemes~$\mathcal
T(\bar f, \bar p, J)$, where $J \subset \mathbb R$ is a compact
interval, $\bar f = (f_1, \ldots, f_n)$ is a
finite collection of uniformly contracting $C^{1+\varepsilon}$ diffeomorphisms of~$\mathbb
R$, which preserve the interval~$J$, i.e. $f_j (J) \subset J$ for $1 \le j \le
n$  and $\bar p$ is a probability vector.

Following Hochman~\cite[\S4.1]{H12}, we say that the measure~$\mu$ is \emph{$\alpha$-regular}, if there exists 
a constant $C$ such that for any $r>0$ and any $x$ we have that 
\begin{equation}
    \label{eq:regmeasure}
    \mu(B_r(x))< C r^\alpha.
\end{equation}
One of the examples of $\alpha$-regular measures are Bernoulli convolutions~\cite[Proposition 2.2]{FL09}.  
%By analogy with lower pointwise dimension of a probability measure~$\mu$ at $x$ defined by
%cf.~\cite{P93},
%$$
%\dim_\mu(x) := \underline{\lim}_{r \to 0} \frac{\log \mu( (x-r,x+r)) }{\log r} 
%$$
We introduce the following dimension-type characteristic of a compactly
supported probability measure~$\mu$ on $\mathbb R$, which is sometimes referred
to as the \emph{Frostman dimension}~\cite{FFK20} (in the context of $\mathbb
R^n$) or the \emph{lower Ahlfors dimension} (in the context of general separable
metric spaces).  
%of a compactly supported
%probability measure~$D_1(\mu)$ 
It is defined as supremum of the regularity exponents: 
$$
D_1(\mu): =  \sup\{\alpha \mid \exists C: \quad \forall x,r \quad
\mu\left( B_r(x) \right) <Cr^{\alpha} \}. %\tag{\ref{def:D1}.1} \label{eq:D1.1} \\ 
$$
\begin{remark}
    We would like to warn the reader that the Frostman dimension doesn't
    satisfy all conditions which a dimension of a measure is expected to satisfy,
    in particular, it is not closed under countable unions. We will see in
    Lemma~\ref{lem:d2d1} that $D_1(\mu) \le D_2(\mu)$ for any probability
    measure~$\mu$. It is not hard to show that it is also a lower bound for the
    packing dimension, as well as other dimensions which can be defined using
    the local dimension. 
    %We conjecture that this inequality holds for other
    %dimensions, too. 
\end{remark}
%Borel measures for which the inequality above holds are sometimes referred to as Frostman
%measures. 

A pair of complementary results, Theorems~\ref{t:certificate-D1} 
and~\ref{t:finding-D1} below allow one to get a lower bound on the
\emph{Frostman} dimension~$D_1(\mu)$ for the stationary measure of an iterated
function scheme~$\mathcal T(\bar f, \bar p, J)$  in terms of an associated linear operator.

%As in the previous subsection, in order to formulate them we need to introduce a diffusion operator. 
%Namely, 
\begin{definition}
Given an iterated function scheme $\mathcal T(\bar f, \bar p, J)$ 
for any $\alpha\in (0,1)$ we define the associated \emph{asymmetric diffusion
operator} $\Da_{\alpha,\mathcal T}$ by
\begin{equation}
\Da_{\alpha,\mathcal T}[\psi](x) := \sum_{j=1}^n p_j \cdot
|(f_{j}^{-1})'(x)|^{\alpha}  \cdot \psi(f_j^{-1}(x)).
\label{eq:dif-oneway}
\end{equation}
We consider this operator to be acting on the space of all functions on the
real line, although it preserves nonnegative functions supported on~$J$. 
\end{definition}
\begin{remark}
    Comparing~\eqref{eq:dif-oneway} with~\eqref{eq:dif-twoway} we see that $\Ds_{\alpha,\mS}$ for
    Bernoulli convolution system described in Example~\ref{ex:BCsystem}
    corresponds to the operator~$\Da_{\alpha,\mathcal T}$ for the system of three
    contractions 
    $$
    \mathcal T:=\{f_1(x) = \lambda x - 1, f_2(x)=\lambda x, f_3(x)
    = \lambda x+1\}
    $$ 
    and probability vector $\overline p = (0.25,0.5,0.25)$. 
\end{remark}
By analogy with $\Ds_{\alpha,\mS}$, the operator $\Da_{\alpha,\mS}$
gives us a way to obtain a lower bound for the Frostman dimension. 

We denote by $B_r(J)$ the closed neighbourhood of the interval $J$ of radius~$r$.
%We need however to modify the definition of an admissible interval first. 

\begin{theorem}
    \label{t:certificate-D1}
Assume that for some $\alpha>0$ there exist $r>0$ and a function $\psi:\bbR\to\bbR_+$,
supported on~$B_r(J)$, positive on $B_r(J)$ and bounded away from $0$ and
from infinity on $B_r(J)$, such that  
\begin{equation*}
\forall x\in B_r(J) \quad [\Da_{\alpha,\mathcal T} \psi](x) < \psi(x).
\end{equation*}
Then the measure~$\mu$ is $\alpha$--regular. 
%the Frostman dimension of the $\mathcal T$-stationary measure~$\mu$ is bounded from below by~$\alpha$:
%\[
%D_1(\mu)\ge \alpha,
%\]
%or, in other words, .
\end{theorem}
\emr{We now give a simple example to illustrate Theorem~\ref{t:certificate-D1} in action.
\begin{example}
    We may consider an iterated function scheme $\mathcal T$ consisiting of two
    maps $f_1(x) = 0.65x$ and $f_2(x) = 0.6x +1$ with probabilities $p_0 = p_1 = 0.5$. 
    Then for the invariant measure $\mu$ we get $\supp \mu \subset J = [0,2.7]$.
    If we choose a function
    $\psi(x) = 1 - 0.4 |x- 1.25|$ on~$J$ and apply $\mathcal D^{(1)}_{0.35,\mathcal T}$
    we see that $\mathcal D^{(1)}_{0.35,\mathcal T} \psi (x) < \psi(x)$. This
    is illustrated in Figure~\ref{fig:ex2}. Therefore we conclude that 
    the Frostman dimension of the stationary measure of this system 
    is bounded from below by~$0.35$.
    \begin{figure}
        \centerline{\includegraphics{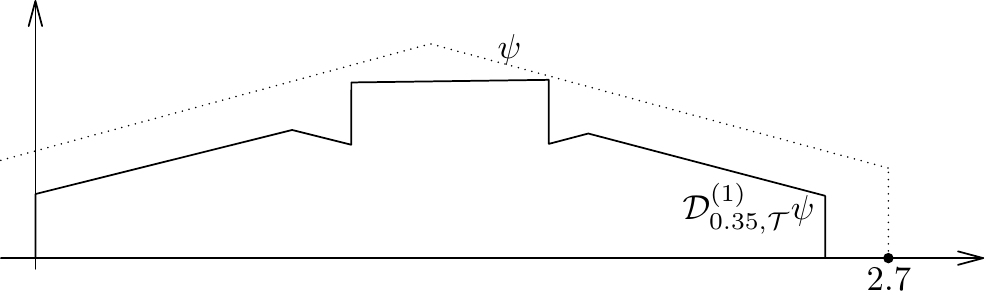}}
        \caption{The image of the function $\psi(x) = 1-0.4|x-1.25|$ is strictly
        smaller than $\psi$.
        }
        \label{fig:ex2}
    \end{figure}
\end{example}}

As in the case of Theorem~\ref{t:finding-D2} for correlation dimension, our next
result states that any lower bound can be found using this method. 
\begin{theorem}\label{t:finding-D1}
Let $\mu$ be the stationary measure of the iterated function scheme $\mathcal
T(\bar f, \bar p, J)$. Then for any $\alpha < D_1(\mu)$ the hypothesis of
Theorem~\ref{t:certificate-D1} holds. In other words there exists a
neighbourhood $B_r(J)$ and a piecewise constant function~$\psi$ with $\supp \psi
\subset B_r(J)$, which is positive and bounded away from 0 and from infinity on
$B_r(J)$, and such that 
\begin{equation*}
\forall x\in B_r(J) \quad [\Da_{\alpha,\mathcal T}\psi](x) < \psi(x).
\end{equation*}
\end{theorem}
This test function can be constructed using the process similar to the one which
is used in the construction of the test function for the
diffusion operator~$\Ds_{\alpha,\mS}$, described in~\S\ref{ss:efbounds-1}.

%\centerline{\Large \sc introduction completed} 

%%% Acknowledgements %%%

\section{Dimension of a measure}
In this section we collect together some preparatory material on different dimensions of a measure and their properties.
A good reference for the background reading is a book by Falconer~\cite{F90}.
See also the work by Mattila at al.~\cite{MMR00} for a discussion and comparison of
notions of dimension of a measure.

It is convenient to summarize some useful notation for the sequel.
%\begin{notation}

For any set $X\subset \mathbb R$ we denote by $\mathcal F_X$ the set of
real-valued positive functions, bounded away from zero and from infinity on $X$ and vanishing on $\mathbb R \setminus X$.
We would like to equip the set of functions $\mathcal F_X$ with the
partial order. We write that $f\prec g$ if  $f(x)<g(x)$ for all $x\in X$ and
$f\preccurlyeq g$ if $f(x)\le g(x)$. Given a finite partition $\mathcal X \eqdef
\{X_j, \, j =1,\ldots N\}$, $X = \cup_{j=1}^N X_j$ we denote by $\mathcal F_{\mathcal X}$ the subset of
piecewise-constant functions associated to the partition $\mathcal X$.

Given a collection of $n$ maps $f_j$, $j=1,\ldots, n$ we use a multi-index
notation for composition of~$k$ of them, namely, we denote 
$f_{\underline j_k}: = f_{j_1} \ldots f_{j_k}$, where $\underline j_k
= j_1,\ldots,j_k \in \{1,\ldots,n\}^k$. 

Finally, we denote by $\mathds{1\!}_X$ the indicator function of $X$. 
%\end{notation}

\subsection{Hausdorff dimension}\label{subsection:dimH}
We briefly recall the definition of the Hausdorff dimension of a set $X \subset \mathbb R$.  
Given $s>0$ and $\delta>0$ we define $s$-dimensional Hausdorff content of~$X$ by
$$
H_\delta^s(X) := \inf 
\left\{
\sum_{i}  (\mathrm{diam}(U_i))^s \mid \{U_i\} \mbox{ is a cover for $X$ and
$\sup_i \{\hbox{\rm diam}(U_i)\} \leq \delta$} 
\right\},
$$
where the supremum is over all countable covers by open sets whose diameter is
at most $\delta$.  We next remove the $\delta$ dependence by defining
$s$-dimensnional Hausdorff measure of $X$ by
$$
H^s(X) := \lim_{\delta \to 0} H_\delta^s(X) \in [0, +\infty].
$$
Finally, we come to the definition of the Hausdorff dimension of the set $X$.
\begin{definition}
\label{def:Hdim}
The {\it Hausdorff dimension} of $X$ is defined by
\begin{equation}
\dim_H(X) := \inf\{ s \geq 0 \mid  H^s(X) =0\}.
\end{equation}
%The \emph{Hausdorff dimension} of a probability measure $\mu$ is defined by
%\begin{equation}\label{eq:dimHmu}
%\dim_H(\mu) := \inf \{\dim_H(X) \mid \mbox{ $X$ is a Borel set with } \mu(X)=1\}.
%\end{equation}
%The {\it Hausdorff dimension} of a probability measure $\mu$ \emb{is defined by}
%\begin{equation}
%\dim_H(\mu) = \inf\{ \dim_H(B) \mid \mu(B)=1\}.
%\end{equation}
\end{definition}
In particular, the Hausdorff dimension of Borel sets (Definition~\ref{def:Hdim}) is used in the definition of the 
Hausdorff dimension of probability measures  (Definition~\ref{def:HDmeasure}).

\subsection{Correlation dimension}\label{subsection:dimC}
A convenient method to obtain a lower bound on the Hausdorff dimension $\dim_H(\mu)$ is 
a standard technique called \emph{the potential principle} (see~\cite[p.~44]{P})
which allows one to relate the
Hausdorff dimension of a measure and convergence of the integral of powers of the distance
function. 
\begin{definition}
    \label{def:emu}
We define \emph{the energy} of a probability measure~$\mu$ by 
\begin{equation}\label{eq:a-int}
    I(\mu,\alpha) := \int_{\mathbb R}\int_{\mathbb R} (d(x,y))^{-\alpha} \, \mu(dx) \, \mu(dy).
\end{equation}
%for $\mu$-measurable compact sets $X$ 
whenever the right-hand is finite. 
\end{definition}

\begin{definition}\label{def:D2}
The {\it correlation dimension} of the measure $\mu$ is defined by
\begin{equation}
    \label{eq:D2-bis}
    D_2(\mu) = \sup\{\alpha \colon I(\mu,\alpha)< + \infty \}.
\end{equation}
%\begin{equation}
%    \label{eq:D2def}
%    D_2(\mu) := \sup\left\{\alpha \hbox{ : } \int_{\mathbb R}\int_{\mathbb R} (d(x,y))^{-\alpha} \,
%\mu(dx) \, \mu(dy) < +\infty\right\}.
%\end{equation}
\end{definition}
%Then the definition of the correlation dimension~\eqref{eq:D2def} takes the form
\noindent This is a special case of the more general $q$-dimensions $D_q(\mu)$
defined analogously~\cite{P93}.

The correlation dimension of~$\mu$ gives a handy lower bound on the Hausdorff
dimension  of~$\mu$. 
\begin{lemma}\label{l:inequality}
$\dim_H(\mu) \geq D_2(\mu)$.
\end{lemma}
\begin{proof}
 The principle  involved is described, for instance, in a book by
 Falconer~\cite[Theorem~4.13]{F90} for sets or in a book by
 Mattila~\cite[\S8]{M92} for measures. 
% note by Shah~\cite[Theorem~2.6]{Shah09}, for measures.  
\end{proof}

The following simple result turns out to be very fruitful. % useful in our approach.
%It is known as the Frostman energy lemma.
\begin{corollary}\label{l:integral}
    If for a Borel probability measure $\mu$ and $0 < \alpha < 1$ the energy
    $I(\mu, \alpha)$ is finite, then $\dim_H(\mu) \geq \alpha$.
%If $\supp \mu$ is Borel-measurable and the value $I(\mu,\alpha)$ is finite, then $\dim_H
%\supp \mu \ge \alpha$. The
%latter implies, in particular, that $\dim_H \mu \ge \alpha$.
%
%
%If for a Borel probability measure~$\mu$ and a number $\alpha \in \mathbb R$ the
%value~$I(\mu,\alpha)$ is finite, then $\dim_H\mu \geq  \alpha$.
\end{corollary}

\begin{remark}
    Developing the method proposed in~\cite{HKPS19} it is possible to
    show~\cite{KV2021} that the strict inequality $\dim_H(\mu_\lambda) > D_2(\mu_\lambda)$ 
    holds, for example, for some Pisot values of parameter~$\lambda$ both in the case
    of Bernoulli convolutions as described in Example~\ref{ex:BCsystem} and in
    the case of $\{0,1,3\}$-system as described in Example~\ref{ex:013}. 
\end{remark}

We now would like to recall that a convolution of a continuous function~$f$ and
a probability measure~$\mu$ is a function given by $(f*\mu) (x)=
\int_{\bbR}f(x-z)
d \mu(z)$. This fact brings us to introducing the last dimension notion we
discuss in this work. 

\subsection{Frostman dimension}
\label{ss:pem}
%The final dimension notion we would like to introduce in this paper is the
%following.
Since this notion is not very well known we will begin by introducing it. 
\begin{definition}
\label{def:D1}
Let us fix a function $f_\alpha(r)=|r|^{-\alpha}$. Let $\mu$ be a compactly
supported probability measure. We define its \emph{Frostman dimension} by
%\begin{equation}\label{eq:D1}
%\begin{array}{ccl}
\begin{align}
D_1(\mu)& =  \sup\{\alpha \hbox{ : } \exists C: \quad \forall x,r \quad
\mu(B_{r}(x))<Cr^{\alpha} \} \tag{\ref{def:D1}.1} \label{eq:D1.1} \\ 
& =  \sup\{\alpha \hbox{ : } \exists C: \quad \forall x \quad  (f_\alpha
*\mu)(x) <C \} \tag{\ref{def:D1}.2} \label{eq:D1.2}  \\ 
& =  \sup\{\alpha \hbox{ : } \text{the convolution } \, f_\alpha *\mu \text{ is a continuous function} \}. 
\tag{\ref{def:D1}.3} \label{eq:D1.3} 
\end{align}
%\end{array}
%\end{equation}
\end{definition}
\begin{remark}
    It is easy to see that three expressions for $D_1$ give the same
    value. Indeed, it follows from the Chebyshev inequality that for any
    $\alpha$ and $C$ such that  $(f_\alpha *\mu)(x) <C$ we have that
    $\mu(B_{r}(x))<Cr^{\alpha}$, so~\eqref{eq:D1.2} implies~\eqref{eq:D1.1}. 
    
    Since $\supp\mu$ is a compact set, the convolution $(f_\alpha * \mu)(x)
    \to 0$ as $x \to \infty$. Therefore if the function $f_\alpha * \mu$ is continuous,
    it is also bounded. Hence~\eqref{eq:D1.3} implies~\eqref{eq:D1.2}. 
    
Finally, let us show that if $\mu(B_{r}(x))<Cr^{\alpha}$ for some $\alpha$ and
$C$, then for any $\alpha^\prime<\alpha$ we have that $f_{\alpha^\prime} * \mu$
is continuous. Indeed, for any $x$ and $\varepsilon>0$ we have an asymptotic
estimate
\begin{equation}
    \label{eq:eps-conv}
\int_{B_{\eps}(x)} |x-y|^{-\alpha'} \,
d\mu(y) \le \int_0^{\eps} r^{-\alpha'} d(Cr^{\alpha}) = \frac{\alpha
C}{\alpha-\alpha'}\cdot \eps^{\alpha-\alpha'} \to 0 \quad \mbox{ as }\eps\to 0. 
\end{equation}
On the other hand, the convolution of $\mu$ with the function
$\bar f_{\alpha^\prime}(r)=\max(|r|,\eps)^{-\alpha^\prime}$ is a convolution of a probability measure with a
continuous bounded function and hence is continuous. It follows
from~\eqref{eq:eps-conv} that these convolutions converge uniformly to
$f_{\alpha}*\mu$, and therefore the latter is everywhere finite and 
continuous as a uniform limit of continuous functions. Thus~\eqref{eq:D1.1}
implies~\eqref{eq:D1.3}.
\end{remark}

It is also not difficult to see that the Frostman dimension is not
larger than the correlation dimension.

\begin{lemma}
    \label{lem:d2d1}
For any compactly supported probability measure $\mu$,
\begin{equation}\label{eq:D1D2}
D_1(\mu)\le D_2(\mu).
\end{equation}
\end{lemma}
\begin{proof}
    Let us consider the function $f_\alpha(r) = |r|^{-\alpha}$. Then for any
    $\alpha$ such that the convolution $f*\mu$ is bounded, one has that
\[
I(\mu,\alpha)=\int_\bbR (f_\alpha*\mu)(x) d\mu(x) < + \infty.
\]
Therefore % the supremum of the left and right hand sides of the inclusion
\[
\{\alpha \mid \exists C: \quad \forall x \quad  (f_\alpha *\mu)(x) <C \} \subset 
\{\alpha \mid I(\mu,\alpha)<+\infty \},
\]
and the desired inequality~\eqref{eq:D1D2} follows.
\end{proof}

\section[Computing lower bounds]{Computing uniform lower bounds on dimension}
\label{s:efbounds}
We begin by modifying the diffusion operator and Theorem~\ref{t:certificate-D2}
in preparation for  computer-assisted proofs of Theorems~\ref{t:main},~\ref{thm:salemnum}
and~\ref{thm:dim013}. 

\subsection{Extension to open set of parameters}
\label{ss:efbounds-2}

  We keep the notation of Section~\ref{subsection:ifs}. 
  Let $\mS(\lambda, \bar c, \bar p)$ be an iterated function scheme of~$n$ similarities with the same scaling
  coefficient~$\lambda$ and probability vector~$\bar p$.

  Given $0<\alpha<1$ and a subset $\Lambda \subset [0,1]$ let $J$ be a
  $\Lambda$-admissible interval and let $\mathcal J = \{J_1,\ldots,J_N\}$ be a 
  partition of $J$. The modified  diffusion operator we introduce below
  preserves the subspace $\mathcal F_{\mathcal J}$ of piecewise constant functions associated to
  the partition~$\mathcal J$.
  
  %We next want to introduce an auxiliary diffusion operator defined as follows.
  
  \begin{definition}
  We define a finite rank nonlinear diffusion operator 
  $ \mathcal D_{\alpha,\Lambda;\mathcal J}: \mathcal F_{\mathcal J} \to  \mathcal
  F_{\mathcal J}$ by 
  \begin{equation}
  \label{eq:D-Lambda}
   \mathcal D_{\alpha,\Lambda, \mathcal J} \psi |_{J_k} = (\inf \Lambda)^{-\alpha}
   \sum_{i,j=1}^n p_i p_j  \sup_{x\in J_k, \, \lambda\in\Lambda} \psi \left( \frac{x-
   c_i + c_j }\lambda \right),
   \quad 1 \le k \le N.
  \end{equation}
  \end{definition}
  We see directly from definition that for any $\lambda \in \Lambda$ we have
  that for any $\psi \in \mathcal F_J$, 
  $$
  \Ds_{\alpha,\mS} \psi \preccurlyeq \mathcal
  D_{\alpha,\Lambda,\mathcal J} \psi.
  $$
  The following adaptation of Theorem~\ref{t:certificate-D2} to the
  operator~$\mathcal D_{\alpha,\Lambda,\mathcal J}$ follows immediately. 
  %\paragraph{Theorem $1.17^*$.}% 
  \begin{theorem}
      \label{thm:5star}
   %\textit{\kern-6pt 
   Assume that for some $\alpha>0$ and a set $\Lambda\subset[0,1]$ 
   there exists an admissible interval~$J_\Lambda$, its 
   partition~$\mathcal J$, and a piecewise-constant function $\psi \in \mathcal F_{\mathcal J}$ which is
   positive and bounded away from~$0$ and from infinity on~$J_\Lambda$, such that  
    \begin{equation}
        \label{eq:D2-hat-test}
        %\forall x\in J_\Lambda         \quad 
        \mathcal D_{\alpha,\Lambda,\mathcal J}\psi \preccurlyeq \psi.
    \end{equation}
    Then for any $\lambda \in \Lambda$ the correlation dimension of the
    $\mS_\lambda$-stationary measure $\mu_\lambda$ is bounded from below by~$\alpha$:
    \[
    D_2(\mu_\lambda)\ge \alpha.
    \]
\end{theorem}
\emr{  We can illustrate the principle with the following example. 
  \begin{example}
      In the setting of Bernoulli convolution with $\lambda = 0.75$
      we may choose an interval $\Lambda = [\lambda - 10^{-8}, \lambda+10^{-8}]$ and  $J =
      [-4.1, 4.1]$. Applying the operator $\mathcal D_{0.2, \Lambda, \mathcal J}$ to a function  
      \begin{align*}
      \psi(x) = 0.15 \cdot
      \mathds{1}_{[-4,4]}+0.1 \cdot \mathds{1}_{[-2.7,2.7]} &+
      0.15 \cdot \mathds{1}_{[-2.1,2.1]} + 0.1\cdot \mathds{1}_{[-1.6,1.6]} + 
       0.125\cdot \mathds{1}_{[-1.5,1.5]} \\ &+ 0.1 \cdot \mathds{1}_{[-1.2,1.2]} +
      0.125 \cdot \mathds{1}_{[-0.5,0.5]} + 0.15\cdot \mathds{1}_{[-0.25,0.25]},
      \end{align*}
      depicted in Figure~\ref{fig:t5star}, we get that $\mathcal D_{0.2, \Lambda, 
      \mathcal J} \psi \prec \psi$ and conclude that $\dim_H \mu_\lambda > 0.2$
      for all $\lambda \in \Lambda$.
      \begin{figure}
          \centerline{\includegraphics{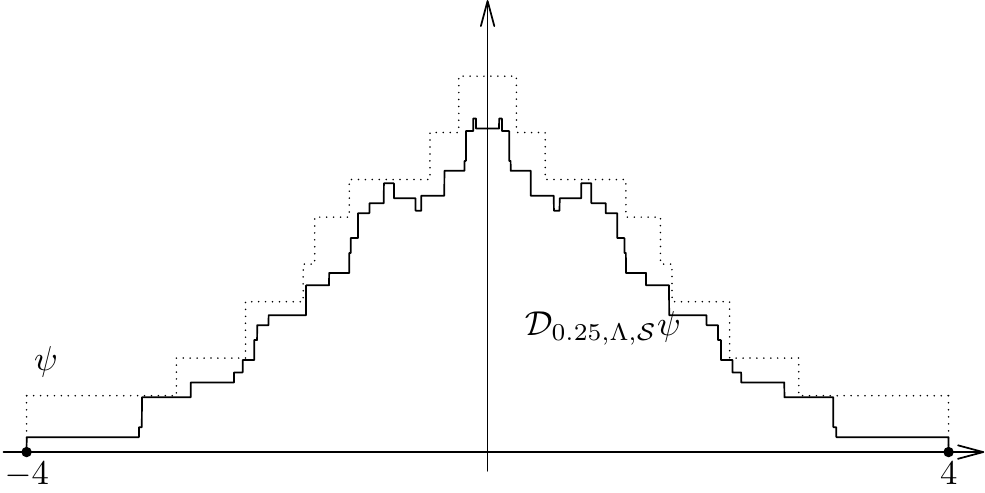}}
          \caption{The image of the piecewise constant function~$\psi$ is
          strictly smaller than~$\psi$.}
          \label{fig:t5star}
      \end{figure}
%0.125*indicator(-0.5,0.5,x)+0.15*indicator(-0.25,0.25,x)) ;
  \end{example}
}

  Therefore in order to show that $D_2( \mu_\lambda) \ge \alpha$ for all
  $\lambda \in \Lambda$, it is sufficient to find an $\Lambda$-admissible
  interval $J_\Lambda$, its partition $\mathcal J$, and a piecewise constant
  function~$\psi$ associated to $\mathcal J$, with $\psi|_{J_\Lambda} >
  0$ such that $\mathcal D_{\alpha,\Lambda,\mathcal J} \psi \preccurlyeq \psi$ 
  and then apply Theorem~\ref{thm:5star}. %:certificate-D2}$^*$.
  \begin{remark}
  Furthermore, by refining the partition in the construction of the operator
 $\mathcal D_{\alpha,\Lambda,\mathcal J}$ and choosing smaller intervals
 $\Lambda$, in the limit we obtain the correlation dimension.
 In particular, this implies a well-known fact that the correlation dimension is 
 lower semicontinuous. %ity of the correlation dimension. 
 \end{remark}

\subsection{Constructing the test function}
\label{ss:efbounds-1}

The construction of a suitable test function~$\psi$ which satisfies the hypothesis of
Theorem~\ref{t:certificate-D2} is based on the following general result for
linear operators. 
\begin{notation}
    Given a linear operator $A$ acting on real-valued functions and a small
    number $\vartheta > 0$ we introduce 
    \begin{equation}
        \label{eq:hatop}
     [\hatop{A} f](x) :=\min([Af](x)+\vartheta,f(x)).  
 \end{equation}
    Observe that if $A$ preserves the subset of positive functions, then $\hatop
    A $ also does so. Furthermore, if $A$ preserves the subspace of continuous
    functions, then $\hatop A$ preseves this subspace too. 
\end{notation}

We say that an operator $A \colon \mathcal F_J \to \mathcal F_J$ is monotone, if
for any $f, g \in \mathcal F_J$ such that $g \preccurlyeq f$ we have that $A g
\preccurlyeq Af$. Note that we don't require the operator~$A$ to be linear in
the definition of monotone. 

We will need the following easy general statement.
\begin{lemma}
    \label{l:min}
Let~$A\colon \mathcal F_X\to\mathcal F_X$ be a monotone operator, and let
$\vartheta>0$ be a real number. Assume that for some function $f\in \mathcal
F_X$ and for some~$x_0 \in X$ we have that $[\hatop A f](x_0) = [Af](x_0)+\vartheta$. 
Then for any~$k \ge 1 $ %$g=$: 
\[
[A\hatop{A}^k f](x_0)+\vartheta\le [\hatop{A}^kf](x_0).
\]
\end{lemma}
\begin{proof}
It sufficient to show that the statement holds for $k=1$. Then the result
follows by induction. By definition of $\hatop A$ we have that $\hatop{A}f \preccurlyeq f$.
Together with monotonicity of $A$ it implies that
\[
[A \hatop A f](x_0)+\vartheta \le [Af](x_0)+\vartheta = [\hatop{A} f ](x_0).
\]
For the inductive step, let us assume that 
\[
[A \hatop{A}^k f](x_0)+\vartheta \le [\hatop{A}^k f ](x_0).
\]
Then $[\hatop{A}^{k+1}f](x_0) = [A\hatop{A}^kf](x_0)+\vartheta$ and therefore using monotonicity
and the fact that $\hatop A f \preccurlyeq f$ we get 
$$
[A \hatop{A}^{k+1} f](x_0) + \vartheta \le [A \hatop{A}^k f](x_0) + \vartheta =
[\hatop{A}^{k+1} f]  (x_0).
$$
\end{proof}

\begin{proposition}
    \label{prop:auxop}
Let $J \subset \mathbb R$ be a closed interval. Let~$A\colon \mathcal F_J\to
\mathcal F_J$ be a monotone operator.  Let $\vartheta>0$ be an arbitrarily small
real number. 
Assume that for some $n>0$ and $f \in \mathcal F_J$ we have that ${\hatop A}^n f \prec f$. Then 
$$
A {\hatop A}^{n} f \preccurlyeq {\hatop A}^{n} f.
$$  
\end{proposition}
\begin{proof}
    By definition of $\hatop A$ for any function $f \in \mathcal
    F_J$ we have $\hatop{A}f\preccurlyeq f$. Since $\hatop{A}$ is monotone, we
    deduce that 
     \begin{equation*}
    {\hatop A}^n f \preccurlyeq {\hatop A}^{n-1}f \preccurlyeq \cdots \preccurlyeq f.
    \end{equation*}
    Since for every $x\in J$ we have $[\hatop{A}^n f](x) < f(x)$, then there
    exists $0\le m(x) \le n-1$ such that a strict inequality $[{\hatop A}^{m(x)+1}
    f](x)<[{\hatop A}^{m(x)} f](x)$ holds. Therefore   
    \[
    [A{\hatop A}^{m(x)} f](x)+\vartheta\le  [{\hatop A}^{m(x)} f](x).
    \]
    Applying Lemma~\ref{l:min} to the function ${\hatop A}^{m(x)} f$ with
    $k=n-m(x)$, we get 
    \[
    [A{\hatop A}^n f](x)+\vartheta\le  [{\hatop A}^n f](x),
    \]
    and the result follows. 
    %As $x\in J$ was arbitrary, we have obtained the desired $A {\hatop A}^{n} f \preccurlyeq {\hatop A}^{n} f$.
\end{proof}

\medskip 
Our numerical results are based on the following Corollaries, which follow
immediately from Theorem~\ref{t:certificate-D2}
and~Theorem~\ref{thm:5star}. %:certificate-D2}$^*$.
%and Proposition~\ref{prop:auxop}.
\begin{corollary}
  \label{cor:single}
  %If there exists $n$ and $\vartheta$ such that 
  If there exists and admissible interval $J_\lambda$ such that Proposition~\ref{prop:auxop}
  holds for $A := \Ds_{\alpha,\mS(\lambda)}$ and $ f:= \mathds{1}_{J_\lambda}$, then
   $D_2(\mu_\lambda) \geq \alpha$.
\end{corollary}

We use the last proposition in order to find a suitable test function $\psi$ for 
Theorem~\ref{thm:5star}.   
\begin{corollary}
  \label{cor:uniform}
  If there exist an admissible interval~$J_\Lambda$ %$n$, $\vartheta>0$,
  and its partition $\mathcal J$ such that Proposition~\ref{prop:auxop} holds
  for $A := {\mathcal D}_{\alpha,\Lambda,\mathcal J}$ and $f := \mathds{1}_{J_\Lambda} $
  then 
  %the function $\hatop{A} \mathds{1}_{J_\Lambda}$
  %$ {\widehat  {\mathcal  D}}_{\alpha,\Lambda,\mathcal J}^{n} \mathds{1}_{J_\Lambda}$ 
  %satisfies the
  %hypothesis of Theorem~\ref{t:certificate-D2}$^*$, and, in particular 
  $D_2(\mu_\lambda) \geq \alpha$ for all   $\lambda \in \Lambda$.
\end{corollary}

    Note that Corollary~\ref{cor:single} can only be applied to rational
    parameter values $\lambda \in \mathbb Q$, which can be represented in
    computer memory exactly. In order to study irrational parameter values, such
    as Pisot or Salem numbers, we need to apply Corollary~\ref{cor:uniform} to
    a tiny interval $\Lambda$ with rational endpoints containig the irrational
    parameter value we would like to study.

%    In particular, one of the limitations of the method we present in this note is that it
%    cannot be used to study irrational parameter values individually. 

\subsection[Practical implementation]{Practical implementation: computing lower bounds for $D_2(\mu)$}

The following method, based on Corollaries~\ref{cor:single}
and~\ref{cor:uniform}, can be used to obtain a lower bound on the correlation
dimension of a stationary measure of an iterated function scheme of
similarities. 

\subsubsection{Verifying a conjectured value}
\label{sss:verifyalpha}
First let us assume that we would like \emph{to check} whether $\alpha$ is a lower
bound for the correlation dimension of
an iterated function scheme of similarities $\mS(\lambda,\bar c, \bar p)$ for
an open set of parameter values $\lambda \in \Lambda$. Then we proceed as
follows:
\begin{enumerate}
  \item Fix an admissible interval $J_\Lambda$, associated to $\Lambda$.
  \item Choose a partition $\mathcal J$ of the interval $J$ consisting of $N$ intervals of the
      same length. From the point of view of efficiency of the practical
      implementation, it is better to choose the length of the intervals of the
      partition to be comparable with $|\Lambda|$. In our computations, we often take $N$
      so that  
      $$
        \frac12 N |\Lambda| \le |J| \le 2 |\Lambda| N.
      $$
  \item Introduce an operator $A:=\mathcal D_{\alpha,\Lambda,\mathcal J}$. 
  \item Take a piecewise-constant function $\mathds{1}_J$ and $\vartheta>0$ and compute
      the images ${\hatop A}^n \mathds{1}_J$, which are
      piecewise-constant functions associated to the partition $\mathcal J$.  
  \item If we find $n_0$ so that ${\hatop A}^{n_0}
      \mathds{1}_J\prec \mathds{1}_J$, then we conclude that $D_2(\mu_\lambda) > \alpha$.
\end{enumerate}

\emr{We can give a simple example to illustrate the method.}
\begin{example}
\emr{    Let us show that for the Bernoulli convolution measure with $\lambda \in
    (0.74,0.76)$ we have $D_2(\mu_\lambda) > 0.75$. 
    The corresponding iterated function scheme consists of two maps and
    probability vector $\bar p$ given by 
    $$
    f_0(x)= \lambda x, \quad f_1(x) = \lambda x+1; \qquad \bar p = \Bigl(
    \frac12,\frac12\Bigr). 
    $$
    Then $\supp \mu_{\lambda} \subseteq [0,4.17]$ and we can choose an admissible
    interval~$J = [-4.5,4.5]$. We can also consider a uniform partition of $J$
    consisting of $120$ intervals.    }

\emr{    We then choose $\theta=10^{-3}$, and compute the images of $\mathds{1}_J$
    under $\hatop A$ for $A = \mathcal  D_{0.75,\Lambda,\mathcal J}$. 
    It turns out that~$25$ iterations is sufficient, in particular we have that $\max 
    {\hatop A}^{25} \mathds{1}_J < 0.99$.  
    This is illustrated in Figure~\ref{fig:simple}.
    }

  \begin{figure}[h]
  \centerline{
  \includegraphics{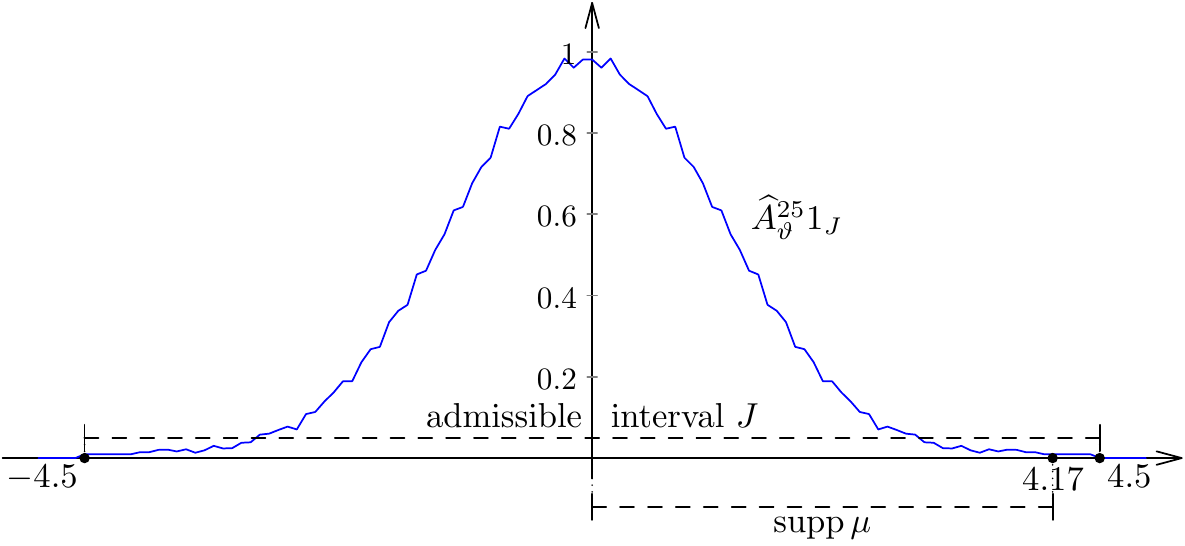}
  }
  \caption{The function ${\hatop A}^{25} \mathds{1}_J\prec\mathds{1}_J$ 
  for $A = \mathcal D_{0.75,\Lambda,\mathcal J}$ and $\vartheta = 10^{-3}$, 
  $\supp \mu_{0.75}$, and an
  admissible interval~$J$.}
  \label{fig:simple}
  \end{figure}

\end{example}

%\newpage

\emr{When it comes to the realisation of an iterative method in practice one of
the common concerns is accumulation of the rounding error. The following remark 
        explains why in the present case this is not a significant issue.}

\begin{remark}
\emr{Applying the operator $\widehat A_\vartheta$ 
        to a given piecewise-constant function $\psi$ changes its value on one of the intervals 
        only if the \emph{computed} value $A_{\mathrm{comp}}[\psi]$ is below the actual value of $\psi$ 
        on this interval by at least $\vartheta$, otherwise the value stays completely 
        unchanged.} 
        
\emr{        Assume that the rounding errors never exceed $\frac{1}{2}\vartheta$; this is 
        quite reasonably the case, for instance, since we typically chose $\vartheta = 10^{-8}$, 
        while making the calculations with quadruple precision, in other
        words, the numbers involved have $32$ significant digits. }
        
\emr{        Then, the computed image $\widehat A_{\vartheta,\mathrm{comp}}[\psi]$ is lower-bounded by 
        a true image with half the ``added value'':
        \[
        	\widehat A_{\vartheta/2}[\psi] \preccurlyeq \widehat A_{\vartheta,\mathrm{comp}}[\psi].
        \]
        By induction, it is then easy to see that after an \emph{arbitrary} number $n$ of iterations one has 
        \[
        	\widehat A_{\vartheta/2}^n[\psi] \preccurlyeq \widehat A_{\vartheta, \mathrm{comp}}^{n}[\psi].
        \]
        Thus, if after some number $n$ of iterations the \emph{computations}
        provide $A_{\vartheta, \mathrm{comp}}^{n}[\mathds{1}_J] \prec
        \mathds{1}_J$, where $A=\Ds_{\alpha,\mS(\lambda)}$, one actually gets
        the desired $\widehat A_{\vartheta/2}^n[\mathds{1}_J] \prec
        \mathds{1}_J$ and hence the applicability of Corollary~\ref{cor:single}.
        In our computations, to avoid problems with the strict inequality
        handling, we have been asking for the inequality 
        \[
        (\widehat{\Ds_{\alpha,\Lambda,\mathcal{J}}})_{\vartheta,
        \mathrm{comp}}^{n}[\mathds{1}_J] \prec 0.995\cdot \mathds{1}_J.
        \]
}        
\end{remark}

\emr{Next, we should describe what we would do if an attempt to obtain a lower bound by a given $\alpha$ was unsuccessful.}

\begin{remark}
    \label{rem:droplam}
It is possible that after a large number of iterations $n$, there exists an
interval $\mathcal J_k$ of the partition $\mathcal J$ where we have the
equality: 
$$
[ {\hatop A }^{n}\mathds{1}_J] |_{\mathcal J_k} = 1. 
$$
Then we cannot reach a definitive conclusion, as there are three possibilities:
\begin{enumerate}
    \item the number of iterations $n$ is not large enough,
    \item the intervals of the partition $\mathcal J$ are too long, or the
        interval $\Lambda$ is too long,
    \item the number $\alpha$ is not a lower bound, i.e.
        there exists $\lambda \in \Lambda$ such that $D_2(\mu_\lambda) \le \alpha$. 
\end{enumerate}
In this case we could try to increase the number of iterations, to choose a
finer partition, or to drop the conjectured value $\alpha$ to $\alpha^\prime <
\alpha$ and to consider the operator $A :=\mathcal
D_{\alpha^\prime,\Lambda,\mathcal J}$. It follows from
Proposition~\ref{prop:D21} that provided our
guess on the lower bound was correct, we will be able to justify it using this
approach, subject to computer resources. 
\end{remark}

Our method can be used not only \emph{to verify a suggested lower bound}, but
also to find a lower bound for the correlation dimension or to improve an
existing lower bound. 
%, by an iterative process.

\subsubsection{Computing a lower bound} 
\label{sss:compute}
Assume that we would like to improve an existing lower bound $d_1 <
D_2(\mu_\lambda)$ 
%for the correlation dimension of the stationary measure for an iterated function scheme
%$\mS(\lambda,\bar c,\bar p)$ for $\lambda \in \Lambda$ 
using no more than $N$ iterations 
of the operator and piecewise constant functions with no more than $K$
intervals. Additionally assume that there is an upper bound
$D_2(\mu_\lambda)<d_2$.

Then we can fix $\varepsilon>0$, a uniform partition $\mathcal J$, an operator 
$A :=\mathcal D_{\alpha,\Lambda,\mathcal J}$ and to search
for an $\alpha$ satisfying the following conditions 
\begin{enumerate}
    \item\label{it:left} there exists $k<N$: ${\hatop A}^{k} \mathds{1}_J
        \prec \mathds{1}_J$
    \item\label{it:right} for all $k  \le N$: ${\hatop A}^{k} \mathds{1}_J
      \not\prec \mathds{1}_J$
\end{enumerate}
One approach to find~$\alpha$ would be to apply the well known bisection method to
the interval $(d_1,d_2)$. However, to obtain good estimates, one has to
allow for a large number of iterations before dropping the conjectured lower
bound and this is very time-consuming. In other words, negative answer is
expensive as we have to examine all possibilities described in
Remark~\ref{rem:droplam}.

It is therefore more efficient to use a partition of the interval $(d_1,d_2)$ into
$M\!: =\![\sqrt{(d_2\!-\!d_1)\varepsilon^{-1}}]$ intervals of equal length
and to test the values 
$\alpha_k = d_1+k\cdot \frac{1}M  $, $k = 1,\ldots,M$ using the method explained in the
previous subsection~\ref{sss:verifyalpha}. We then want to find a $0 \le k \le M$ such that 
for $A_k :=\mathcal D_{\alpha_k,\Lambda,\mathcal  J}$ there exists $n < N$ with
the property
%interval $\left(d_1 + \frac{k}M,d_1 + \frac{k+1}M\right)$ such that 
$$
{\widehat A_{k,\vartheta}}^{n} \mathds{1}_J
\prec \mathds{1}_J \qquad \mbox{ and } \qquad 
{\widehat A_{k+1,\vartheta}}^{N} \mathds{1}_J
\not\prec \mathds{1}_J. 
$$
Then we repeat the procedure again, dividing the interval  $\left(\alpha_k,
\alpha_{k+1}\right)$ into $M$ intervals of length $\varepsilon$. This way would need to apply all $N$
iterations only twice (to confirm the second condition~\ref{it:right}) to find the desired value~$\alpha$.    

Finally, we note that in order to compute a good lower bound on a
large interval of parameter values, we consider \emph{a cover} of this interval
by a large number of small overlapping intervals and compute a lower bound on
each of them. 

\begin{remark}
    We would like to emphasize that the value $\alpha+\varepsilon$ is not an
    upper bound for correlation dimension, as it depends on the number~$K$ of
    iterations allowed and the number of intervals for the space of piecewise
    constant functions, and might increase (together with $\alpha$) when we increase those values.  
\end{remark}
\begin{definition}
  We call~$\varepsilon$ the \emph{refinement} parameter. 
\end{definition}

In \emr{Subsection~\ref{sss:mainproof}} we give details of the application of our method to
computing lower bound for correlation dimension of Bernoulli convolutions.

%\subsubsection{Alternate approach: one-step division after a lot of iterations}

\emr{We conclude this subsection by presenting the following alternate approach
to the use of the diffusion operator, that was suggested to us by an anonymous
referee. This suggestion is particularly helpful, and we are glad to be able to present it
here:} 
\begin{remark}
\emr{For a given sufficiently small interval $\Lambda$ of values of $\lambda$ one can: 
\begin{itemize}
\item Take the initial function $\psi_0=\mathds{1}_J$, sufficiently small
    $\vartheta$, initial lower bound $\alpha_0$ and a threshold $t\in (0,1)$; 
\item Apply the operator ${\hatop A}$, where $A=D_{\alpha_0,\Lambda,\mathcal
    J}$, until the maximum descends below the chosen threshold. Let $k$ be the
    smallest number such that  
    \begin{equation}
        \label{eq:psik}
\psi:={\hatop A}^{k}\mathds{1}_J \prec t \cdot \mathds{1}_J.
\end{equation}
In particular, as $t<1$, this implies that ${\hatop A}[\psi]\preccurlyeq\psi$;
\item Then, one gets a lower bound $\alpha$ for the correlation dimension,
    choosing its value to be the maximal for which $D_{\alpha,\Lambda,\mathcal  J}[\psi]\preccurlyeq \psi$, by setting 
\begin{equation}\label{eq:new-alpha}
\alpha:=\alpha_0+ \log_{\lambda_{\min}} \max_J \frac{D_{\alpha_0,\Lambda,\mathcal  J}[\psi]}{\psi}= 
\log_{\lambda_{\min}}  \max_J \frac{D_{0,\Lambda,\mathcal  J}[\psi]}{\psi},
\end{equation}
where $\lambda_{\min}=\min \Lambda$. Observe that since one can actually use
\emph{any} function $\psi$ in order to look for the lower bound $\alpha$
(applying then Theorem~\ref{thm:5star}), rounding errors during the iterations
are not much of an issue. Indeed, there is only \emph{one} iteration and one
division applied in~\eqref{eq:new-alpha}, with the denominator bounded from below by $\vartheta$ due the
construction of~$\psi$, and compared to the precision of calculations
$\vartheta$ is not a small number at all. 
\end{itemize}
}
\end{remark}
\emr{This method really works quite well. For instance, taking
$\vartheta=10^{-7}$ and $\alpha_0=0.82$ (lower bound by Hare and Sidorov),
taking $\Lambda$ to be $0.5\cdot 10^{-5}$-neighborhood of some $\lambda$ and
separating the interval $J=[-r,r]$ into $4\cdot 10^4$ intervals, where
$r=\frac{1.1}{1-\lambda_{\max}}$, $\lambda_{\max}=\max \Lambda$, and choosing
the threshold $t=\frac{1}{20}$ (that is quite small so that $k$
in~\eqref{eq:psik} is quite large), one gets the estimates 
\begin{itemize}
\item $\alpha = 0.9923757365$ for the Fibonacci value $\lambda = 0.6180339887$
    (compare with the lower bound $0.992395833333$ in
    Table~\ref{tab:multinacci});
\item $\alpha = 0.9642020738$ for the tribonacci value $\lambda = 0.54368901$
    (compare with the lower bound $0.964214555664$ in
    Table~\ref{tab:multinacci});
\item $\alpha = 0.999641567$ for one of the Salem numbers $\lambda = 0.71363917$
    (compare with the lower bound $0.999687500$ in the table in Section~\ref{ap:salem}).
\end{itemize}
}

\subsubsection{Proof of Theorem~\ref{t:main}}
\label{sss:mainproof}

Recall that %the for algebraic parameter values $\lambda$ we have that $
$\dim_H \mu_\lambda \ge \dim_H \mu_{\lambda^2}$ and therefore it is sufficient to
compute a lower bound for $\lambda \in [0.5,0.8]$.   

To obtain a uniform lower bound on the correlation dimension $D_2(\mu_\lambda)$,
for Bernoulli convolution measures $\mu_\lambda$ on the entire interval of
parameter values
%\footnote{ 
%In the case of Bernoulli convolutions, we know from the work of Hare and
%Sidorov, that $\inf_{\frac12\le \lambda \le 1} \dim_H\mu_\lambda  = \inf_{\frac12
%\le \lambda \le \frac1{\sqrt2}} \dim_H \mu_\lambda$ and, moreover $\dim_H \mu_\lambda > 0.82$.}
$[0.5,0.8]$ we proceed in two steps. First, we consider a
cover of the interval $[0.5,0.8]$ by $100$ overlapping intervals of the
same size. We then apply
the method explained in~\S\ref{sss:compute} with $N = 7 \cdot 10^6$ partition
intervals for the test function, and set the maximum for the number of iterations
of the diffusion operator to $K=150$. We choose a lower bound $d_1 =
0.82$, an upper bound $d_2 = 1$ and set the refinement parameter to $\varepsilon = 0.01$. 
The computation takes about 10 minutes for each interval and can be done in
parallel; the result is presented in Table~\ref{tab:lowerbound}.  

\begin{table}[h] %\label{table}
\begin{center}
\begin{tabular}{|ccc||ccc|}
  \hline
  $\Lambda$ &  & $\alpha$ & $\Lambda$ &  & $\alpha$ \\ 
  \hline
$[0.500  ,0.515]$&& $0.96612$&     $[0.566  ,0.569]$&& $0.96612$\\
$[0.515  ,0.518]$&& $0.95402$&     $[0.569  ,0.614]$&& $0.97822$\\
$[0.518  ,0.542]$&& $0.96370$&     $[0.614  ,0.617]$&& $0.96612$\\
$[0.542  ,0.545]$&& $0.95402$&     $[0.617  ,0.743]$&& $0.97580$\\
$[0.545  ,0.554]$&& $0.96612$&     $[0.743  ,0.800]$&& $0.96612$\\
$[0.554  ,0.566]$&& $0.97580$&                      && \\ 
\hline
\end{tabular}
\end{center}
  \caption{Uniform lower bounds for the correlation dimension of Bernoulli
  convolution measures, after the first step.}
  \label{tab:lowerbound}
\end{table}
Afterwards, we use the bounds we computed as an initial guess for the corresponding
parameters $\lambda$ and improve them by applying the same method again. This
time, based on the first estimates, we take uniform covers of $[0.499,0.575]$
and $[0.572,0.8]$ by $5000$ intervals each. We then use   
$N = 10^7$ intervals for the space of piecewise-constant functions; 
set the maximum $K = 1000$ for the number of iterations for the diffusion operator; and choose $\varepsilon =
10^{-4}$ as the refinement parameter. This second computation takes about two
weeks with 32 threads running in parallel.

The result is presented in Figure~\ref{fig:plotBC}. %~\ref{fig:plotBC-p2}.
In support of the conjecture that dimension drops occur at Pisot parameter
values, we identify minimal polynomials of algebraic numbers which seem to
correspond to the bigger drops and verified that they are Pisot values, i.e. all their
Galois conjugates lie inside the unit circle. 
\begin{remark} 
It follows from the \emph{overlaps conjecture} by Simon~\cite{Si96} which was proved by
Hochman~\cite{hochman} for algebraic parameter values, that the
dimension drop occurs only for the roots of polynomials with
coefficients~$\{-1,0,1\}$. We see that some of the polynomials indicated in the
plot have~$\pm2$ among their coefficients. 
This doesn't contradict the result of Hochman, because the polynomials we give
are the minimal polynomials. Each of the polynomials with coefficients $\pm2$ 
becomes a polynomial with coefficients $\{-1, 0, 1\}$ after multiplying by 
an appropriate factor. For instance, $x^5-x^3-2x^2-2x-1$ after multplying by $(x-1)$ becomes
$x^6-x^5-x^4-x^3+x+1$.

%This doesn't contradict the result of
%Hochman, because the polynomials we give are the minimal polynomials. Each of
%these polynomials becomes a polynomial with coefficients $\{-1,0,1\}$ only after
%multiplying by one of the polynomials~$x-1$, $x+1$, $x^2+1$, $x^2+x+1$, or $x^3+x^2+x+1$.
%It is
%possible that the corresponding algebraic value is a root of a polynomial with
%larger degree and coefficients $\{-1,0,1\}$ only. 
\end{remark}
\begin{figure}
   \begin{subfigure}{150mm}
   \includegraphics{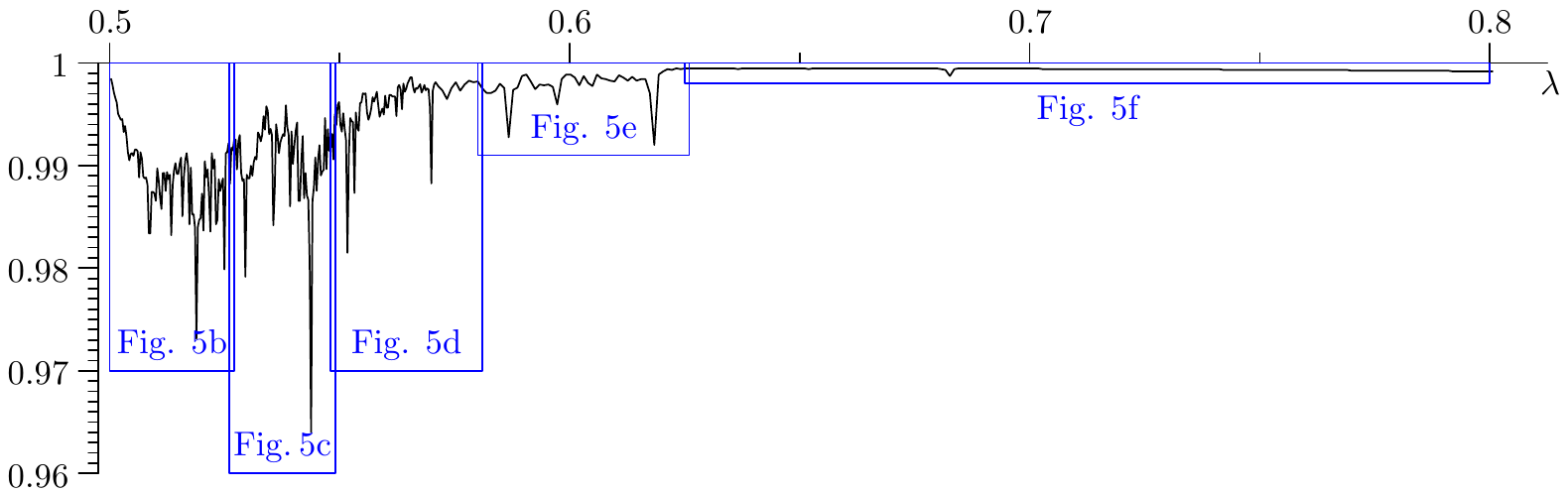}
   \caption{Plot of lower bounds for correlation dimension of Bernoulli
   convolution measures for $0.5\le \lambda \le 0.8$ with more detailed plots
   in~\ref{fig:bc1},~\ref{fig:bc2}, \ref{fig:bc3},~\ref{fig:bc4},~\ref{fig:bc5}
   below.} 
   \end{subfigure}

   \begin{subfigure}{150mm}
   \includegraphics{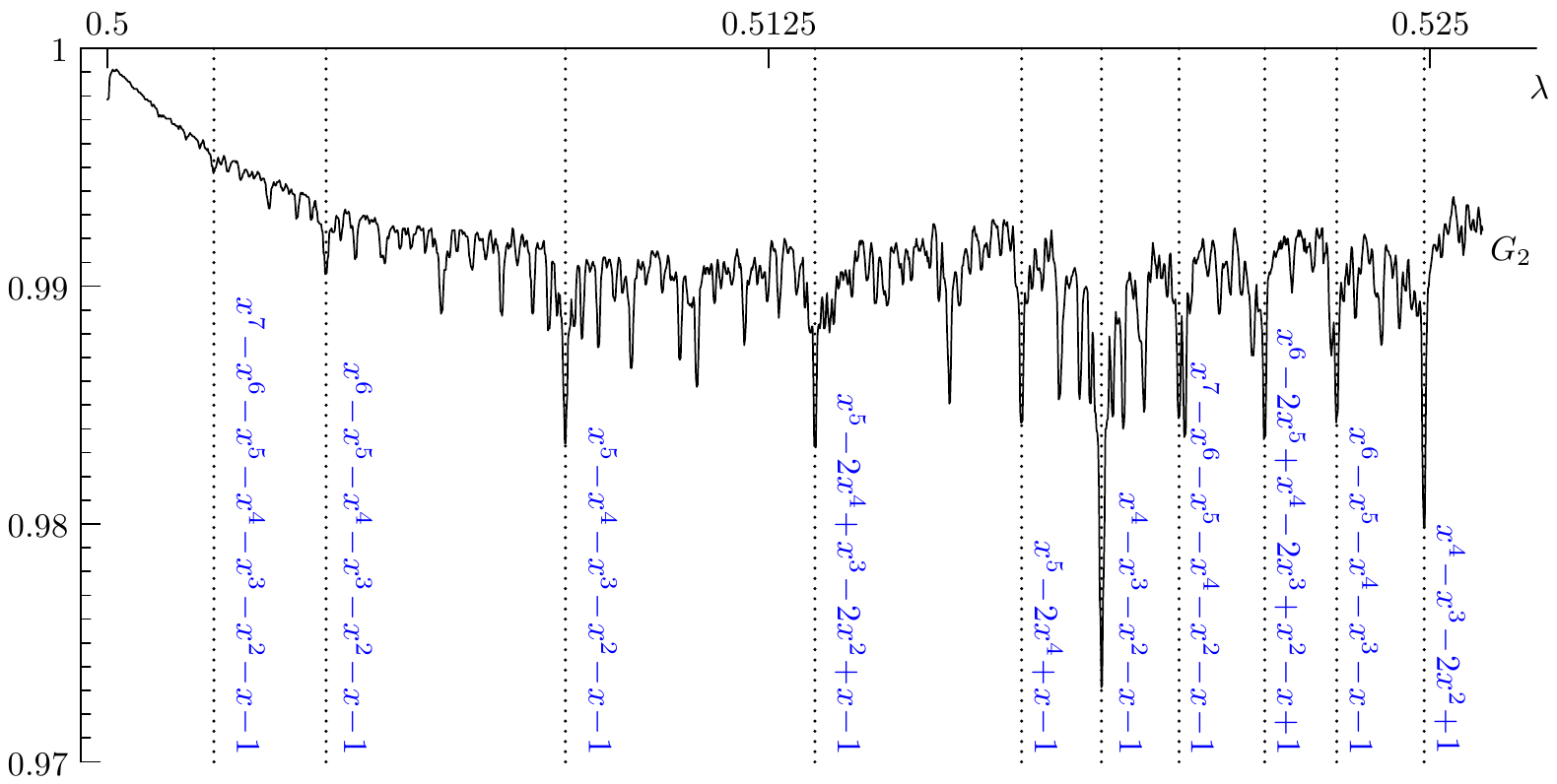}
   \caption{$0.499 < \lambda < 0.526$ }
   \label{fig:bc1}
   \end{subfigure}

   \begin{subfigure}{150mm}
   \includegraphics{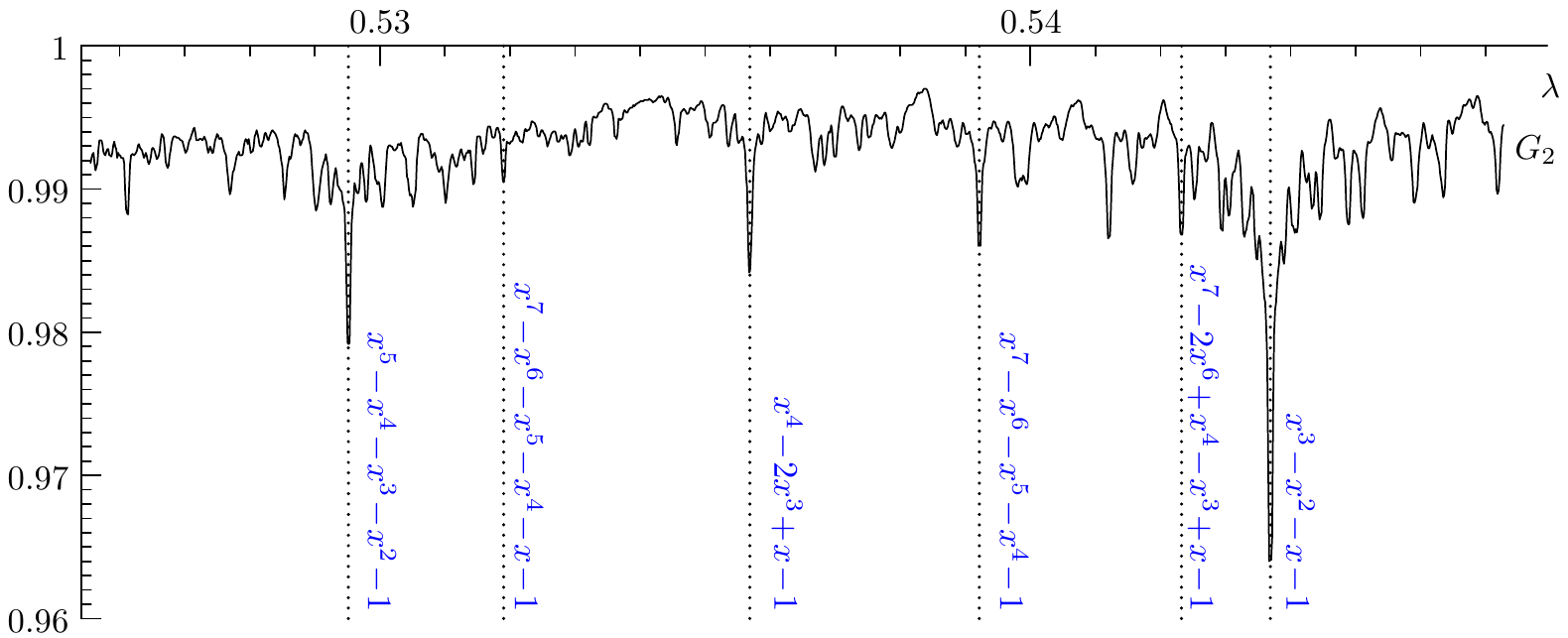}
   \caption{$0.523 < \lambda < 0.548$ }
   \label{fig:bc2}
   \end{subfigure}
   
    \caption{The plot of the piecewise constant function $G_{2}(\lambda)$, which gives
    lower bounds on correlation dimension of Bernoulli convolution $D_2(\mu_\lambda)$. }
    \label{fig:plotBC}
\end{figure}

\begin{figure}
   \ContinuedFloat
   \begin{subfigure}{150mm}
   \includegraphics{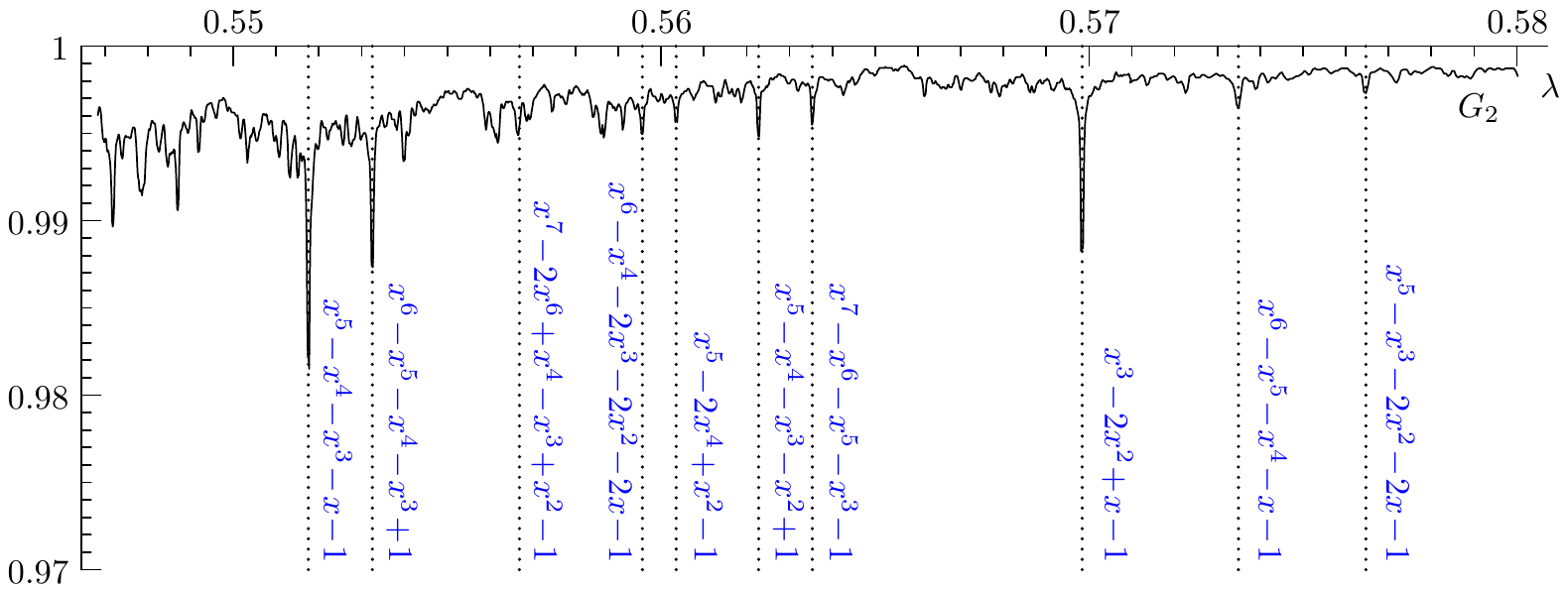}
    \caption{$0.548 < \lambda < 0.58$ }
   \label{fig:bc3}
   \end{subfigure}

   \begin{subfigure}{150mm}
   \includegraphics{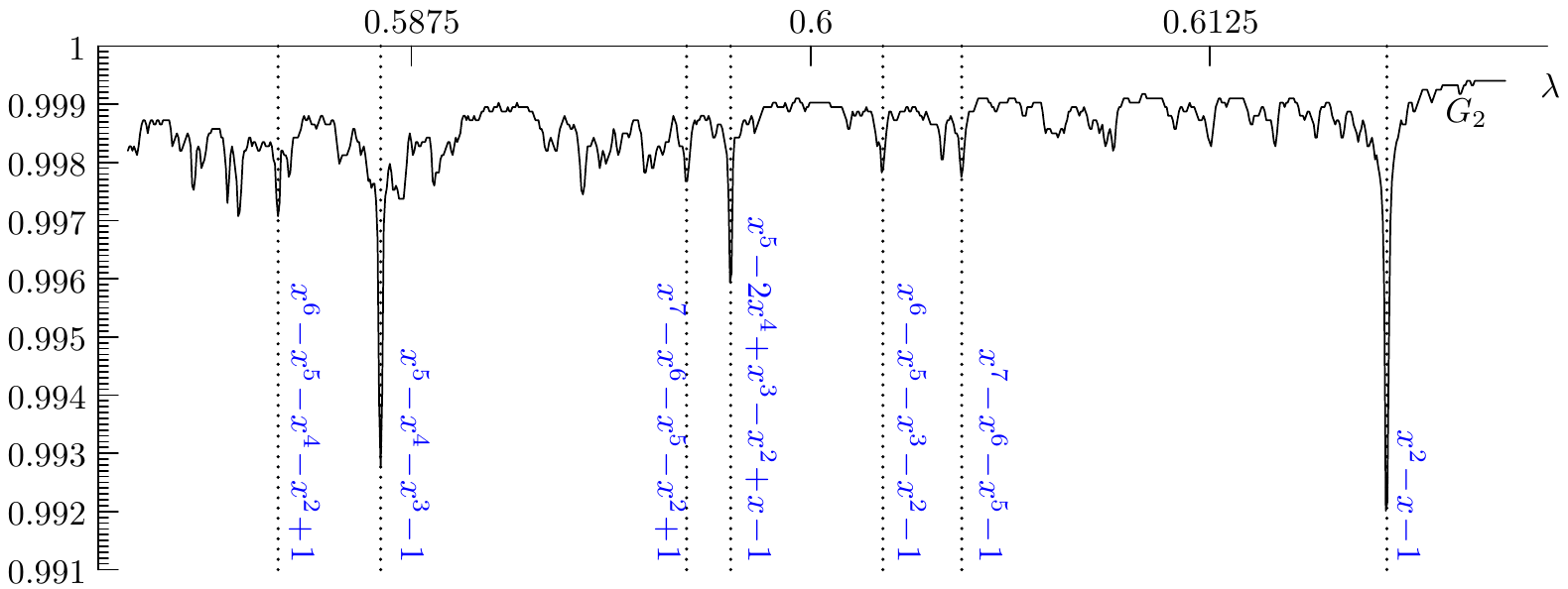}
   \caption{$0.58 < \lambda < 0.625$ }
   \label{fig:bc4}
   \end{subfigure}
 
   \begin{subfigure}{150mm}
   \includegraphics{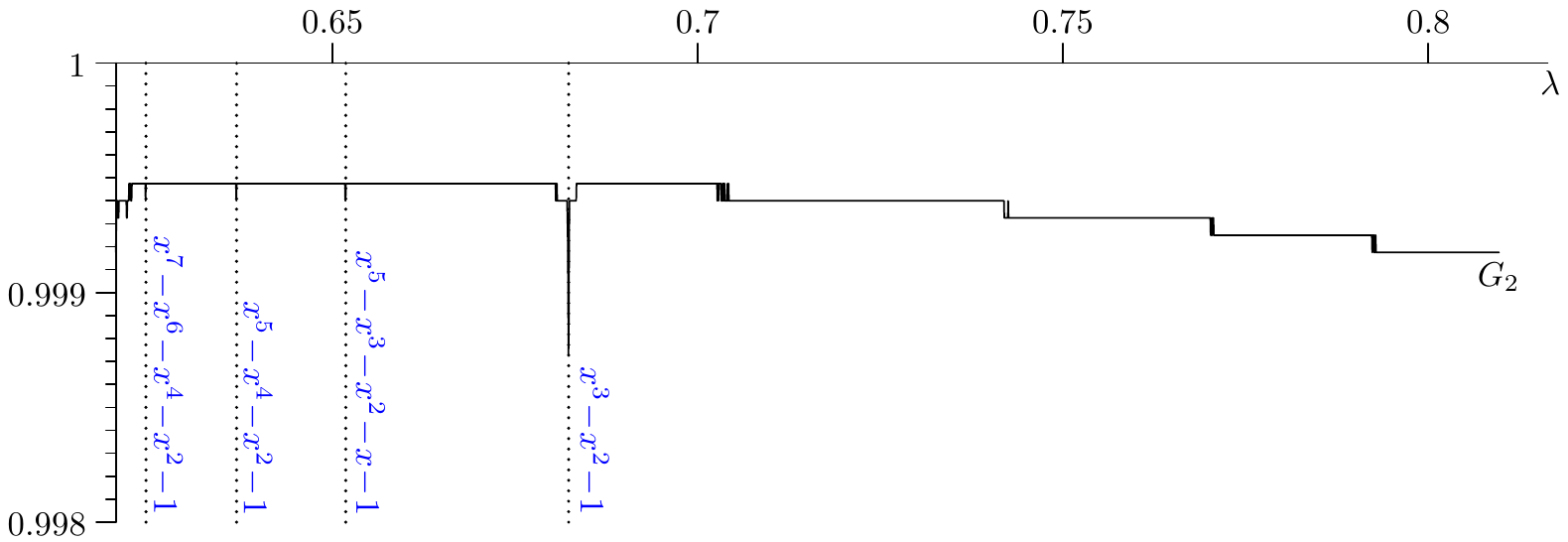}
    \caption{$0.625 < \lambda < 0.81$ }
   \label{fig:bc5}
   \end{subfigure}

    \caption{(Continued). The plot of the piecewise constant function $G_{2}(\lambda)$, which gives
    lower bounds on correlation dimension of Bernoulli convolution $D_2(\mu_\lambda)$. 
    The polynomials indicated are the minimal polynomials of the corresponding values
    $\lambda$, which are Pisot. }
    \label{fig:plotBC}
\end{figure}

Based on the graph of the lower bound function~$G_2(\lambda)$ shown in Figure~\ref{fig:plotBC} we 
\emph{conjecture} that for reciprocals of Fibonacci $\lambda = 2/(1+\sqrt 5)$ and ``tribonacci'' ($\lambda = \beta^{-1}$, where $\beta$ is the largest
root of $x^3-x^2-x-1$) parameter values there exists a sequence $\lambda_n^\prime$ of Pisot
numbers such that $\lambda_n^\prime \to \lambda$ as $n\to\infty$ such that the dimensions
$D_2(\mu_{\lambda'_n})<1-\varepsilon$ for some $\varepsilon>0$ and moreover
the limit $\lim_{n\to \infty} D_2(\mu_{\lambda'_n}) < 1$. 

\subsubsection{Proof of Theorem~\ref{thm:dim013} }
Contrary to the case of Bernoulli convolutions, there were no apriori estimates
on dimension drop known. 
%Therefore, we proceed in a similar way with some
%adjustments.

First, we consider a cover of the interval of parameter values
$(0.249,0.334)$ by $100$ overlapping intervals $\Lambda_k$, $k=1,\ldots,100$. 
We set the limit $K=200$ for the number of iterations and $N=10^5$ for the
number of intervals for piecewise constant functions. We then choose $\varepsilon = 0.01$
as the refinement parameter and $d_1=0.5$, $d_2 = -\frac{\log
3}{\log\inf\Lambda_k}$ as lower and upper bounds, respectively. 
Applying the algorithm described in Section~\ref{sss:compute} we obtain rough
estimates. The result is presented in Table~\ref{tab:013step1}. Afterwards, we improve this estimate. We choose the refinement parameter
$\varepsilon = 10^{-4}$, set $N = 10^7$ to be the number of intervals for the
step function, $K = 1000$ for the maximal number of iterations and choose the
lower bound which was already computed. 

The lower bounds for $\lambda \in (0.333,0.401)$ we compute applying the same
steps with $d_0 = 0$ and $d_1 = 1$. 

\begin{table}
    \centering
    \begin{tabular} {|cc||cc|}
        \hline 
        $\Lambda$ & $\alpha$ &  $\Lambda$ & $\alpha$ \\
        \hline
 $[0.25000,0.26501]$& $0.77082$ & $[0.32839, 0.33173]$ &$0.85657$ \\
 $[0.26501,0.26918]$& $0.79581$ & $[0.33173, 0.33434]$ &$0.79659$ \\
 $[0.26918,0.28169]$& $0.80051$ & $[0.33434, 0.33702]$ &$0.87375$ \\
 $[0.28169,0.28669]$& $0.83549$ & $[0.33702, 0.34372]$ &$0.89479$ \\
 $[0.28669,0.29086]$& $0.85245$ & $[0.34372, 0.34908]$ &$0.91583$ \\
 $[0.29086,0.30587]$& $0.83663$ & $[0.34908, 0.35712]$ &$0.94388$ \\
 $[0.30587,0.30838]$& $0.87025$ & $[0.35712, 0.36717]$ &$0.91583$ \\
 $[0.30838,0.31338]$& $0.86111$ & $[0.36717, 0.37722]$ &$0.95440$ \\
 $[0.31338,0.32089]$& $0.88076$ & $[0.37722, 0.38526]$ &$0.95791$ \\
 $[0.32089,0.32839]$& $0.86694$ & $[0.38526, 0.40000]$ &$0.98246$ \\
        \hline
    \end{tabular}
    \caption{
Uniform lower bounds for the correlation dimension of the stationary measure in
the $\{0,1,3\}$-problem, after the first step.  }
\label{tab:013step1}
\end{table}

The result is presented in Figure~\ref{fig:plot013}. 
We managed to identify minimal polynomials of algebraic numbers which seem to 
correspond to some of the biggest dimension drops and verified that the corresponding parameter values 
are reciprocals of hyperbolic numbers. 
%\footnote{An algebraic number is called \emph{hyperbolic}, if all its Galois conjugates lie inside the unit circle}. %, i.e. all their Galois conjugates lie inside the unit circle. 
In the case of the $\{0,1,3\}$-system, the overlaps conjecture implies that 
the dimension drop \emph{for algebraic parameter values} takes place only for the roots of polynomials with
coefficients $\{0,\pm1,\pm2,\pm 3\}$, and the polynomials we have identified
satisfy this property. 

\begin{figure}
  
   \begin{subfigure}{150mm} 
   \includegraphics{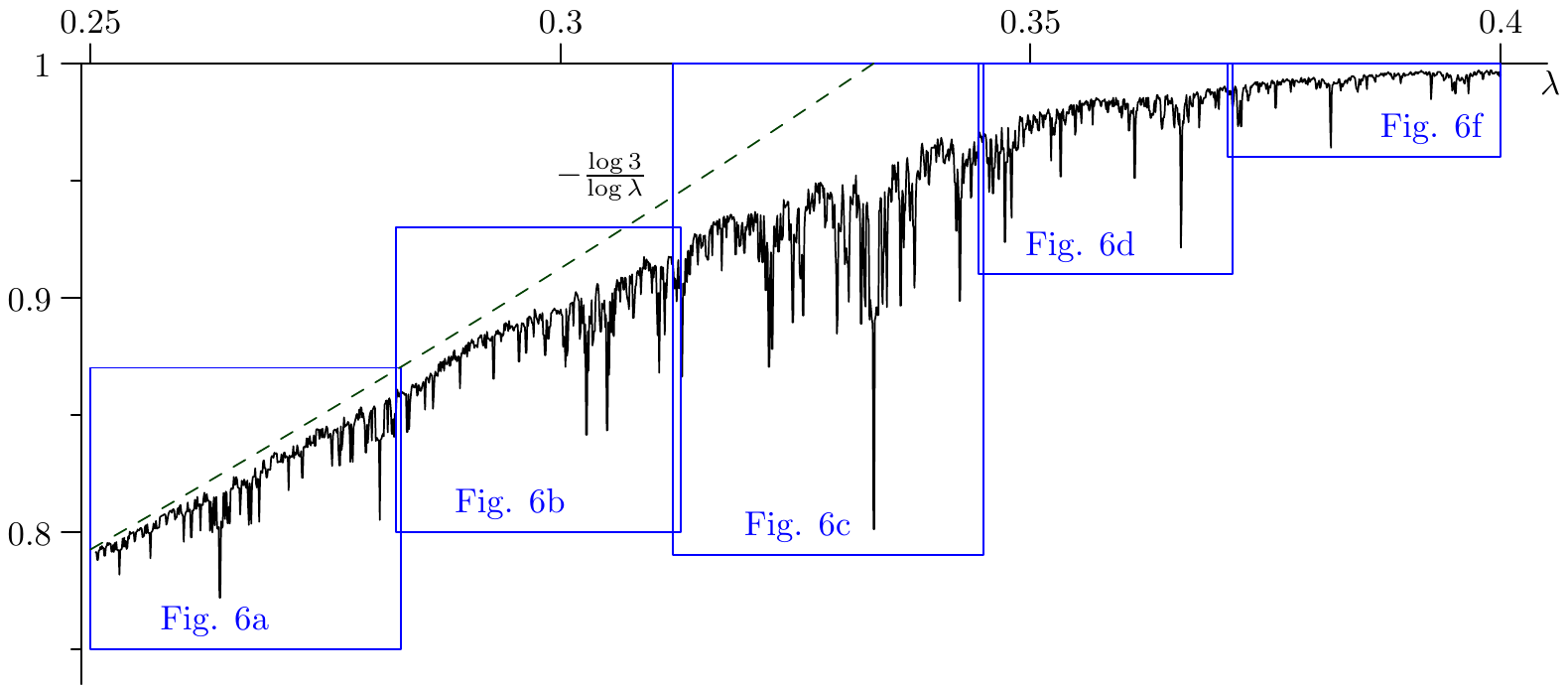}
   \caption{Plot of lower bounds for correlation dimension of Bernoulli
   convolution measures for $0.249 \le \lambda \le 0.401$ with more detailed plots
   in~\ref{fig:013-p1},~\ref{fig:013-p2},
   \ref{fig:013-p3},~\ref{fig:013-p4},~\ref{fig:013-p5}
   below.} 
   \end{subfigure}
   
   \begin{subfigure}{150mm} 
   \includegraphics{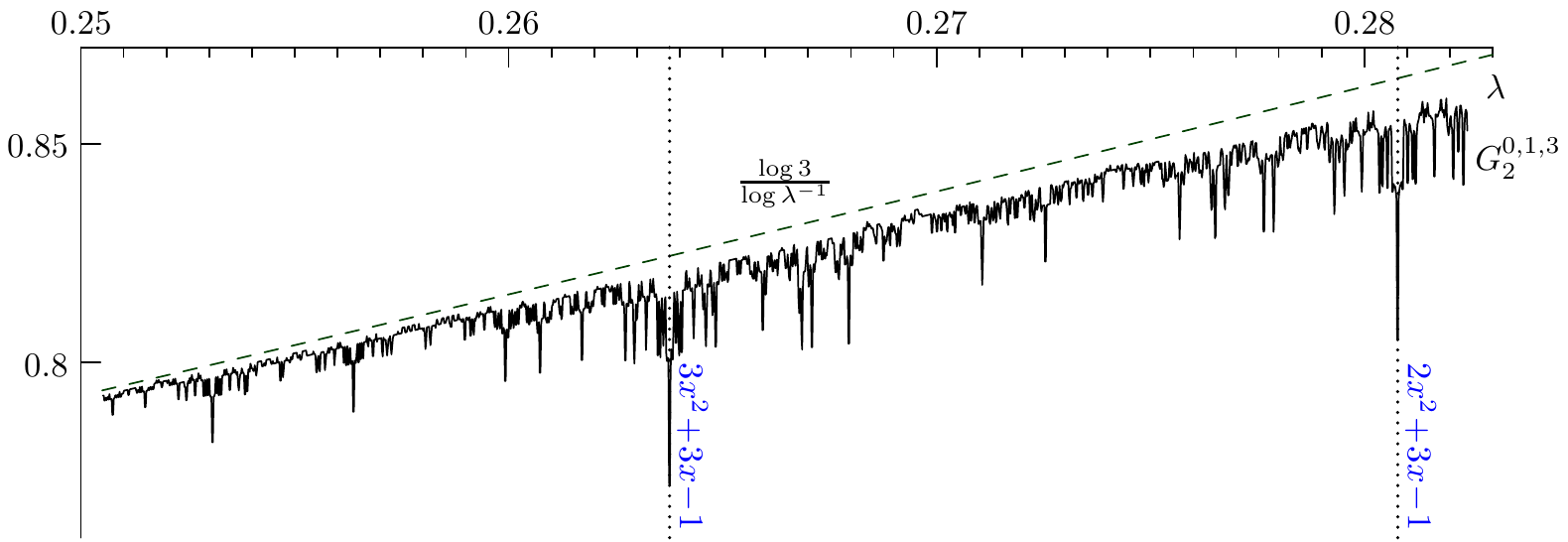}
   \caption{$0.25<\lambda<0.283$}
   \label{fig:013-p1}
   \end{subfigure}
   
   \begin{subfigure}{150mm} 
   \includegraphics{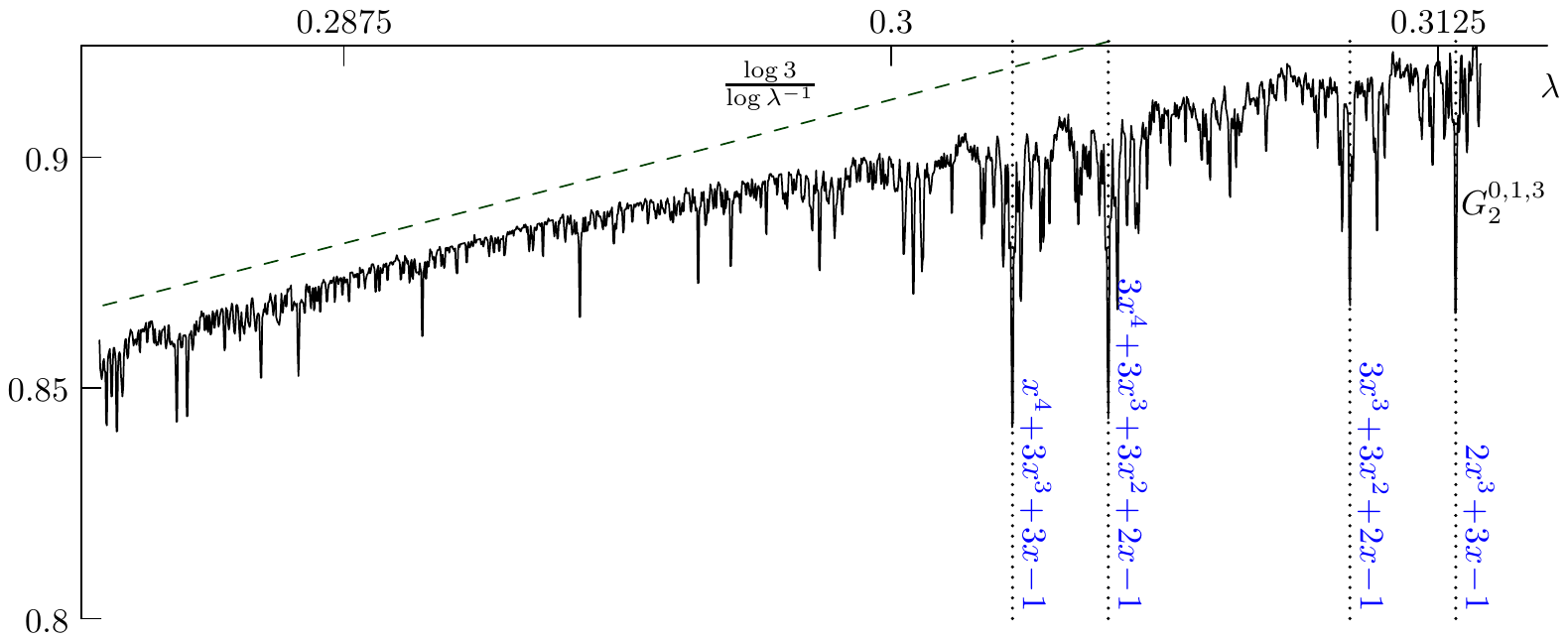}
    \caption{$0.2825<\lambda<0.3128 $}
   \label{fig:013-p2}
   \end{subfigure}

   \caption{The plot of the piecewise constant function $G^{0,1,3}_{2}(\lambda)$, which gives
   lower bounds on Hausdorff dimension $D_2(\mu^{0,1,3}_\lambda)$ of the stationary
   measure for the $\{0,1,3\}$-system. The polynomials indicated are the minimal polynomials 
   of the corresponding values $\lambda$, which are hyperbolic.}
    \label{fig:plot013}
\end{figure}

\begin{figure}
   \ContinuedFloat
   \begin{subfigure}{150mm} 
   \includegraphics[height=67mm]{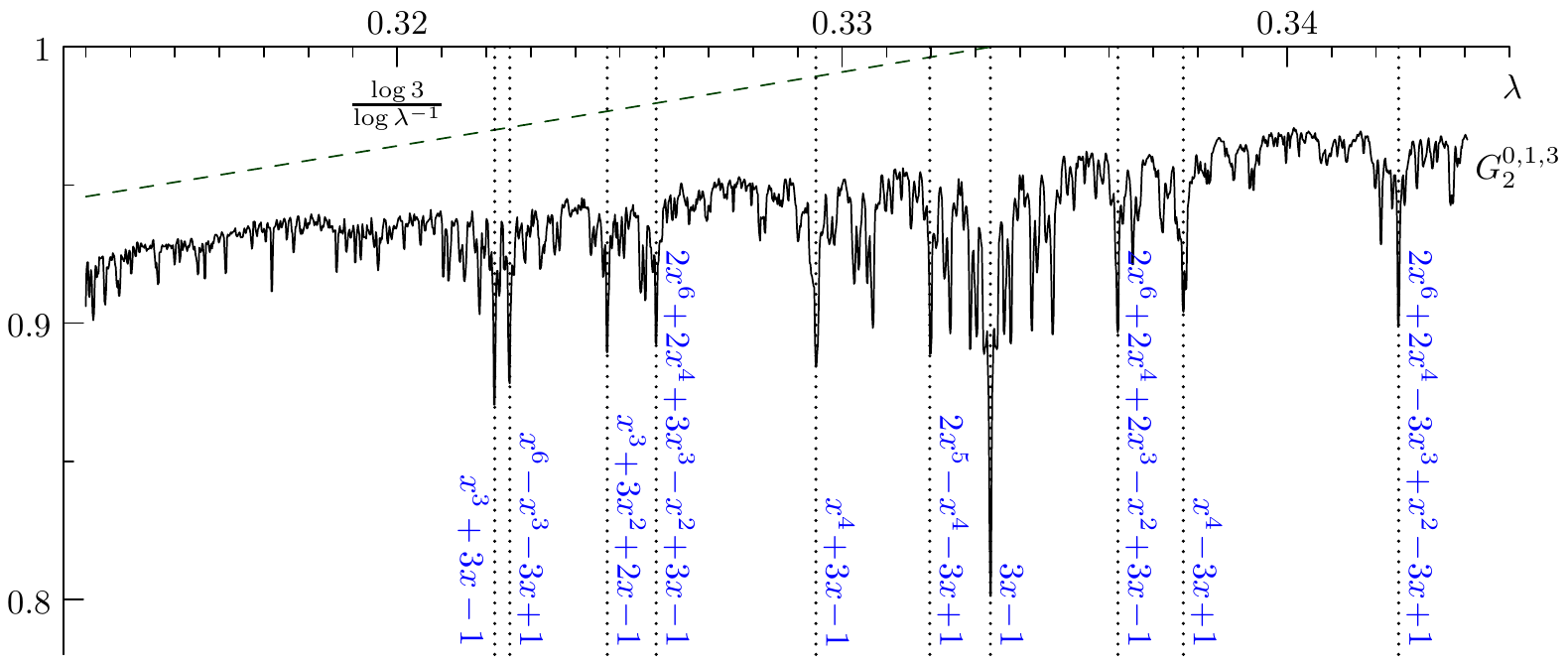}
   \caption{$0.312<\lambda<0.345 $}
   \label{fig:013-p3}
   \end{subfigure}
 
   \begin{subfigure}{150mm} 
       \includegraphics[height=67mm]{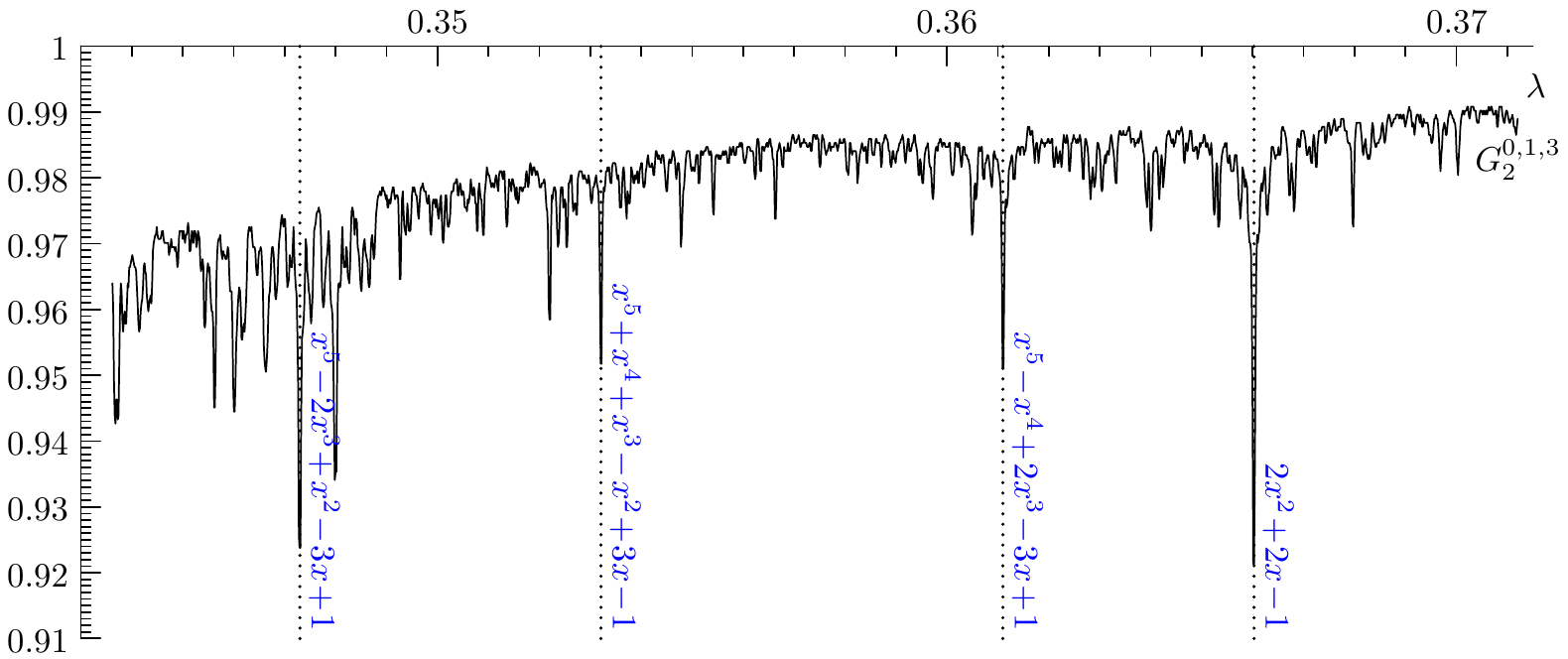}
   \caption{$0.3445<\lambda<0.3715 $}
   \label{fig:013-p4}
   \end{subfigure}

   \begin{subfigure}{150mm} 
   \includegraphics[height=67mm]{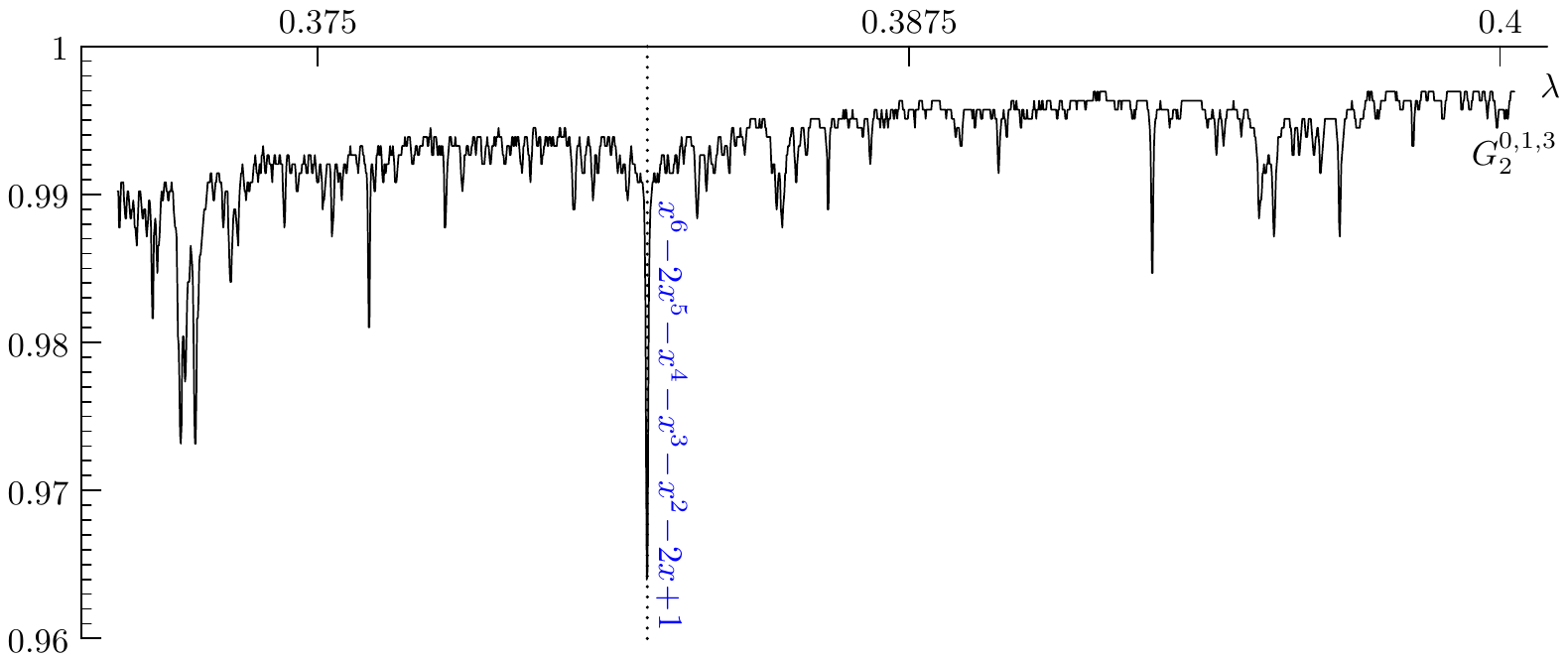}
   \caption{$0.3710<\lambda<0.4 $}
   \label{fig:013-p5}
   \end{subfigure}
  
   \caption{The plot of the piecewise constant function $G^{0,1,3}_{2}(\lambda)$, which gives
   lower bounds on Hausdorff dimension $D_2(\mu^{0,1,3}_\lambda)$ of the stationary
    measure for the $\{0,1,3\}$-system. }
    \label{fig:plot013}
\end{figure}

\subsection{Selected algebraic parameter values}
\label{s:algebraicLambda}
So far we have applied our method to compute uniform lower bounds on dimension
of the stationary measures. As we highlighted already in the end of
\S\ref{ss:efbounds-1}, in order to get a lower bound on $D_2(\mu_\lambda)$ 
for an algebraic~$\lambda$, we need to consider a small interval containing the value. 
In this section, we compute a lower bound on correlation, and hence
Hausdorff, dimensions of $\mu_\lambda$, for selected algebraic values and compare
our results with the existing data. For some specific values these results are not
as accurate as existing estimates. For other parameter values, e.g. Salem
numbers, we give a new improved lower bound. 

We begin by recalling some known results. %Let $0 < \lambda < 1$ be an algebraic number. 
In~\cite{G63} Garsia introduced a notion of entropy of an algebraic
number~$\lambda$ (also see~\cite{HKPS19} for an alternative definition)
$$
h(\lambda) = \lim_{N \to +\infty} -\frac{1}{2^n} 
\sum_{i_1, \cdots, i_N \in \{0,1\}} \log \left(
\frac{1}{2^n}\hbox{\rm Card}\left\{
j_1, \cdots, j_N \in \{0,1\} \hbox{ : } \sum_{k=1}^n(i_k-j_k)\lambda^k = 0 
\right\} \right).
$$
Garsia entropy was first used to estimate Hausdorff dimension of Bernoulli
convolution corresponding to the Golden mean~$\lambda = \frac{2}{1+\sqrt5}$. The
method has been subsequently extended in~\cite{GKT02} to the roots of the polynomials 
$$
P_n(x) = x^n - x^{n-1} - \ldots - x - 1.
$$ 
In Table~\ref{tab:multinacci} we give a comparison of the lower bounds we have
computed using the diffusion operator and the results of Grabner at
al.~\cite{GKT02}. This illustrates that our bounds are quite close to the known
values.

\begin{table}
    \centering
\begin{tabular}{|ccc||ccc|}
\hline
$n$ & $\dim_H(\mu_\lambda)$ & $\alpha$ & $n$ & $\dim_H(\mu_\lambda)$ & $\alpha$ \\
$2$ &  $ 0.995713126685555$   & $  0.992395833333  $  & $6$ &  $ 0.996032591584967$   & $  0.990673828125  $  \\ 
$3$ &  $ 0.980409319534731$   & $  0.964214555664  $  & $7$ &  $ 0.997937445507094$   & $  0.994959490741  $  \\    
$4$ &  $ 0.986926474333800$   & $  0.973324567994  $  & $8$ &  $ 0.998944915449832$   & $  0.997343750000  $  \\     
$5$ &  $ 0.992585300274171$   & $  0.983559570313  $  & $9$ &  $ 0.999465368055570$   & $  0.998640046296  $  \\    
\hline
\end{tabular}
\caption{Comparison of the lower bound for the correlation dimension $\alpha < D_2(\mu_\lambda)$ computed using
the diffusion operator and the Hausdorff dimension computed in~\cite[\S4]{GKT02}
for mutlinacci parameter values.}
    \label{tab:multinacci}    
\end{table}
The connection between $h(\lambda)$ and $\dim_H(\mu_\lambda)$ comes by a result
of Hochman~\cite{hochman}: 
\begin{equation}
\label{eq:entropy}
\dim_H(\mu_\lambda) = \min\left\{-\frac{h(\lambda)}{\log\lambda},1\right\}.
\end{equation}
%where $h(\mu_\lambda)$ is the {\it Garsia entropy} defined by
%(cf.~\cite{HKPS19},~\cite{G63}) 
However, despite~\eqref{eq:entropy} being exact, the value $h(\lambda)$ for algebraic numbers 
is often quite difficult to estimate for all but a small number of examples, see~\cite{AFKP} (also~\cite{lalley})
which gave algorithms to compute the entropy based on Lyapunov exponents of random matrix products.
For several explicit (non-Pisot) examples they showed that
$\frac{h(\lambda)}{\log(\lambda^{-1})} > 1$; together with~\eqref{eq:entropy}
this implies $\dim_H(\mu_\lambda) =1$.

On the other hand, Breuillard and Varj\'u~\cite{BV20}
gave an estimate on $h(\lambda)$ in terms of the Mahler measure~$M_\lambda$:  
\begin{equation}
    \label{eq:BV20}
c \cdot \min\left\{1, \log M_\lambda\right\} \le h(\lambda) \le \min\left\{1,
\log M_\lambda\right\}.
\end{equation}
Non-rigorous numerical calculations suggest that one can take $c=0.44$. 
The upper bound in~\eqref{eq:BV20} is often strict. In particular, it is known that
$h(\lambda) < M_\lambda$, provided~$\lambda$ has no Galois conjugates on the
unit circle~\cite{BV20}. 

%Algebraic numbers with this property are sometimes called~\emph{hyperbolic}. 
The dimension of the Bernoulli convolution measure for certain hyperbolic parameter
values~$\lambda$ can be computed explicitly, too.
In a recent work~\cite{HKPS19} Hare et al. considered hyperbolic
algebraic numbers of degree~$5$. For a number of them they showed that the
stationary measure has full Hausdorff dimension~\cite[Tables 5.1, 5.2]{HKPS19}. 
We present our lower bound for the correlation dimension for comparison in
Table~\ref{t:old}, which shows that our lower bounds are accurate to~$3$ decimal
places.
\begin{table}[h] 
\begin{center}
\begin{tabular}{|cccl|}
\hline
$\lambda $ & $\alpha$ &  $ \beta=\lambda^{-1} $ & \qquad \quad polynomial \\
 $  0.862442360254 $ &    $ 0.999609375000  $&$1.159497777573$ &  $x^5 + x^4 - x^3 - x^2 - 1$ \\
 $  0.827590407756 $ &    $ 0.999687500000  $&$1.208327199818$ &  $x^5 - x^4 + x^3 - x - 1$ \\
 $  0.874449227129 $ &    $ 0.999609375000  $&$1.143576972768$ &  $x^5 + x^3 - x^2 - x - 1$ \\
 $  0.710434255787 $ &    $ 0.999843750000  $&$1.407589783086$ &  $x^5 - x^4 + x^3 - x^2 - x - 1$ \\
 $  0.779544663821 $ &    $ 0.999765625000  $&$1.282800135015$ &  $x^5 - x^3 - x^2 + x - 1$ \\
 $  0.791906429308 $ &    $ 0.999765625000  $&$1.262775452996$ &  $x^5 - x^4 + x^2 - x - 1$ \\
 $  0.786151377757 $ &    $ 0.999765625000  $&$1.272019649514$ &  $x^4 - x^2 - 1$ \\
 $  0.699737022113 $ &    $ 0.999843750000  $&$1.429108319838$ &  $x^5 - x^3 - x^2 - 1$ \\
 $  0.779544663821 $ &    $ 0.999765625000  $&$1.282800135015$ &  $x^5 - x^3 - x^2 + x - 1$ \\
 $  0.800094994405 $ &    $ 0.999765625000  $&$1.249851588864$ &  $x^5 - x^4 + x^3 - x^2 - 1$ \\
 $  0.876611867657 $ &    $ 0.999609375000  $&$1.140755717433$ &  $x^5 + x^4 - x^3 - x - 1$ \\
 $  0.655195524260 $ &    $ 0.999843750000  $&$1.526261952307$ &  $x^5 - x^4 - x^2 - x + 1$ \\
 $  0.848374895732 $ &    $ 0.999687500000  $&$1.178724176105$ &  $x^4 + x^3 - x^2 - x - 1$ \\
 $  0.819172513396 $ &    $ 0.999687500000  $&$1.220744084605$ &  $x^4 - x - 1$ \\
 $  0.774804113215 $ &    $ 0.999765625000  $&$1.290648801346$ &  $x^4 - x^3 + x^2 - x - 1$ \\
 $  0.730440478359 $ &    $ 0.999843750000  $&$1.369036943635$ &  $x^5 - x^3 - x^2 - x + 1$ \\
 $  0.833363173425 $ &    $ 0.999687500000  $&$1.199957031806$ &  $x^5 - x^3 + x^2 - x - 1$ \\ 
\hline
\end{tabular}
\end{center}
\caption{Lower bound $\alpha < D_2(\mu)$ for the correlation dimension for selected algebraic numbers for
which it is known~\cite{HKPS19} that $\dim_H(\mu_\lambda)=1$, computed  
using~$5\cdot10^6$ partition intervals and~$500$ iterations of the diffusion
operator with the refinement parameter $\varepsilon=10^{-4}$. The
value of the root~$\beta$ is given to simplify the comparison with~\cite{HKPS19}.}
\label{t:old}
\end{table}
On the other hand, there are a number of algebraic parameter values to which the method presented
in~\cite{HKPS19} doesn't apply, they are listed in~\cite[Table 5.3]{HKPS19}. For
these values we give a new lower bound in Table~\ref{t:new}. 
\begin{table}[h]
\begin{center}
\begin{tabular}{|cccl|}
\hline
    $\lambda $ & $\alpha$ &  $ \beta=\lambda^{-1} $ & \qquad \quad polynomial \\
%\hline
  $0.593423522613  $	&	$   0.998714285714 $	&	$ 	 1.685137110165 $	&	$ 	 z^5-z^4-z^2-z-1 $\\
%\hline
  $0.595089298038  $	&	$   0.999000000000 $	&	$ 	 1.680420070225 $	&	$ 	 z^5-z^4-z^3-z+1 $\\
%\hline
  $0.557910446633  $	&	$   0.997857142857 $	&	$ 	 1.792402357824 $	&	$ 	 z^5-z^4-z^3-z^2+z-1 $ \\
%\hline
  $0.712452611946  $	&	$   0.999800000000 $	&	$ 	 1.403602124874 $	&	$ 	 z^5-z^4-z^2+z-1 $ \\
%\hline
  $0.645200388386  $	&	$   0.999800000000 $	&	$ 	 1.549906072594 $	&	$ 	 z^5-z^4-z^3+z-1 $ \\
%\hline
  $0.667960707496  $	&	$   0.999800000000 $	&	$ 	 1.497094048762 $	&	$ 	 z^5-z^4-z-1 $ \\
%\hline
  $0.808730600479  $	&	$   0.999600000000 $	&	$ 	 1.236505703391 $	&	$ 	 z^5-z^3-1 $ \\
%\hline
  $0.837619774827  $	&	$   0.999722222222 $	&	$ 	 1.193859111321 $	&	$ 	 z^5-z^2-1 $ \\
%\hline
  $0.856674883855  $	&	$   0.999652777778 $	&	$ 	 1.167303978261 $	&	$ 	 z^5-z-1 $\\
%\hline
  $0.889891245776  $	&	$   0.999513888889 $	&	$ 	 1.123732821001 $	&	$ 	 z^5+z^4-z^2-z-1 $ \\
\hline
\end{tabular}
\end{center}
\caption{Lower bound $\alpha<D_2(\mu)$ for the correlation dimension for selected algebraic numbers for which there are no previous lower bounds, computed 
using~$5\cdot10^6$ partition intervals and~$500$ iterations of the diffusion
operator and the refinement parameter $\varepsilon =10^{-4}$.
The value of the root~$\beta$ is given to simplify the comparison
with~\cite{HKPS19}. }
\label{t:new}
\end{table}

\subsubsection{Estimates for Salem numbers: Proof of Theorem~\ref{thm:salemnum}}
In a recent work, Breuillard and Varj\'u state an open problem~\cite[Problem
3]{BV20}, asking whether it is true that $h(\lambda) = M_\lambda$ for all Salem
parameter values $\lambda \in \left(\frac12,1\right)$. 
This equality would imply that $\dim_H(\mu_\lambda)=1$ for Salem parameter
values. 
 
We apply the method described in~\S\ref{sss:compute} to $99$ Salem numbers of
degree no more than~$10$ and to~$47$ small Salem numbers. Our computational set-up
had the following choices. First, we compute each Salem number $s_k$ with an accuracy
of $10^{-32}$ and consider a neighbourhood of radius $\delta=10^{-8}$, i.e. $\Lambda_k =
B_\delta(s_k)$. Then, based on the existing results, we choose $d_1 = 0.98$ and
$d_2 = 1$, as conjectured lower and upper bounds. The number of intervals for piecewise constant
functions is $N=6\cdot10^5$, and the allowed number of iterations for the
diffusion operator is $K=300$. We also set the refinement parameter $\varepsilon = 10^{-4}$. 
The detailed result is presented in Appendix \S\ref{ap:salem}
and~\S\ref{ap:smallsalem}.  

\section[Proof of Theorem 1.8]{Asymptotic bounds: proof of Theorem~\ref{thm:near1} }
\label{s:abounds}
We have provided a uniform lower bound on the correlation dimension
$D_2(\mu_\lambda)$ of Bernoulli convolution measures $\mu_\lambda$. %, for $\frac{1}{2} < \lambda < 1$ .
Now we will give an asymptotic lower bound for $D_2(\mu_\lambda)$ in a
neighbourhood of~$1$ using the diffusion operator approach.  

\begin{proof}[of Theorem~\ref{thm:near1}]
%%%%%%%%%%%%%%%%%%%%%%%%%%
        Given a small $\varepsilon>0$ let us set $\lambda = 1 - \varepsilon$.
        Then the symmetric diffusion operator~\eqref{eq:dif-twoway} takes the form
        $$
        [\Ds_{\alpha,\mS}\psi](x) = (1-\varepsilon)^{-\alpha}\frac14
        \left(\psi\left(\frac{x+1}{1-\varepsilon}\right) + 2
        \psi\left(\frac{x}{1-\varepsilon}\right) +
        \psi\left(\frac{x-1}{1-\varepsilon}\right) \right)
        $$
        In order to prove the result, it is sufficient to find a function
        $f_\varepsilon$ such that for any $c>\frac32$ and for any
        $\alpha < 1-c\varepsilon$ we have that 
        \begin{equation}
            \label{eq:dsfe}
            \Ds_{\alpha,\mS} f_\varepsilon \prec f_\varepsilon. 
        \end{equation}
        We will specify the function $f_\varepsilon$ explicitly. Let us introduce a
        shorthand notation $\delta:=2\varepsilon-\varepsilon^2>0$ and define 
        \begin{equation}
        \label{rhs:eq}
        f_\varepsilon(x):= \exp(-\delta(1-\varepsilon)^2x^2) = \exp(-\delta
            x^2) \cdot \exp(\delta^2 x^2)
        \end{equation}
        It is not difficult to see that $f_\varepsilon$ satisfies~\eqref{eq:dsfe}.
        Indeed, note that $\frac{1+\exp(-\delta)}2\le
        (1-\varepsilon)^{1-c\varepsilon}$ for any $\varepsilon$ sufficiently
        small and any $c>\frac32$. Therefore to establish~\eqref{eq:dsfe} it is
        sufficient to show that 
        \begin{equation}
            \label{near1:eq}
        \frac14\left(f_\varepsilon\left(\frac{x+1}{1-\varepsilon}\right)+
        f_\varepsilon\left(\frac{x-1}{1-\varepsilon}\right)+2f_\varepsilon\left(\frac{x}{1-\varepsilon}\right)\right)\le
        \frac{1+\exp(-\delta)}{2}f_\varepsilon(x).
        \end{equation}
        To prove \eqref{near1:eq} we first note that 
   $$
    \begin{aligned}
            f_\varepsilon\left(\frac{x+1}{1-\varepsilon}\right) +
            f_\varepsilon\left(\frac{x-1}{1-\varepsilon}\right) &=
            \exp(-\delta (x+1)^2) + \exp(-\delta(x-1)^2) \cr &= \exp(-\delta) \cdot
            \exp(-\delta x^2) \left( \exp(-2 \delta x) + \exp(2\delta x)\right)
            \cr &= 2 \exp(-\delta) \cdot \exp(-\delta x^2) \cdot \cosh(2\delta x).
    \end{aligned}
    $$
       Moreover, since $f_\varepsilon\left(\frac{x}{1-\varepsilon}\right) = \exp(-\delta x^2)$ we
        conclude for the left hand side of~\eqref{near1:eq} that 
        \begin{equation}
            \label{lhs:eq}
        \frac14\left(f_\varepsilon\left(\frac{x+1}{1-\varepsilon}\right)+
        f_\varepsilon\left(\frac{x-1}{1-\varepsilon}\right)+2f_\varepsilon\left(\frac{x}{1-\varepsilon}\right)\right)
        = \frac12  \exp(-\delta x^2) \cdot \left(1 +  \cosh(2\delta x) \exp(-\delta)
        \right).
        \end{equation}
        Combining~\eqref{rhs:eq} and~\eqref{lhs:eq} we see that~\eqref{near1:eq}
        is equivalent to
        $$
            1 + \exp(-\delta)\cosh(2\delta x) \le (1+\exp(-\delta) )\exp \left(
            \delta^2  x^2\right), 
        $$
        which in turn, is equivalent to
        \begin{equation}\label{ineq:eq}
            \frac{1} {1+\exp(-\delta) } + \frac{\exp(-\delta)} {1+\exp(-\delta) }\cosh(2\delta x) \le \exp \left(
            \delta^2  x^2\right), 
        \end{equation}
        To establish (\ref{ineq:eq}) it is sufficient to show that 
        $$
        \frac12+\frac12\cosh(2\delta x) \le \exp\left( \delta^2 x^2\right).  
        $$
 %       and use convexity. 
 This last inequality  can be established by comparision of the
        Taylor series coefficients term by term. More precisely, the coefficient in front of
        the term $(\delta x)^{2k}$ of the function $\cosh(2\delta x)$ is
        $\frac{2^{2k}}{(2k)!}$ and the same coefficient of the function
        $\exp\left(\delta^2 x^2\right)$ is equal to $\frac{1}{k!}$.  
\end{proof}

\section[Diffusion operator $\Ds_{\alpha,\lambda}$]{Diffusion operator $\Ds_{\alpha,\lambda}$ and correlation dimension}
\label{section:diffusion}

We would like to start by explaining the idea behind the diffusion operator and 
its connection with the correlation dimension which lead us to it. 

Let us recall the energy integral~\eqref{eq:a-int}
$$
I(\mu,\alpha) = \int_\bbR \int_\bbR |x-y|^{-\alpha} \mu(dx) \mu(dy)
$$
and the definition of the correlation dimension~\eqref{eq:D2-bis}: $D_2(\mu) =
\sup\{\alpha \colon I(\mu,\alpha) \mbox{ is finite } \}$. In other words,
$I(\mu,\alpha)$ is finite for any $\alpha < D_2(\mu)$. 

Let $\mS(\lambda,\bar c, \bar p)$ be an iterated function scheme of~$n$
similarities $f_j = \lambda x - c_j$, $j = 1, \ldots, n$ and let $\mu_\lambda$
be its stationary measure. Then $\mu_\lambda$ is the fixed point of the operator 
on Borel probability measures
$$
T_{\mS} \colon \mu \mapsto \sum_{j=1}^n p_j {f_j}_* \mu
$$
We now would like to study the induced action of~$T_{\mS}$ on~$I(\mu,\alpha)$. 
To this end, we want to incorporate $I(\mu,\alpha)$ into a family.
More precisely, we consider a family of functions given by
\begin{equation}
\label{eq:psi-def}
\psi_{\alpha,\mu}:\mathbb R\to\mathbb R^+\cup \{+\infty\} \qquad
\psi_{\alpha,\mu}(r) : = \int_{\bbR}\int_{\bbR} |(x-y)-r|^{-\alpha} \,
\mu(dx) \, \mu(dy).
\end{equation}

\begin{notation}
We denote by $-\mu$ the push-forward of the measure~$\mu$ under $x \mapsto -x$. 
\end{notation}
In the sequel, we will need the following technical lemma which 
 helps us to decide whether or not $\psi_{\alpha,\mu}(r)$ is finite\footnote{Or in other words whether the function $(x-y-r)^{-\alpha}$ is
integrable with respect to $\mu \times \mu$.}. % for an appropriate choice of $\alpha$. 
\begin{lemma}
  \label{lem:psi}
  Let $\mu$ be a probability measure. 
  Assume that $I(\mu,\alpha) = \psi_{\alpha,\mu}(0)$ is finite. Then
  $\psi_{\alpha,\mu}(r)$ is finite for any $r \in \bbR$ and, moreover, we have that
  $\psi_{\alpha,\mu}(r) < \psi_{\alpha,\mu}(0)$. In particular,
  $\psi_{\alpha,\mu}$ is a continuous function. 
\end{lemma}
\begin{proof}
 Let us denote $\nu \colon = \mu*(-\mu)$. It is easy to see that its Fourier
 transform is a nonnegative function: 
 $$
     \hat \nu(t) = \int_\bbR e^{-itz} \nu(dz) = \int_\bbR \int_\bbR e^{-it(x-y)}
     \mu(dx) \mu(dy) = \hat \mu(t) \overline{\hat \mu(t)} = |\mu(t)|^2 \ge 0.
 $$
 We may write
  $$
  \psi_{\alpha,\mu}(r) = \int_{\bbR^2} |x-y-r|^{-\alpha} \mu(dx) \mu(dy) = \int_\bbR
  |z - r|^{-\alpha} \nu(dz).
  $$ 
 Then the desired inequality $\psi_{\alpha,\mu}(r)<\psi_{\alpha,\mu}(0)$ for all $r \in \mathbb R$ is equivalent to 
 $$ 
 \int_{\mathbb R} |z-r|^{-\alpha} \nu (d z) \le \int_{\bbR} |
 z|^{-\alpha} \nu(d z) .
 $$  
% The argument is using the inverse Fourier transform. 
% By assumption 
% $$
% \hat  \nu(t) = \int_{\mathbb R} e^{itz} \nu(dz) > 0.
% $$
%%%%polina%%%% %
%%%%polina%%%% %We begin by computing the Fourier transform of $\nu$.
 Let us consider the function $f_\alpha(s) \eqdef |s|^{-\alpha}$. Its
 Fourier transform is known\footnote{One possible
 approach is via the Gamma function. We rewrite $|s|^{-\alpha} = \frac{2
 \pi^{\alpha/2}}{\Gamma(\alpha/2)}\int_0^\infty t^{\alpha-1} e^{-\pi t^2 s^2}
 dt$ and compute the Fourier transform of the latter by swapping the order of
 integrals.} to be
 \begin{equation}
   \label{eq:FTf}
   \hat f_\alpha(t) = \frac{\pi^{\alpha-1/2}\Gamma\left( (1-\alpha)/2
   \right)}{\Gamma(\alpha/2)} |t|^{\alpha-1} = C_\alpha |t|^{\alpha-1} \ge 0, \quad 0<\alpha<1.
 \end{equation}
 where $C_\alpha = \frac{\pi^{\alpha-1/2}\Gamma\left( (1-\alpha)/2
   \right)}{\Gamma(\alpha/2)}$. 
 Therefore $\widehat{f_\alpha * \nu_\lambda} = \hat  f_\alpha \cdot \hat
 \nu_\lambda $ is real and non-negative. 
Using the inverse Fourier transform formula we obtain an upper bound.
\begin{multline*}
\psi_{\alpha,\mu}(r) = \int_{\mathbb R} |z-r|^{-\alpha} \nu (dz) = 
   (f_\alpha * \nu )(r) =    \frac{1}{2\pi} \int_{\mathbb R}
   e^{-itr}\hat f_\alpha(t) \cdot \hat \nu  (t) dt \\ \le   \frac{1}{2\pi} \int_{\mathbb R}
   \hat f_\alpha(t) \cdot \hat \nu  (t) dt = \psi_{\alpha,\mu}(0) <\infty.
 \end{multline*}
%%%
%%% \begin{multline*}
%%%\int_{\mathbb R} |z-r|^{-\alpha} \nu (dz) = 
%%%   (f_\alpha * \nu )(r) =    \frac{1}{2\pi} \int_{\mathbb R}
%%%   e^{-itr}\hat f_\alpha(t) \cdot \hat \nu  (t) dt \\ \le   \frac{1}{2\pi} \int_{\mathbb R}
%%%   \hat f_\alpha(t) \cdot \hat \nu  (t) dt = \int_{\bbR} C_\alpha
%%%   |t|^{\alpha-1} \int_{\bbR} e^{itz} \nu (d z) \, dt  = \int_{\bbR} C_\alpha
%%%   \int_{\bbR} e^{itz} |t|^{-( 1-\alpha)} dt \, \nu (dz) \\ = C_\alpha^2 \int_{\bbR} |z|^{-\alpha}
%%%   \nu ( dz) = C_\alpha^2 \psi_{\alpha,\lambda}(0) 
%%% \end{multline*}
%%%

The function $\psi_{\alpha,\lambda}$ is continuous since it is an inverse Fourier
 transform of an $L_1$ function. 

\end{proof}

\begin{remark}
\label{rem:a-bound}
Let $m = \inf\{ x \mid \supp \mu*(-\mu) \subseteq [-x,x]\}$ and let $J =
[-a,a]\supsetneq [-m,m]$. 
Then $\psi_{\alpha,\mu}$ has a bounded continuous extension to $\mathbb R \setminus 
J$ 
%\supset \mathbb R \setminus \supp \mu\ominus\mu$ 
with an upper bound
$$
\psi_{\alpha,\mu}(r)\le  (m-a)^{-\alpha} \quad \mbox{ for all } r, \, |r| > a.
$$
\end{remark}
Therefore, the behaviour of $\psi_{\alpha,\mu}$ on $\supp\mu*(-\mu)$ is the most
important to us. 
The next Proposition ties together the symmetric diffusion operator
$\Ds_{\alpha,\mS}$, a family of functions $\psi_{\alpha,\mu}$, and the action on
measures $T_{\mS}$.  
\begin{proposition}
    \label{prop:dspsi}
    Let $\mS(\lambda,\bar c, \bar p)$ be an iterated function scheme of~$n$ similaritites. 
   Let~$\mu$ be a probability measure such that $\supp \mu \subset J$ for a
   closed interval~$J$. Assume that for some $\alpha>0$ the function
   $\psi_{\alpha,\mu}$ is bounded.
   %the integral $I(\alpha,\mu)$ is finite. 
   Then 
   $$
   \psi_{\alpha,T_{\mS} \mu} = \Ds_{\alpha,\mS}\psi_{\alpha,\mu}.
   $$
\end{proposition}
\begin{proof}
    For convenience, recall the definition of the symmetric diffusion
    operator~\eqref{eq:dif-twoway}: 
    $$
    [\Ds_{\alpha,\mS}\psi](x) := \lambda^{-\alpha} \cdot \sum_{i,j=1}^k p_i p_j
    \cdot \psi\left(\frac{x+c_i-c_j}{\lambda}\right). 
    $$
    By straightforward computation, 
    \begin{align*}
        \psi_{\alpha,T_{\mS}\mu}(r) &= \int \int|y-(x-r)|^{-\alpha}
        T_{\mS}\mu(dx)  T_{\mS} \mu(dy) \\
        &=\int\int |y - (x -r)|^{-\alpha} \left(\sum p_j{f_j}_*\mu\right) (dx)
        \left( \sum p_k {f_k}_* \mu \right) (dy) \\
        &=\sum_{j,k} p_jp_k \int \int |f_j(x) - (f_k(y)-r)|^{-\alpha} \mu(dx)
        \mu(dy) \\
        &=\sum_{j,k} p_j p_k \int \int |\lambda x - c_j - \lambda y + c_k +
        r|^{-\alpha} \mu(dx) \mu(dy) \\
        &= \lambda^{-\alpha}\sum_{j,k} p_j p_k \int \int
        \left|x-y+\lambda^{-1}(r-c_j+c_k)\right|^{-\alpha} \mu(dx) \mu(dy) \\
        &=\lambda^{-\alpha} \sum_{j,k} p_j p_k
        \psi_{\alpha,\mu}\left(\lambda^{-1}(r-c_j+c_k)\right) = 
        [\Ds_{\alpha,\mS}\psi](r).
    \end{align*}
\end{proof}
\begin{corollary}
    \label{cor:fixedpoint}
Let $\mu$ be the unique stationary measure~$\mu$ of an iterated function scheme
$\mS(\lambda,\bar c, \bar p)$. Assume that $I(\mu,\alpha)$ is
bounded. Then $\psi(\alpha,\mu)$ is the fixed point of $\Ds_{\alpha,\mS}$. 
\end{corollary}

\begin{remark}
Of course the constant function $f \equiv 1 \subset L^\infty(\mathbb R)$  is an eigenvector for the 
diffusion operator with the eigenvalue $\lambda^{-\alpha}$,  but it does
not satisfy the hypothesis of Theorem~\ref{t:certificate-D2} and is of no use to
us.  
\end{remark}

In the next section we give a proof for Theorem~\ref{t:certificate-D2}, which
provides the grounds for the numerical estimates of the correlation dimension.  
\subsection{Random processes viewpoint}
The random processes viewpoint will be used in the arguments for
Theorems~\ref{t:certificate-D2}, \ref{t:finding-D2}, \ref{t:certificate-D1},
and~\ref{t:finding-D1}. We would like therefore
to make a preparatory description of the  setup. 
\begin{definition}
    \label{def:compproc}
Let $\mS\left(\lambda,\bar c,\bar p \right)$ be an iterated function scheme.
We want to consider the set of pairwise differences $\{d_k \mid d_k = c_i - c_j \}$, a probability vector $q_k = \sum\limits_{i,j \colon c_i - c_j =
d_k} p_i p_j$, and to define \emph{a complementary} iterated function scheme~$\mS(\lambda,\bar d,
\bar q)$:
$$
g_k(x) = \lambda x - d_k. 
$$
\end{definition}
If $\mu$ is the unique stationary measure of $\mS(\lambda,\bar c, \bar p)$ then 
the unique stationary measure of the complementary iterated function scheme is $\mu*(-\mu)$.
Since the measure $\mu$ is compactly supported, we may define: %In the sequel, we use the following notation:
\begin{equation}
    \label{eq:suppmu-mu}
    m : = \inf\{x \mid \supp(\mu*(-\mu)) \subseteq [-x,x] \}.
\end{equation}
%and it corresponds
%to the distribution of the random variable 
%$$
%\sum_{j=1}^\infty \xi_j \lambda^j,
%$$
%where $\xi_j$ is assuming values $d_k$ with probability $q_k$. 
The symmetric diffusion operator can be written in terms of the maps of 
the scheme~$\mS(\lambda,\bar d, \bar q)$:
\begin{equation}
    \label{eq:d2-back}
    [\Ds_{\alpha,\mS}\psi](x) = \lambda^{-\alpha} \sum_{i,j=1}^n p_i p_j
    \psi\left(\frac{x+c_i-c_j}{\lambda}\right) = 
    \lambda^{-\alpha} \sum_k q_k  \psi\left( g_k^{-1}(x)\right) .
%    = \lambda^{-\alpha} \mathbb
%    E \left(\psi\left(\frac{x+\zeta}\lambda\right) \right) 
\end{equation}
This observation brings us to the idea of introducing the backward process associated to
$\mS(\lambda,\bar d, \bar q)$, which can be defined as follows.  

Let $x$ and $y$ be two independent $\mu$-distributed random points 
\begin{equation}
  \label{eq:xy}
  x = \sum_{k=0}^\infty \xi_k\lambda^k,  \quad   y = \sum_{k=0}^\infty \eta_k \lambda^k.
\end{equation}
where $\xi_k, \eta_k$ are i.i.d. random variables assuming values $c_j$ with
probabilities $p_j$, $j = 1, \ldots, n$.  
Consider the random  process given by renormalized differences 
\begin{equation}
    \label{eq:zkproc}
    z_k := \lambda^{-k-1}\left( \sum_{j=0}^k \xi_j\lambda^j -
\sum_{j=0}^k \eta_j\lambda^j 
\right). 
\end{equation}
It is easy to see that $z_{k+1} = \lambda^{-1} (z_k + \zeta)$, where $\zeta = \xi_{k+1}
-\eta_{k+1}$ is a random variable assuming values $d_k=c_i-c_j$ with 
probabilities $q_k = \sum\limits_{i,j \colon c_i - c_j =
d_k} p_i p_j$. 
\begin{definition}
    \label{def:backproc}
We call $z_k$ the backward process associated to~$\mS(\lambda,\bar d,\bar q)$.
\end{definition}
With this notation, the symmetric diffusion operator takes the form 
\begin{equation}
    \label{eq:d2-expect}
    [\Ds_{\alpha,\mS}\psi](x) = 
    \lambda^{-\alpha} \sum_k q_k  \psi\left( g_k^{-1}(x)\right) 
    %\lambda^{-\alpha} \cdot \sum_{i,j=1}^k p_i p_j
    %\cdot \psi\left(\frac{x+c_i-c_j}{\lambda}\right) 
    = \lambda^{-\alpha} \mathbb E \left(\psi\left(\frac{x+\zeta}\lambda\right)
    \right).
\end{equation}
\noindent We conclude this preparatory discussion by commenting on the r\^ole of the
admissible interval.
\begin{lemma}
    \label{lem:return}
If a trajectory of the random process $z_n$ leaves an admissible interval~$J$ for the
$\mS(\lambda,\bar c, \bar p)$, then it never returns to it. In other words, if
there exists $k$ such that $z_k \not\in J$, then $z_n\not\in J$ for any $n>k$. 
\end{lemma}
\begin{proof}
Evidently, %Since the measure $\mu$ is compactly supported, there exists $a$ such that 
$\supp \mu*(-\mu) \subseteq \{x - y \mid x, y \in \supp \mu\}$. 
%Since $\mu$ is compactly supported, there exist 
%$$
%a\colon = \inf\{x \mid \supp \mu*(-\mu) \subseteq
%[-a,a]\}.
%$$ 
At the same time, 
$$
\frac{\max |d_k|}{1-\lambda} = \max |d_k| \sum_{k=1}^\infty \lambda^k \in \supp
\mu*(-\mu)
$$ 
and in particular $\frac{\max |d_k|}{1-\lambda} < m$, where $m$ is defined
by~\eqref{eq:suppmu-mu}. Thus $\max |d_k| < m(1-\lambda)$.  
Assume that $J = [b_1,b_2] \supsetneq [-m,m]$ is an admissible interval
and $z_k \not\in J$. Without loss of generality we may assume that $z_k > b_2 $ then
$$
|z_{k+1}| = \Bigl|\frac{z_k+d_k}{\lambda}\Bigr| \ge \frac{|z_k| -
\max|d_k|}{\lambda} \ge \frac{b_2 - \max|d_k|}\lambda > \frac{b_2 -
m(1-\lambda)}\lambda > b_2.
$$
The case $z_k< b_1$ is similar.
\end{proof}

\subsubsection{Proof of Theorem~\ref{t:certificate-D2}}
For the convenience of the reader, we recall the statement. 
\paragraph{Theorem~\ref{t:certificate-D2}}%\label{t:certificate-D2}
{\it \kern-12pt Let $\mS(\lambda,\bar c, \bar p)$ be an iterated function scheme of
similarities. Assume that for some $\alpha>0$ there exists an admissible compact interval $J \subset
\mathbb R$ and a function $\psi \in \mathcal F_J$ such that
%$\psi:\bbR\to\bbR^+$ with $\supp \psi \subset J$ which is
%positive and bounded away from~$0$ and from infinity on~$J$, such that  
\begin{equation*}
[\Ds_{\alpha,\mS}\psi] \prec \psi.
\end{equation*}
Then the correlation dimension of the $\mS$-stationary measure $\mu$ is bounded
from below by~$\alpha$:}
\[
D_2(\mu)\ge \alpha.
\]

The proof of Theorem~\ref{t:certificate-D2} relies on the following lemma which 
relates the time for which the backward process of the complementary iterated
function scheme %defined by~\eqref{eq:zkproc} 
remains in an admissible interval~$J$ to the correlation dimension of the measure~$\mu$. 
 \begin{lemma}
   \label{lem:pdim}
   Let $J = [-a,a]$ be an admissible interval for
   $\mS(\lambda,\bar c,\bar p)$ with the stationary measure~$\mu$. Let $z_n$ be the backward process for the
   complementary scheme. If $\mathbb P(z_n \in  J) \le C_0 \lambda^{\alpha n}$ for some
   constant $C_0$, independent of $z_0$, then $D_2 (\mu) \ge \alpha$. 
 \end{lemma}
 \begin{proof}
%Let us denote 
%$$
%a\colon = \inf\{x \mid \supp \mu*(-\mu) \subseteq
%[-a,a]\}.
%$$ 
Let~$m$ be as defined in~\eqref{eq:suppmu-mu}. Let $x$ and $y$ be two
independent $\mu$-distributed random variables defined by~\eqref{eq:xy}, and let
the backward process~$z_n$ be defined by~\eqref{eq:zkproc}.
Observe that the difference between $|x-y|$ and the finite sum 
$\left|\sum_{j=0}^n (\xi_j - \eta_j) \lambda^j\right|$ is no more than $m\lambda^{n+1}$.
By a straightforward calculation, we have
    % It follows from Lemma~\ref{lem:return} that
    \begin{equation*}
    \begin{split}
     \mathbb P \left(z_n \in J \right) &= \mathbb
     P\left(\lambda^{-n-1}\left|\sum_{j=0}^n (\xi_j - \eta_j) \lambda^j\right| \le  a \right) \\&
     %= \mathbb P \left(z_n \in J \right)  
     = \mathbb P\left( \left|\sum_{j=0}^n (\xi_j - \eta_j)
     \lambda^j\right| \le a  \lambda^{n+1} \right)  
     \ge \mathbb P\left( |x-y| \le (a-m) \lambda^{n+1} \right).
    \end{split}
    \end{equation*}
   Therefore the hypothesis of the Lemma implies $\mathbb P\left( |x-y| \le
   (a-m)\lambda^{n+1}\right) \le C_0 \lambda^{\alpha n}$ and thus
   for any $r$ we have that 
   \begin{equation}
     \label{eq:l6p1}
     \mathbb P( |x-y| \le r) \le C_0 \cdot (a - m)^{-\alpha} r^\alpha = \colon
     C_1(m,a,\alpha) r^\alpha.
   \end{equation}
   In order to show that $D_2 (\mu) \ge \alpha$ it is sufficient to
   show that for any $\alpha^\prime < \alpha$ the integral 
   $\int_{\bbR^2} |s-t|^{-\alpha^\prime} \mu(ds) \mu(dt)$ is finite.
   Indeed, 
   $$
   \int_{\bbR^2}  |s-t|^{-\alpha^\prime} \mu(ds) \mu(dt)
    = \mathbb E (|x-y|^{-\alpha^\prime} ) = 
    \int_{\bbR} \mathbb P(|x-y|^{-\alpha^\prime} > r ) d r.
   $$
   Evidently, $\mathbb P(|x-y|<r) \le C_1(m,a,\alpha) r^\alpha$ implies $\mathbb
   P (|x-y|^{-\alpha^\prime} > r) \le \min(1, C_2 \cdot
   r^{-\alpha/\alpha^{\prime}})$ for some constant~$C_2$, which depends on~$m$,
   $a$, and~$\alpha$ only. Hence for some constants~$C_3$ and~$C_4$, which
   depend on~$m$,~$a$, and~$\alpha$, but do not depend on~$r$ we have that 
   $$
   \int_{\bbR^2}  |s-t|^{-\alpha^\prime} \mu(ds) \mu(dt) \le
   \int_{0}^{C_3} 1 d r + \int_{C_3}^{+\infty} r^{-\alpha/\alpha^{\prime}} dr < C_4,
   $$ 
   since $\frac{\alpha}{\alpha^{\prime}}>1$. 
 \end{proof}
 \bigskip
 
Finally, we can  proceed to the proof of  Theorem~\ref{t:certificate-D2}. We use the same  notation as above.  

\smallskip

% \begin{proof}
%\paragraph{Proof of Theorem \ref{t:certificate-D2}.}
  \begin{proof}[of Theorem~\ref{t:certificate-D2}]
   By the hypothesis of the Theorem there exists $\theta>0$ such that
   for any $x \in J$ we have that $\theta < \psi(x) <\theta^{-1}$.

   Consider a discrete random process defined by~$w_n = \lambda^{-\alpha n}
   \psi(z_n)$. Then for any $z_n \in  J$ taking into
   account~\eqref{eq:d2-expect}, we compute
   % $\mathbb E (w_{n+1} \mid w_n )
  % \le w_n$. Indeed
   \begin{equation}
   \begin{split}
     %\mathbb E (w_{n+1} \mid w_n ) &= 
     \mathbb E (w_{n+1} \mid z_n ) & = \mathbb E
     (\lambda^{-\alpha (n+1)} \psi(z_{n+1}) \mid z_n) = \lambda^{-\alpha(n+1)} \mathbb
     E \left(\psi \left(\lambda^{-1} (z_n+\zeta) \right) \mid z_n  \right) \\ &=  
     \lambda^{-\alpha (n+1)} \mathbb E \left(\psi \left( \lambda^{-1}(
     z_n+\zeta) \right)\right) = \lambda^{-\alpha n} \Ds_{\alpha,\mS} \psi
     ( z_n ) \\& \le \lambda^{-\alpha n} \psi (z_n) = w_n. 
   \end{split}
    \label{eq:wnmart}
   \end{equation}
  Thus the process $w_n$ is a supermartingale, as $\mathbb E(w_{n+1} \mid z_n ) \le  w_n$. In particular, 
  \[
  \mathbb E w_n\le w_0=\psi(z_0)=\psi(0)<\theta^{-1}.
  \]
  On the other hand, 
  \[
    \mathbb E w_n = \lambda^{-\alpha
  n} \cdot \mathbb E \psi (z_n) \ge \lambda^{-\alpha n} \cdot \inf \psi |_{J} \cdot
  \mathbb P(z_n \in J) \ge \lambda^{-\alpha n} \cdot \theta \cdot \mathbb P(z_n \in J). 
  \]
%  We  now want to define a stopping time $T = \min\{n \mid z_n \not \in
%   J\}$.
  %Observe that by definition of $z_n$ we have that $T$ is finite for any~$z_0$. 
%  Then the process $\tilde w_n = w_{\min(T,n)}$ is a supermartingale
%  by~\eqref{eq:wnmart}, since $\mathbb E(\tilde w_{n+1} \mid \tilde w_n ) \le \tilde w_n$. 
%  In particular, we have that $\mathbb E \tilde w_n \le \tilde w_{0} = \psi(z_0)<
%  \theta^{-1}$ is finite.  On the other hand, 
%  $$
%  \mathbb E \tilde w_n \ge \mathbb E w_n = \lambda^{-\alpha
%  n} \cdot \mathbb E \psi (z_n) \ge \lambda^{-\alpha n} \cdot \inf \psi |_{J} \cdot
%  \mathbb P(z_n \in J) \ge \lambda^{-\alpha n} \cdot \theta \cdot \mathbb P(z_n \in J). 
%  $$ 
  Therefore $\mathbb P (z_n \in J) \le \theta^{-2} \lambda^{\alpha n}$
  and the Theorem follows from Lemma~\ref{lem:pdim}.
%  \qed
\end{proof}

%%%%%%%%%%%%%%%%%%%%%%%%Theorem 4
\subsection{Effectiveness of the algorithm}
Let $\mu$ be the stationary measure of an iterated function scheme of
similarities $\mS(\lambda,\bar c, \bar p)$. In this section we shall show that for any $\alpha < D_2(\mu)$ 
the method described in Section~\ref{sss:verifyalpha} will be able to
confirm this inequality, subject to computer resources and time.
In other words we shall show the following. 
\begin{proposition}
    \label{prop:D21}
Let $\mu$ be the stationary measure of an iterated function scheme of
similarities $\mS(\lambda,\bar c, \bar p)$. Assume that $\alpha < D_2(\mu)$. 
Then there exist:
\begin{enumerate}
    \item A sufficiently small $\varepsilon>0$ and an admissible interval
        $J_\Lambda$ for $\Lambda = B_\varepsilon(\lambda)$; 
    \item A sufficiently fine partition $\mathcal J$ of $J_\Lambda$;
    \item A sufficiently large $n \in \mathbb N$; and
    \item A sufficiently small $\vartheta >0$,
\end{enumerate}
so that the hypothesis of Corollary~\ref{cor:uniform} holds, more precisely, for $A
=\mathcal D_{\alpha,\Lambda,\mathcal J}$ we have  
$$
\hatop A^n \mathds{1}_{\mathcal J} \prec
\mathds{1}_{\mathcal J}. 
$$
\end{proposition}
Theorem~\ref{t:finding-D2} follows immeditately from Proposition~\ref{prop:D21}.
We begin with the following technical fact. 
\begin{lemma}
    \label{lem:find-D2}
 Let $\mS(\lambda,\bar c,\bar p)$ be an iterated function scheme.
 Let $\mu$ be the unique stationary measure and assume that
 $\alpha<D_2(\mu)$.      
 Then for any sufficiently small $\varepsilon>0$ there exist an admissible interval~$J$,
 a continuous function~$\varphi$, $\varphi|_J > \theta > 0$, and $n$ such that
 $(\Ds_{\alpha,\mS})^n \varphi\prec(1-\eps)\varphi$.
\end{lemma}
\begin{proof}
    Since $\alpha<D_2(\mu)$, then $I(\alpha,\mu)$ is finite and by
    Lemma~\ref{lem:psi} the function 
    $$
    \psi_{\alpha,\mu}(r) = \int\int |x-y-r|^{-\alpha} d\mu(x)d\mu(y) 
    $$
    is finite for all~$r \in \mathbb R$. Moreover, by
    Corollary~\ref{cor:fixedpoint} it is the fixed point of
    the symmetric diffusion operator~$\Ds_{\alpha,\mS}$.

    Let $\mS(\lambda,\bar d, \bar q)$ be the complentary iterated function
    scheme of $N$ similarities. Let $m$ be as defined in~\eqref{eq:suppmu-mu}, so that $\supp
    \mu*(-\mu) \subseteq [-m,m]$. Let us choose an admissible interval $J:=[-a,a]$
    and consider the intersection of its preimages under $g_j = \lambda x +d_j$ 
    $$
    \tilde J: = \bigcap_{j=1}^N g_j^{-1}([-a,a]) = [- \tilde a, \tilde a]
    \supsetneq [-a,a] \supsetneq [-m,m].  
    $$
   Then for any $x\notin J$ and any $j=1,\dots,N$ one has
   $f_j^{-1}(x)\notin \tilde J$. 
    
    We define a continuous function~$\varphi$ by 
    \begin{equation}
        \label{eq:psi-cut}
        \varphi(r) = 
        \begin{cases}
            \psi_{\alpha,\mu}(r), & \mbox{ if } |r| \le a , \\
            \psi_{\alpha,\mu}(r) \cdot \frac{\tilde a-r}{\tilde a-a}, & \mbox{ if }
            a <r< \tilde a, \\
            \psi_{\alpha,\mu}(r) \cdot \frac{\tilde a+r}{\tilde a- a}, & \mbox{ if }
            -\tilde a <r< - a, \\
            0, & \mbox{ otherwise.}
        \end{cases}
    \end{equation}
    \begin{figure} 
        \centering
        \includegraphics[height=55mm,width=120mm]{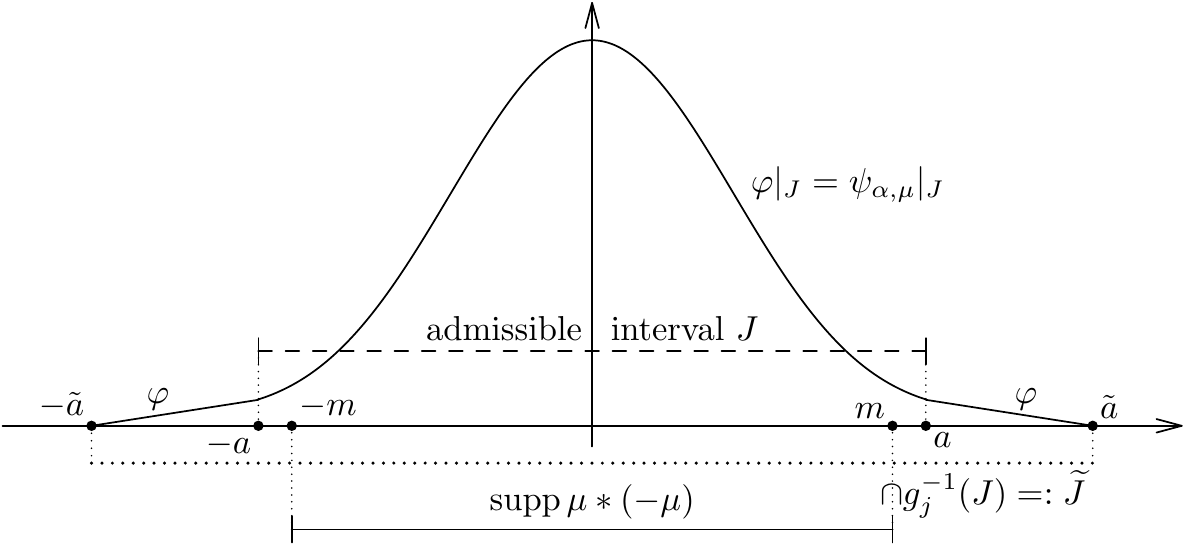}
        \caption{The construction of function $\varphi$ in Lemma~\ref{lem:find-D2}.
        It agrees with $\psi_{\alpha,\mu}$ on the admissible interval $J$, is
        linear on $\widetilde J \setminus J$ and vanishes outside of $\widetilde
        J$.} 
        %an admissible $J$ and intersection of its preimages $\widetilde J$ }
        \label{fig:funphi}
    \end{figure}

    It is easy to see that $\varphi(r) \le
    \psi_{\alpha,\mu}(r)$ for all $r \in \mathbb R$.
    Taking into account monotonicity of the diffusion
    operator~\eqref{eq:d2-back} we obtain for any $r\in J$
    \begin{equation}\label{eq:Ds-phi}
    [\Ds_{\alpha,\mS} \varphi](r) \le [\Ds_{\alpha,\mS} \psi_{\alpha,\mu}](r) = \psi_{\alpha,\mu}(r) =\varphi(r). 
    \end{equation}
     On the other hand, for any $r\in (-\tilde a, \tilde a) \setminus J$ one has
     $[\Ds_{\alpha,\mS} \varphi](r) =0<\varphi(r)$, where the equality
     is due to Lemma~\ref{lem:return} and formula~\eqref{eq:d2-expect}
     for~$\Ds_{\alpha,\mS}$.  Together these
     two observations give $\Ds_{\alpha,\mS} \varphi \preccurlyeq
     \varphi$.
    
    We shall now show that for sufficiently large~$n$ we have $
    [(\Ds_{\alpha,\mS})^n \varphi](r)<\varphi(r) $ for all $r\in J$. 
    Indeed, let $d_k = \max|c_i - c_j|$. Then for any $x \in [0,a]$ we have that 
    $$
    g_k^{-n} (x) \ge d_k \sum_{j=1}^n
    \lambda^{-k} \ge d_k \lambda^{-n} \to \infty \mbox{ as } n \to \infty,
    $$
    and the case $x \in [-a,0]$ is similar. 
    
    Therefore we may choose~$n$ such that for any $x \in J$ there exists
    a sequence $\underline j_n$ such that $g_{\underline j_n}^{-1}(x) \not\in J$. 
    We may write for any $r\in J$
%    We may apply Proposition~\ref{prop:dspsi} to write for any $r \in J$
    \begin{equation}
    \begin{split}
        [(\Ds_{\alpha,\mS})^n\varphi](r) &= \sum_{\underline j_n} q_{j_1} \ldots
        q_{j_n} \varphi\left(g_{\underline j_n}^{-1} (r)\right)
        \\ &=
        \sum_{\underline j_n \, : \,  g_{\underline j_n}^{-1}(r) \in J} q_{j_1} \ldots
        q_{j_n} \psi_{\alpha,\mu} \left(g_{\underline j_n}^{-1} (r)\right) +
        \sum_{\underline j_n \, : \, g_{\underline j_n}^{-1} \not\in J} q_{j_1} \ldots
        q_{j_n} \varphi\left(g_{\underline j_n}^{-1} (r)\right) %\notag
        \\ &<
        \sum_{\underline j_n \, : \,  g_{\underline j_n}^{-1}(r) \in J} q_{j_1} \ldots
        q_{j_n} \psi_{\alpha,\mu} \left(g_{\underline j_n}^{-1} (r)\right) +
        \sum_{\underline j_n \, : \, g_{\underline j_n}^{-1} \not\in J} q_{j_1} \ldots
        q_{j_n} \psi_{\alpha,\mu} \left(g_{\underline j_n}^{-1} (r)\right) %\notag 
        \\ &= \psi_{\alpha,\mu}(r) = \varphi(r), %\notag
    \end{split}
    \label{eq:Ds-itn}
    \end{equation}
   where the equality in the second line comes from the fact that
   $\varphi|_J=\psi_{\alpha,\mu}|_J$, while the inequality in the third is due
   to the strict inequality between the corresponding terms of the sums over
    the set $\{\underline j_n : g_{\underline j_n}^{-1} \not\in J\}$. 
Note that $(\Ds_{\alpha,\mS})^n\varphi$ is a continuous function, and a strict inequality 
\[
[(\Ds_{\alpha,\mS})^n\varphi](r) < \varphi(r)
\]
for all~$r \in J$ implies, taking into account compactness of $J$, that for some
$\varepsilon > 0$ one has
%for all $r$ from a closed interval $J$ implies that for some $\eps>0$ one has  
   \[
   (\Ds_{\alpha,\mS})^n\varphi \prec (1-\varepsilon) \varphi.
   \] 
\end{proof}

We now proceed to prove Proposition~\ref{prop:D21}. 

\medskip

\begin{proof}[of Proposition~\ref{prop:D21}] 
Let an admissible interval $J$ be fixed. By Lemma~\ref{lem:find-D2} we know that
there exist $n$, $\eps$, $\theta>0$ and a function $\varphi(r)>\theta$ such
that $\Ds_{\alpha,\mS} \varphi \preccurlyeq \varphi$ and for all $r\in J$ 
   \[
   [(\Ds_{\alpha,\mS})^n\varphi](r) < (1-\varepsilon) \varphi(r).
   %, \quad    \text{ and } \, \varphi(r)>\theta.
   \] 
Note that for $c=\frac{1}{\theta}$ we have $\mathds{1}_J\preccurlyeq c \varphi$, and
in particular for every~$m$ we get
\[
(\Ds_{\alpha,\mS})^{mn}(\mathds{1}_J) \preccurlyeq (\Ds_{\alpha,\mS})^{mn}(c \varphi) = c (\Ds_{\alpha,\mS})^{mn}(\varphi) \preccurlyeq c (1-\eps)^m \varphi.
\]
Then for~$m$ sufficiently large so that $c(1-\eps)^m \cdot \max_J \varphi<\frac
12$, we obtain a strict inequality
\[
(\Ds_{\alpha,\mS})^{mn}(\mathds{1}_J)\prec \frac{1}{2} \mathds{1}_J.
\]
Let us fix $n':=nm$. Note that~$\varphi$ and its images under $\Ds_{\alpha,\mS}$
are continuous. The finite rank operator $\mathcal D_{\alpha,\Lambda, \mathcal
J}^{n^\prime}$ depends continuously on the partition $\mathcal J$ and $\Lambda$, therefore
\[
\sup_J  \left| (\mathcal D_{\alpha,\Lambda, \mathcal J})^{n'} \varphi - (\Ds_{\alpha,\mS})^{n'}  \varphi \right|
 \to 0 \mbox{ as } \varepsilon \to 0 \mbox{ and } M \to \infty. 
\]
In particular, for all sufficiently small $\eps$ and sufficiently large $M$ one
has for all $x \in J$
\[
(\mathcal D_{\alpha,\Lambda, \mathcal J})^{n'}(c \varphi)(x) < \frac{1}{2},
\]
since the operator $\mathcal D_{\alpha,\Lambda, \mathcal J}$ is monotone, the
latter implies
\[
(\mathcal D_{\alpha,\Lambda, \mathcal J})^{n'}(\mathds{1}_J) \prec \frac{1}{2}\mathds{1}_J.
\]
Now, let us denote $A=D_{\alpha,\Lambda, \mathcal J}$; then for any nonnegative
$\psi$ we get $\widehat{A}_0 \psi \preccurlyeq A \psi$, and hence 
\[
\widehat{A}_0^{n'} (c\varphi) \preccurlyeq A^{n'} (c\varphi) \prec \frac{1}{2}\mathds{1}_J.
\]
On the other hand, $(\hatop A)^{n'} (c\varphi)$ converges to $\widehat{A}_0^{n'} (c\varphi)$ uniformly as $\vartheta\to 0$. 
 %Then, the same holds for $\hatop A$, where $A=D_{\alpha,\Lambda, \mathcal J}$: 
Hence for all sufficiently small $\vartheta$ we get 
\[
(\hatop A)^{n'}(c \varphi) < \frac{1}{2}
\]
everywhere on $J$. Finally we conclude 
\[
(\hatop A)^{n'}(\mathds{1}_J) \prec \frac{1}{2}\mathds{1}_J.
\]

 \end{proof}

 \subsubsection{Proof of Theorem~\ref{t:finding-D2}}
 Let $\mS(\lambda)$ be an interated function scheme of similarities and let $\mu$ be its
 unique stationary measure. Assume that $\alpha < D_2(\mu)$. Then by
 Proposition~\ref{prop:D21} there exist an interval $\Lambda\ni\lambda$, an
 admissible interval $J_\Lambda$, 
 and its partition~$\mathcal J$ such that for the finite rank diffusion operator
 $A = \mathcal D_{\alpha,\Lambda,\mathcal J}$ we have that $\varphi:=\hatop{A}^n \mathds{1}_J \prec \mathds{1}_J$.  
 Then by Proposition~\ref{prop:auxop} we have  $\mathcal D_{\alpha,\Lambda,\mathcal J} \varphi \prec \varphi$
 in other words, that the function $\varphi$ satisfies the hypothesis of
 Theorem~\ref{thm:5star}.

\section[Diffusion operator $\Da_{\alpha,\lambda}$]{Diffusion operator $\Da_{\alpha,\lambda}$ and regularity of the measure}

In this section we consider general iterated function schemes of
orientation-preserving contracting $C^{1+\varepsilon}$ diffeomorphims and show
that the diffusion operator approach can be used to get a lower bound on the
regularity exponent of the stationary measure. 

We briefly recall the setting.
Let $J \subset \mathbb R$ be a compact interval. Consider  an interated function
scheme~$\mathcal T(\bar f, \bar p, J)$ consisting
of~$n$ uniformly contracting diffeomorphisms $f_j \colon \mathbb R \to \mathbb R$, $f_j \in
C^{1+\varepsilon}(\mathbb R)$ which preserve the interval $J$: $f_j (J) \subset
J$ for $j = 1, \ldots, n$ and probability vector $\bar p = (p_1, \ldots, p_n)$. 
Let $\mu$ be the stationary measure so that $
\sum_{j=1}^n p_j {f_j}_* \mu = \mu $; evidently, $\supp \mu \subset J$. 

The asymmetric diffusion operator is defined by~\eqref{eq:dif-oneway}. 
\begin{equation*}
\Da_{\alpha,\mathcal T}[\psi](x) := \sum_{j=1}^k p_j \cdot |(f_{j}^{-1})'(x)|^{\alpha}  \cdot \psi(f_j^{-1}(x)).
%\label{eq:dif-oneway}
\end{equation*}

\begin{example}[Bernoulli convolution revisited] 
In the special case of Bernoulli convolution scheme~$\mS$ as defined in
Example~\ref{ex:BCsystem} and $\alpha=1$ the operator
$\Da_{1,\mS} \colon L^1(\mathbb R) \to L^1(\mathbb R)$ has a fixed point
$\Da_{1,\mS} h = h$ precisely when $\mu$ has an $L^1$ density, i.e., $\frac{d\mu}{dx} = h \in
L^1(\mathbb R)$. %\cite{kempton} 
\end{example}
 
We will need the following technical fact for the proof of
Theorem~\ref{t:certificate-D1}.

\begin{lemma}
    \label{lem:dist}
    Let $\mathcal T(\bar f, \bar p)$ be an iterated function scheme of uniformly
    contracting
    $C^{1+\varepsilon}$-diffeomorphims which preserve a compact
    interval~$J$. Then the distortion is uniformly bounded. In other words,
    there exist two constants $c_1$, $c_2$ such that for any sequence
    $\underline j_n$ we have for the distortion of the composition $f_{\underline j_n} =
    f_{j_n} \circ \ldots \circ f_{j_1}$ that for all $ x,y \in J$
    $$
    e^{c_1} < \frac{f_{\underline j_n}^\prime
    (x)}{f_{\underline j_n}^\prime (y)}
    < e^{c_2} 
    $$
\end{lemma}
\begin{proof}
The argument generalises the classical argument for a single function, which can be found, in
particular in~\cite[\S 3.2]{MS93}. More precisely, we define the distortion of~$f$
on the interval $J$ by 
$$
\varkappa(f,J) := \max_J \log f^\prime - \min_J \log f^\prime. 
$$
It is easy to see that it is subadditive with respect to composition, in
particular, for any $f$, $g$ we have 
$$
\varkappa(f\circ g, J) \le \varkappa (g,J) + \varkappa(f, g(J)).
$$
Since by assumption $f_j$ are uniformly contracting $C^{1+\varepsilon}$
diffeomorphisms, there exist constants $C = C(\mathcal T)$ such that
$\varkappa(f_j,J) \le C \cdot |J|^\varepsilon$ and $\lambda<1$ such that
$f_j^\prime(x) < \lambda$ for all $j = 1, \ldots n$ and for any $x \in J$. 
Therefore 
\begin{multline*}
\varkappa(f_{\underline j_n}, J) \le \sum_{k=1}^n \varkappa(f_{j_k}, f_{j_{k-1}}\circ
\ldots \circ f_{j_1}(J)) \le C \sum_{k=1}^n |f_{j_{k-1}}\circ
\ldots \circ f_{j_1}(J)
|^\varepsilon 
\\ \le C \sum_{k=0}^{n-1} \lambda^{k\varepsilon} |J|^{\eps}
\le \frac{ C
|J|^\varepsilon}{1-\lambda^\varepsilon},
\end{multline*}
and the result follows.
\end{proof}

\subsection{Proof of Theorem~\ref{t:certificate-D1} }
For the convenience of the reader, we recall the statement. 
\paragraph{Theorem~\ref{t:certificate-D1}.}
\emph{Assume that for some $\alpha>0$ there exists a function $\psi \in \mathcal
F_{B_r(J)} $   
%:\bbR\to\bbR_+$, supported on~$B_r(J)$, positive on $B_r(J)$, and bounded away from $0$ and
%from infinity on $B_r(J)$, 
such that for any $x \in B_r(J)$ we have  that
\begin{equation*}
 [\Da_{\alpha,\mathcal T}\psi](x) < \psi(x).
\end{equation*}
Then the measure~$\mu$ is $\alpha$--regular. }
%\end{theorem}
\begin{proof}
  Let $I = B_\delta(x) \subset J$ be an interval of length~$|I|=2\delta$. We
  shall show that there exists $c \in \mathbb R$ such that
  %there exists a constant~$C$ so that
  \begin{equation}
      \label{eq:d1-proof1}
       \mu(I) \le c \cdot \delta^\alpha. 
  \end{equation}
    
  We will use a random processes approach as in the previous section. 
  Let us first consider a random process
  \begin{equation}
    \label{eq:past}
    F_{\underline \omega_n}(x) = f_{\omega_1}\circ \ldots \circ f_{\omega_n}(x),
  \end{equation}
  where $\omega_j$ are the i.i.d. distributed with $\mathbb P(\omega_j = j) = p_j$. 
  Using the induced action on the stationary measure~$\mu$ we define another random
  process by
  \begin{equation}
      \label{eq:xi}
      \xi_n(\omega) = \mu\left(F_{\underline \omega_n}^{-1}(I)\right).
  \end{equation}
  It follows from the invariance of~$\mu$ that the process~$\xi$ is a
  martingale. Indeed, % by straightforward calculation, 
  \begin{equation*}
      \mathbb E \xi_n (\omega) = \sum_{\underline j_n} p_{j_1} \ldots p_{j_n}
      \mu(f_{\underline j_n}^{-1}(I)) = \mu(I) = \xi_0. 
  \end{equation*}
  %and let $r$ to be the distance $r = \dist(\mathbb R
  %\setminus J,\supp \mu) > 0$. 
  By assumption the maps are uniformly contracting, therefore their inverses are
  uniformly expanding and by compactness the derivatives are bounded on
  $B_r(J)$. Let us denote by $\beta_{min} $ and $\beta_{max} $ the lower and the upper
  bound, respectively:
  $$
  1 < \beta_{min}:= \inf_{\stackrel{x \in B_r(J)}{1 \le j \le n}} (f_j^{-1})^\prime(x )
   < \sup_{\stackrel{x \in B_r(J)}{1 \le j \le n}} (f_j^{-1})^\prime(x) =:
   \beta_{max}
  $$ 
  and consider the stopping time
  \begin{equation}
      \label{eq:stopt}
      T(\omega) = \min\left\{ n \mid F_{\underline \omega_n}^{-1}(x) \not\in B_r(J) \mbox{ or }
      \left|F_{\underline \omega_n}^{-1}(I)\right| > \frac{r}{\beta_{max}} \right\}. 
  \end{equation}
  Since the diffeomorphisms~$f_j$ are contracting, we have for any $\omega$ 
  $$
  T(\omega) < \frac{\log(2\beta_{max} \varepsilon )}{\log \beta_{min}} + 1.
  $$
  Consider the backward random process associated with the inverses of
  diffeomorphisms~$f_j$
  \begin{equation}
      \label{eq:eta}
      \eta_n (\omega) = \left( ( F_{\underline \omega_n}^{-1})^\prime (x)\right)^\alpha \cdot \psi
      (F_{\underline \omega_n}^{-1}(x)) .
  %= \left(\left(f_{\omega_n}^{-1} \circ \ldots \circ
  %f_{\omega_1}^{-1}\right)^{\prime} \right)^\alpha(x) \cdot \psi\left(f_{\omega_n}^{-1} \circ \ldots \circ
  %f_{\omega_1}^{-1}\right).
  \end{equation}
  We claim that $\eta$ is a supermartingale. Indeed, by assumption 
  $$
  [\Da_{\alpha,\mathcal T} \psi] (F_{\underline \omega_n}^{-1}(x) ) \le \psi (F_{\underline \omega_n}^{-1}(x) )
  $$
  therefore 
  \begin{align*}
  \mathbb E (\eta_{n+1} & \mid \omega_1 \ldots \omega_n) = \\ & =  \left( F_{\underline \omega_n} ^{-1} (x)
  \right) ^\alpha \cdot\sum_{j=1}^n  p_{j_{n+1}} \left( (f^{-1}_{j_{n+1}})^\prime  
   \left(  F_{\underline \omega_n} ^{-1}  (x)\right) \right)^\alpha 
  \cdot \psi \left(f_{j_{n+1}}^{-1}( F_{\underline \omega_n}^{-1}(x) ) \right) \\ 
  &=  \left( F_{\underline \omega_n} ^{-1} (x) \right)^\alpha \cdot
  [\Da_{\alpha,\mathcal T} \psi]
  (F_{\underline \omega_n}^{-1}(x) ) \\
  &\le \left( F_{\underline \omega_n} ^{-1} (x) \right)^\alpha \psi(F_{\underline \omega_n}^{-1}(x) ) =
  \eta_n .
  \end{align*} 
  In particular, since $T(\omega)$ is finite, 
  \begin{equation}
      \label{eq:expeta}
  \mathbb E \eta_{T(\omega)} \le \eta_0 = \psi (x).
  \end{equation} 
 
  We next want to consider the expectations $\mathbb E \xi_{T(\omega)}$ and
  $\mathbb E \eta_{T(\omega)}$. 
  By definition~\eqref{eq:stopt}
of $T(\omega)$  we have that at least one of the following events takes place: 
  \begin{align*}
      A :& = \left[F_{\underline \omega_T}^{-1}(x) \not\in B_r(J)\right] \\ % \label{case a} \tag{Case a}\\
      B :& = \left[ F_{\underline \omega_T}^{-1}(x) \in B_r(J)  \mbox{ and }
      \left|F_{\underline \omega_T}^{-1}(I)\right| > \frac{r}{\beta_{max}}
      \right] .% \label{case  b} \tag{Case b}   
  \end{align*}
  We claim that if~$B$ doesn't occur, then~$\eta_{T(\omega)}=0$ and 
  $ \xi_{T(\omega)}=0$. Indeed, the first follows from~\eqref{eq:eta} and the fact that $\supp \psi
  \subset B_r(J)$. For the second, note that, by definition~\eqref{eq:stopt} of $T(\omega)$, in this case we have
  $F_{\underline \omega_{T-1}}^{-1}(x) \in J$ and $\left|F_{\underline
  \omega_{T-1}}^{-1}(I)\right| <
  \frac{r}{\beta_{max}} $. Therefore 
  $$
  \left| F_{\underline \omega_T}^{-1}(I)\right| \le \beta_{max} \cdot \frac{r}{\beta_{max}}
  = r, 
  $$
  hence $F_{\underline \omega_T}^{-1}(I) \cap J = \varnothing$, and 
  $\xi_{T(\omega)} = \mu( F_{\underline \omega_T}^{-1}(I))=0$.   

  Now assume that~$B$ occurs. Then $\left|
  F_{\underline \omega_T}^{-1}(I)\right| \le r$ and therefore $F_{\underline
  \omega_T}^{-1}(I)\subset B_{2r}(J)$. By Lemma~\ref{lem:dist} applied to the
  interval $B_{2r}(J)$, we see that there exist~$c_1$ and~$c_2$ such that for
  any $y\in I$ 
  $$
      e^{c_1} \le
      \frac{(F_{\underline \omega_T}^{-1})^\prime(x)}{(F_{\underline \omega_T}^{-1})^\prime(y)} \le
      e^{c_2}. 
  $$
 Indeed, if we denote $\bar x:=(F_{\underline \omega_T}^{-1})(x), \bar y:=(F_{\underline \omega_T}^{-1})(y)$; then 
  $\bar x,\bar y\in B_{2r}(J)$ and 
  \[
  \frac{(F_{\underline \omega_T}^{-1})^\prime(x)}{(F_{\underline \omega_T}^{-1})^\prime(y)} = \frac{(F_{\underline \omega_T})^\prime(\bar y)}{(F_{\underline \omega_T})^\prime(\bar x)}.
  \]
      %\frac{|F_{\underline \omega_T}(I)|}{|I|} \le
  In particular, %for some constants $C_1$, $C_2$ we have that  
  \begin{equation} 
      \label{eq:difbound}
    (F_{\underline \omega_T}^{-1})^\prime(x) \ge e^{c_1} \frac{|F_{\underline \omega_T}^{-1} (I)|}{|I|}
    \ge \frac{e^{c_1} r}{\delta \cdot \beta_{max}}.
  \end{equation}  
 % $$
 % \frac{C_2}{\varepsilon} \le e^{-\const} \frac{|F_{\underline \omega_T}^{-1} (I)|}{|I|} \le
 % (F_{\underline \omega_T}^{-1})^\prime(x) \le e^{-\const} \frac{|F_{\underline \omega_T}^{-1}(I)|}{|I|} \le
 % \frac{C_1}{\varepsilon}
 % $$
  We have an upper bound for $\mu(I) = \mu(B_\delta(x))$: 
  \begin{equation}
    \label{eq:01}
    \mu(I) = \mathbb E \xi_{T(\omega)} = 
    \mathbb E \mu(F_{\underline \omega_T}^{-1}(I)) =
    \mathbb E \left(\mu(F_{\underline \omega_T}^{-1} (I)) \cdot
    \mathds{1\!}_B \right) \le \mathbb P(B), 
    %\\ \le 
    % \mathbb E \xi_{T(\omega)} \mathbb P\left( |F_{\underline \omega_T}(I)| \ge \frac{r}{\beta_{max}}, \,
    %F_{\underline \omega_T}^{-1}(x) \in J \right).
  \end{equation}
  since $\mu(F_{\underline \omega_T}^{-1}(I)) \le 1$ and if $B$ doesn't take place, then $\mu(F_{\underline \omega_T}^{-1} (I)) =0$. 
  
  By the hypothesis of the Theorem, there exists $c_3>0$ such that for any $x
  \in B_r(J)$ we have that $ \frac1{c_3} < \psi(x) < c_3$.
  Therefore, using~\eqref{eq:expeta} and~\eqref{eq:difbound}
  \begin{equation}
      c_3 \ge \psi(x) = \eta_0 \ge \mathbb E \eta_{T(\omega)} = \mathbb E
      \left( \left( (F_{\underline \omega_T}^{-1} )^\prime (x) \right)^\alpha\cdot \psi
      (F_{\underline \omega_T}^{-1}(x) ) \right) \ge
      \frac1{c_3} \cdot \Bigl(\frac{e^{c_1} r}{\delta \cdot \beta_{max}}\Bigr)^\alpha \cdot
    \mathbb P(B). 
 %   \Bigl(F_{\underline \omega_T}^{-1}(I)| \ge \frac{r}{C_f}, \, F_{\underline \omega_T}^\prime(x) \in
 %   J\Bigr ). 
  \end{equation}
  In particular, we obtain an upper bound $\mathbb P(B) \le c_4 \cdot
  \delta^\alpha$ and the desired estimate~\eqref{eq:d1-proof1} follows from~\eqref{eq:01}.
\end{proof}

\subsection{Proof of Theorem~\ref{t:finding-D1} }
\newcommand{\ap}{\bar{\alpha}}
\newcommand{\Prob}{\mathbb P}

\newcommand{\tx}{\widetilde{x}}
\newcommand{\ty}{\widetilde{y}}

Theorem~\ref{t:finding-D1} follows immediately from Proposition~\ref{prop:auxop}
and Proposition~\ref{prop:D11} below.

We begin with the following lemma, which is analogous to
Lemma~\ref{lem:find-D2}. However, in the case of general iterated function
scheme~$\mathcal T(\bar f, \bar p, J)$ we don't know the
eigenfunction of $\Da_{\alpha,\mathcal T}$. 
 
\begin{lemma}
    \label{lem:find-D1}
    Let~$\mu$ be the stationary measure of $\mathcal T(\bar f,\bar p, J)$.
    Assume that $\alpha < D_1(\mu)$. Then there exists $r>0$ such that
    $$
    \lim_{n\to \infty} \bigl\|(\Da_{\alpha,\mathcal T})^n  \mathds{1\!}_{B_r(J)}
          \bigr\|_\infty = 0.
          $$
  \end{lemma} 

\begin{proof} 
    Let us introduce a shorthand notation $I := B_r(J)$ and
    \begin{equation*}
    \psi_n(x) := (\Da_{\alpha,\mathcal T})^n \mathds{1}_I.
    \end{equation*}
    Since $\Da_{\alpha,\mathcal T}$ preserves non-negative functions, it is
    sufficient to show that 
    \begin{equation}
        \label{eq:psin}
    \psi_n(x) \to 0 \mbox{ as } n \to \infty \mbox{ uniformly in } x.
    \end{equation}
    %$(\Da_{\alpha,\mathcal     T})^n[\mathds{1}_{B_r(J)}]$. 
 Let $x \in I$ be fixed.   We may rewrite~$\psi_n$ using the definition of the assymetric
    operator~\eqref{eq:dif-oneway} as follows
    \begin{equation}\label{eq:D-ap}
        \psi_n (x) = \sum_{\underline j_n} p_{j_1} \ldots p_{j_n} 
        ((F_{\underline j_n}^{-1})^\prime (x))^{\alpha} \mathds{1}_I
        (F_{\underline j_n}^{-1}(x)), 
    \end{equation}
    where $F_{\underline j_n}=f_{j_1}\circ \dots \circ f_{j_n}$. % and $p_=p_{w_1}\dots p_{w_n}$.

    By Lemma~\ref{lem:dist} there exists $c_1$ such that for any word
    $\underline j_n$ for any $y_1, y_2 \in I$ 
    %with $F_{\underline
    %j_n}^{-1}(y_1) \in I$ and $F_{\underline j_n}^{-1}(y_2) \in I$ we have that
    \[
    e^{-c_1}<\frac{F_{\underline j_n}^\prime(y_1)}{F_{\underline j_n}^\prime (y_2)}<e^{c_1}
    \]
 in particular, there exists $c_2>0$ such that for any $y \in I$ we have
\begin{equation}\label{eq:F-lower}
  %\forall y\in I \quad   
  F_{\underline j_n}^\prime (y) \ge c_2 |F_{\underline j_n} (I)| .
\end{equation}

Let us collect nonzero terms from the right hand side of~\eqref{eq:D-ap} by
length of $F_{\underline j_n}(I)$. More precisely, given $\delta>0$ consider the set
of words 
\[
R_{\delta} := \{ \underline j_n \mid x \in F_{\underline j_n}(I), \quad \delta <
|F_{\underline j_n} (I)|\le 2\delta \}.
\]
Note that since~$\mu$ is stationary, $\mu = \sum_{\underline j_n} p_{\underline
j_n} (F_{\underline j_n})_* \mu$, and $\supp ((F_{\underline j_n})_*
\mu)\subset F_{\underline j_n} (I)$, we have for any word  $\underline j_n \in
R_{\delta}$ that $ \supp (F_{\underline j_n})_* \mu \subset B_{2\delta}(x)$.
Hence
$$
\mu(B_{2\delta}(x)) \ge \sum_{ \underline j_n \in R_\delta} p_{\underline j_n}
 ((F_{\underline j_n})_* \mu)(B_{2\delta}(x)) = \sum_{\underline j_n \in
R_\delta} p_{\underline j_n} \cdot 1 = \Prob
(R_{\delta}). 
$$
By assumption $\alpha < D_1(\mu)$. Then there exists $ \alpha < \bar \alpha < D_1(\mu) $ such that 
\begin{equation}
    \label{eq:probup}
\Prob (R_{\delta}) \le \mu(U_{2\delta}(x)) \le \const \cdot
(2\delta)^{\bar \alpha}.
\end{equation}

Note that a term of the sum~\eqref{eq:D-ap} corresponding to a given $\underline j_n$
is nonzero only if $F_{\underline j_n}^{-1}(x)\in I$ or, equivalently, if $x\in
F_{\underline j_n}(I)$. It follows from~\eqref{eq:F-lower} that for any $\underline
j_n \in R_\delta$ we have 
\begin{equation}
    \label{eq:fprimeup}
(F_{\underline j_n}^{-1})'(x) = \frac{1}{F_{\underline j_n}'(y)} \le
\frac{1}{c_2\delta}. 
\end{equation}
Thus, combining~\eqref{eq:probup} with~\eqref{eq:fprimeup} we get an upper bound for a part of the sum from~\eqref{eq:D-ap},
which corresponds to~$\underline j_n \in R_{\delta}$
\begin{equation}
    \label{eq:sumrd}
    \sum_{\underline j_n \in R_{\delta}} p_{\underline j_n} ((F_{\underline
    j_n}^{-1})'(x))^{\alpha} \mathds{1}_I(F_{\underline j_n}^{-1}(x))
\le \Prob(R_{\delta}) \cdot \frac{1}{(c_2 \delta)^{\alpha}} 
%\le \frac{C\cdot (2\delta)^a}{c_2^{\alpha} \cdot \delta^{\alpha}} =
\le c_5 \cdot \delta^{\bar\alpha-\alpha}.
\end{equation}
Let us denote $\delta_n:=\max_{\underline j_n} F_{\underline j_n}(I)$.
Since by assumption $f_j$ are uniformly contracting, $\delta_n \to 0 $ as $n \to \infty$ exponentially fast. 
%Let us consider $\delta_k :=2^{-k}$ then 
%$$
%\{\underline j_n \mid F_{\underline j_n} (x) \in I\} = \bigcup_{k=1}^\infty
%R_{\delta_k}
%$$  
Then, all $R_{\delta}$ corresponding to $\delta > \delta_n$ are empty, and therefore
\[
\{ \underline j_n \mid x \in F_{\underline j_n}(I) \}= \bigcup_{k=0}^{\infty}
R_{2^{-k} \delta_n}. 
\]
Thus
\begin{align*}
\psi_n(x) &= (\Da_{\alpha,\mathcal T})^n \mathds{1}_I (x) = \sum_{\underline j_n}
p_{\underline j_n} ((F_{\underline j_n}^{-1})^\prime(x))^\alpha \mathds{1}_I
(F_{\underline j_n}^{-1}(x)) \\
 &\le \sum_{k=0}^{\infty} \sum_{\underline j_n \in R_{2^{-k}\delta_n}}
 p_{\underline j_n} ((F_{\underline j_n}^{-1})'(x))^{\alpha}
 \mathds{1}_I(F_{\underline j_n}^{-1}(x)) 
 \le \const \cdot \sum_{k=0}^\infty (2^{-k} \delta_n)^{\bar \alpha - \alpha} 
  = \const \cdot \delta_n^{\bar \alpha-\alpha}.
\end{align*} 
Since $\delta_n \to 0$ as $n\to \infty$ and $\bar \alpha > \alpha$ the
estimate~\eqref{eq:psin} follows.

\end{proof}

Discretization of the operator~$\Da_{\alpha,\mathcal T}$ can be defined similarly
to the discretization of the symmetric operator $\Ds_{\alpha, \mathcal S}$ given
by~\eqref{eq:D-Lambda}. Let $\mathcal J$ be a partition of $B_r(J)$ of $N$
intervals. We introduce a non-linear finite rank operator 
 \begin{equation}
  \label{eq:D1-Lambda}
   \Da_{\alpha, \mathcal J} \psi |_{J_k} = 
   \sum_{j=1}^n p_j  \sup_{x\in J_k} |(f_j^{-1})^\prime(x)|^\alpha \cdot \sup_{x\in J_k}  \psi \left( f_j^{-1}(x) \right),
   \quad 1 \le k \le N.
  \end{equation}
An analogue of Proposition~\ref{prop:D21} holds for $\Da_{\alpha, \mathcal J}$ with obvious modifications. 
\begin{proposition}
    \label{prop:D11}
Let $\mu$ be the stationary measure of an iterated function scheme of
diffeomorphisms $\mathcal T(\bar f ,\bar p, J)$. Assume that $\alpha < D_1(\mu)$. 
Then there exist:
\begin{enumerate}
    \item a sufficiently small interval $B_r(J)$; 
    \item a sufficiently fine partition $\mathcal J$ of $B_r(J)$;
    \item a sufficiently large $k \in \mathbb N$; and
    \item a sufficiently small $\vartheta >0$,
\end{enumerate}
so that for $A = \Da_{\alpha,\mathcal J}$ we have that $\hatop A^k
\mathds{1}_{B_r(J)} \prec \mathds{1}_{B_r(J)} $.
\end{proposition}
\begin{proof}
    The argument is along the same lines as the proof of Proposition~\ref{prop:D21}.
    However, in this case we do not know the eigenfunction, so we proceed as
    follows. 

    By Lemma~\ref{lem:find-D1} there exists $r>0$ such that 
   \begin{equation}
       \label{eq:propD11:0}
       \lim_{n\to \infty} \bigl\|(\Da_{\alpha,\mathcal T})^n  \mathds{1\!}_{B_r(J)}
          \bigr\|_\infty = 0.
   \end{equation}
    Let us consider a continuous function~$\psi$ defined by 
    $$
    \psi(x) = 
    \begin{cases}
        1, & \mbox{ if } x \in B_{r/2} (J); \\
        0, & \mbox{ if } x \not\in B_{r}(J); \\
        \mbox{linear}, & \mbox{ otherwise. } 
    \end{cases}
    $$
    Since the operator $\Da_{\alpha,\mathcal T}$ is monotone on~$\mathcal
    F_{B_r(J)}$, and $\psi \preccurlyeq
    \mathds{1}_{B_r(J)}$, it follows from~\eqref{eq:propD11:0} that  
    $$
    \lim_{n\to \infty} \bigl\|(\Da_{\alpha,\mathcal T})^n  \psi  \bigr\|_\infty
    = 0.
    $$
    In particular, we may choose $k$ such that $(\Da_{\alpha,\mathcal
    T})^k  \psi \prec \frac12 \mathds 1_{B_r(J)}$. 

    As in the proof of Proposition~\ref{prop:D21}, the function $\psi$ and all
    its images $(\Da_{\alpha,\mathcal T})^m  \psi$ are continuous. Hence, as the
    size~$\varepsilon$ of intervals of the partition~$\mathcal J$ decreases to
    zero, the images of $\psi$ under the iterations of the finite rank operator converge %approaches the non-discretised one: 
    \[
    \sup_{B_r(J)} \left| (\Da_{\alpha,\mathcal J})^k  \psi - (\Da_{\alpha,\mathcal T})^k  \psi \right|\to 0.
    \]
%    Since the finite rank
%    operators $\Da_{\alpha,\mathcal J}$ depend continuously on the
%    size~$\varepsilon$ of intervals of the partition~$\mathcal J$ we have that 
%    $$
%    \left\| (\Da_{\alpha,\mathcal J} - \Da_{\alpha,\mathcal T}) \psi
%    \right\| \to 0 \mbox{ as } \varepsilon \to 0. 
%    $$
    Let us denote $A:=\Da_{\alpha,\mathcal J}$. Then for a sufficiently fine partition $\mathcal J$, we obtain
    $$
    A^k \psi = (\Da_{\alpha,\mathcal J})^k \psi \prec
    \frac34\mathds{1}_{B_r(J)}.
    $$ 
    The latter implies for ${\hatop A}$ defined by~\eqref{eq:hatop} 
    \[
    \widehat{A}_0^k \psi \preccurlyeq A^k \psi \prec \frac34\mathds{1}_{B_r(J)}.
    \]
    Finally, $(\hatop A)^k \psi$ converges uniformly to $(\widehat A_0)^k \psi$
    as $\vartheta\to 0$, hence for all $\vartheta$ sufficiently small we get
    $\widehat{A}_\vartheta^k \psi \prec \frac56\mathds{1}_{B_r(J)}$.
    On the other hand, by monotonicity ${\hatop A}^k \mathds{1}_{B_{r/2}(J)}
    \preccurlyeq \widehat{A}_\vartheta^k \psi$. Hence everywhere on $B_{r/2}(J)$ we have
    $$ 
    {\hatop A}^k \mathds{1}_{B_{r/2}(J)}(x) < \mathds{1}_{B_{r/2}(J)}(x),
    $$
    and thus with respect to partial order on $\mathcal F_{B_{r/2}(J)}$
    \[
        {\hatop A}^k \mathds{1}_{B_{r/2}(J)} \prec \mathds{1}_{B_{r/2}(J)}.
    \]
    
\end{proof}

Numerical experiments show that the Frostman dimension behaves differently to
correlation dimension and to Hausdorff dimension.  We give estimates for
multinacci numbers for comparison in Table~\ref{tab:frost}.  In particular it appears that 
in the case of Bernoulli convolutions the measure $\mu_\lambda$ corresponding to
the root of $x^4-x^3-x^2-x-1$ has smaller regularity exponent than the measure
corresponding to the root of $x^3 - x^2 -x-1$.

\begin{table}
    \centering
\begin{tabular}{|cccc|}
\hline
$n$ & $\dim_H(\mu_\lambda)$ & $\alpha_2 < D_2(\mu)$ & $\alpha_1 < D_1(\mu)$ \\
$2$ &  $ 0.995713126685$   & $  0.992395833333  $  &  $0.940215301807$ \\
$3$ &  $ 0.980409319534$   & $  0.964214555664  $  &  $0.853037293349$ \\
$4$ &  $ 0.986926474333$   & $  0.973324567994  $  &  $0.844963475586$ \\
$5$ &  $ 0.992585300274$   & $  0.983559570313  $  &  $0.854479046521$ \\
$6$ &  $ 0.996032591584$   & $  0.990673828125  $  &  $0.866685890042$ \\
$7$ &  $ 0.997937445507$   & $  0.994959490741  $  &  $0.880046136101$ \\  
$8$ &  $ 0.998944915449$   & $  0.997343750000  $  &  $0.891195693964$ \\   
$9$ &  $ 0.999465368055$   & $  0.998640046296  $  &  $0.900999615532$ \\   
\hline
\end{tabular}
\caption{Hausdorff dimension and lower bounds on correlation and Frostman
dimension for the multinacci parameter values, i.e. the largest roots of $x^n -
x^{n-1} - \ldots -x -1$. }
\label{tab:frost}
\end{table}

\begin{remark}
Nevertheless, the nonsymmetric operator can be applied to Bernoulli convolutions
mesures to show that at the other end of the range of parameters, there exists
$c > 0$ and $ \varepsilon > 0$ such that 
$\dim_H(\mu_\lambda) \geq  1-\frac{c}{\log (\lambda-\frac{1}{2})^{-1}}  $ for $\frac{1}{2}<
\lambda < \frac{1}{2}+ \varepsilon$.
Indeed, it suffices to apply $N=\lfloor\log_2 (\lambda-\frac{1}{2})^{-1}\rfloor$
iterations of $\Da_{\alpha,\mS}$ to the initial function $\mathds{1}_{[-0.1,2.1]}$. 
The scalar factor $\lambda^{N\alpha}$ will be no larger than $\frac14$, while each point
of $J=[-0.1,2.1]$ will be covered by at most two images of~$J$ under composition
of~$n$  maps of the iterated function scheme which corresponds to
Bernoulli convolutions as defined in Example~\ref{ex:BCsystem} 
$f_{0}(x)=\lambda x$, and $f_1(x)=\lambda x-1$.
\end{remark}

%\begin{remark}
%The same results apply verbatim for the holomorphic contracting maps in the
%complex plane: the distortion control is still valid. 
%\end{remark}

\appendix
\section{Numerical data}
\subsection{Salem numbers of degree up to~$10$.} 
\label{ap:salem}
%\begin{landscape}
    A lower bound $\alpha < D_2(\mu)$ for the correlation dimension of the Bernoulli convolution
    measure corresponding to Salem parameter values of degree up to~$10$. 
    Computed using~$300$ iterations of the diffusion operator and~$6\cdot 10^6$
    partition intervals.

    Since the coefficients form a palindromic sequence, i.e. they read the same
    backward and forward, we give only the coefficients of the first
    half of each polynomial, from the leading coefficient to the middle
    coefficient.

    \begin{longtable}{|c|c|c|c|l|}
        \hline                
        %    \begin{tabular}{ccccc}
        $\lambda $   &       $ \alpha $     &  degree   &  $\beta = \lambda^{-1}$   &
        coefficients \\
        \hline                
        \endhead
        \hline
        \endfoot
        \endlastfoot
        $ 0.58069183  $ & $     0.997968750   $ & $      4  $ & $  1.72208380  $ & $    1,-1,-1,-1, 1 $ \\
        %\hline
        $ 0.53101005  $ & $     0.990312500   $ & $      4  $ & $  1.88320350  $ & $    1,-2, 1,-2, 1 $ \\
        \hline
        \hline
        $ 0.71363917  $ & $     0.999687500   $ & $      6  $ & $  1.40126836  $ & $    1, 0,-1,-1 $ \\
        %\hline
        $ 0.66395080  $ & $     0.999765625   $ & $      6  $ & $  1.50613567  $ & $    1,-1, 0,-1 $ \\
        %\hline
        $ 0.64266105  $ & $     0.999765625   $ & $      6  $ & $  1.55603019  $ & $    1,-1,-1, 1 $ \\
        %\hline
        $ 0.63197255  $ & $     0.999765625   $ & $      6  $ & $  1.58234718  $ & $    1, 0,-1,-2 $ \\
        %\hline
        $ 0.61140647  $ & $     0.999062500   $ & $      6  $ & $  1.63557312  $ & $    1,-2, 2,-3 $ \\
        %\hline
        $ 0.56127948  $ & $     0.997187500   $ & $      6  $ & $  1.78164359  $ & $    1,-1,-1, 0 $ \\
        %\hline
        $ 0.54612702  $ & $     0.994218750   $ & $      6  $ & $  1.83107582  $ & $    1,-2, 0, 1 $ \\
        %\hline
        $ 0.51364860  $ & $     0.988671875   $ & $      6  $ & $  1.94685626  $ & $    1,-1,-1,-1 $ \\
        %\hline
        $ 0.50928087  $ & $     0.988828125   $ & $      6  $ & $  1.96355303  $ & $    1,-2,-1, 3 $ \\
        %\hline
        $ 0.50637559  $ & $     0.991406250   $ & $      6  $ & $  1.97481870  $ & $    1,-2, 1,-2 $ \\
        %\hline
        $ 0.50307044  $ & $     0.993828125   $ & $      6  $ & $  1.98779316  $ & $    1, 0,-2,-3 $ \\
        \hline
        \hline
        $ 0.78086069  $ & $     0.999609375   $ & $      8  $ & $  1.28063815  $ & $    1, 0, 0,-1,-1 $ \\
        %\hline
        $ 0.73529427  $ & $     0.999687500   $ & $      8  $ & $  1.35999971  $ & $    1,-1, 1,-2, 1 $ \\
        %\hline
        $ 0.70175179  $ & $     0.999687500   $ & $      8  $ & $  1.42500526  $ & $    1,-1, 0,-1, 1 $ \\
        %\hline
        $ 0.68587694  $ & $     0.999687500   $ & $      8  $ & $  1.45798747  $ & $    1, 0,-1,-1, 0 $ \\
        %\hline
        $ 0.65657284  $ & $     0.999765625   $ & $      8  $ & $  1.52306024  $ & $    1,-1,-1, 0, 1 $ \\
        %\hline
        $ 0.64633011  $ & $     0.999765625   $ & $      8  $ & $  1.54719696  $ & $    1,-2, 2,-3, 3 $ \\
        %\hline
        $ 0.62287838  $ & $     0.999687500   $ & $      8  $ & $  1.60544982  $ & $    1,-2, 1, 0,-1 $ \\
        %\hline
        $ 0.60974342  $ & $     0.999375000   $ & $      8  $ & $  1.64003408  $ & $    1, 0,-2,-1, 1 $ \\
        %\hline
        $ 0.60202964  $ & $     0.999062500   $ & $      8  $ & $  1.66104776  $ & $    1,-2, 1,-1, 1 $ \\
        %\hline
        $ 0.59350345  $ & $     0.998828125   $ & $      8  $ & $  1.68491015  $ & $    1,-1,-1, 0, 0 $ \\
        %\hline
        $ 0.59049048  $ & $     0.999296875   $ & $      8  $ & $  1.69350738  $ & $    1,-1, 0,-1,-1 $ \\
        %\hline
        $ 0.55677045  $ & $     0.997343750   $ & $      8  $ & $  1.79607232  $ & $    1,-1,-1, 0,-1 $ \\
        %\hline
        $ 0.55550255  $ & $     0.997968750   $ & $      8  $ & $  1.80017173  $ & $    1,-3, 4,-5, 5 $ \\
        %\hline
        $ 0.55255049  $ & $     0.996015625   $ & $      8  $ & $  1.80978933  $ & $    1,-1, 0,-2, 0 $ \\
        %\hline
        $ 0.55199367  $ & $     0.995078125   $ & $      8  $ & $  1.81161496  $ & $    1,-2, 0, 1,-1 $ \\
        %\hline
        $ 0.54499335  $ & $     0.993671875   $ & $      8  $ & $  1.83488477  $ & $    1, 0,-1,-2,-3 $ \\
        %\hline
        $ 0.54065766  $ & $     0.996718750   $ & $      8  $ & $  1.84959921  $ & $    1, 1,-1,-4,-5 $ \\
        %\hline
        $ 0.53646341  $ & $     0.996250000   $ & $      8  $ & $  1.86406000  $ & $    1,-1,-2, 0, 2 $ \\
        %\hline
        $ 0.52178497  $ & $     0.991250000   $ & $      8  $ & $  1.91649826  $ & $    1,-1,-1,-1, 0 $ \\
        %\hline
        $ 0.52066307  $ & $     0.992421875   $ & $      8  $ & $  1.92062783  $ & $    1,-3, 3,-2, 1 $ \\
        %\hline
        $ 0.51899824  $ & $     0.985468750   $ & $      8  $ & $  1.92678880  $ & $    1, 0,-2,-2,-1 $ \\
        %\hline
        $ 0.51142999  $ & $     0.992187500   $ & $      8  $ & $  1.95530180  $ & $    1,-2, 0,-1, 3 $ \\
        %\hline
        $ 0.50150345  $ & $     0.996250000   $ & $      8  $ & $  1.99400419  $ & $    1,-2, 1,-2, 1 $ \\
        \hline
        \hline
        $ 0.85013713  $ & $     0.999375000   $ & $      10 $ & $   1.1762808  $ & $     1, 1, 0,-1,-1,-1$ \\
        %\hline
        $ 0.82210362  $ & $     0.999531250   $ & $      10 $ & $   1.2163916  $ & $     1, 0, 0, 0,-1,-1$ \\
        %\hline
        $ 0.81274948  $ & $     0.999531250   $ & $      10 $ & $   1.2303914  $ & $     1, 0, 0,-1, 0,-1$ \\
        %\hline
        $ 0.79287619  $ & $     0.999609375   $ & $      10 $ & $   1.2612309  $ & $     1, 0,-1, 0, 0,-1$ \\
        %\hline
        $ 0.77310464  $ & $     0.999609375   $ & $      10 $ & $   1.2934859  $ & $     1, 0,-1,-1, 0, 1$ \\
        %\hline
        $ 0.74776798  $ & $     0.999687500   $ & $      10 $ & $   1.3373132  $ & $     1,-1, 0, 0, 0,-1$ \\
        %\hline
        $ 0.74020322  $ & $     0.999687500   $ & $      10 $ & $   1.3509803  $ & $     1,-1, 0, 0,-1, 1$ \\
        %\hline
        $ 0.72273314  $ & $     0.999687500   $ & $      10 $ & $   1.3836365  $ & $     1,-1, 0,-1, 1,-1$ \\
        %\hline
        $ 0.69881155  $ & $     0.999687500   $ & $      10 $ & $   1.4310009  $ & $     1,-1,-1, 1, 0,-1$ \\
        %\hline
        $ 0.69040602  $ & $     0.999687500   $ & $      10 $ & $   1.4484230  $ & $     1,-2, 2,-2, 1,-1$ \\
        %\hline
        $ 0.67918489  $ & $     0.999687500   $ & $      10 $ & $   1.4723531  $ & $     1,-1, 0, 0,-1, 0$ \\
        %\hline
        $ 0.67585437  $ & $     0.999687500   $ & $      10 $ & $   1.4796086  $ & $     1, 0,-2,-2, 1, 3$ \\
        %\hline
        $ 0.66056171  $ & $     0.999765625   $ & $      10 $ & $   1.5138630  $ & $     1, 0, 0,-1,-2,-1$ \\
        %\hline
        $ 0.65234742  $ & $     0.999765625   $ & $      10 $ & $   1.5329254  $ & $     1,-1,-1, 0, 0, 1$ \\
        %\hline
        $ 0.62865265  $ & $     0.999765625   $ & $      10 $ & $   1.5907035  $ & $     1,-2, 1, 0,-2, 3$ \\
        %\hline
        $ 0.62617211  $ & $     0.999765625   $ & $      10 $ & $   1.5970050  $ & $     1, 0,-1,-1,-1,-1$ \\
        %\hline
        $ 0.62552337  $ & $     0.999687500   $ & $      10 $ & $   1.5986612  $ & $     1,-2, 1,-1, 2,-3$ \\
        %\hline
        $ 0.61538018  $ & $     0.999140625   $ & $      10 $ & $   1.6250117  $ & $     1,-1,-1,-1, 1, 1$ \\
        %\hline
        $ 0.61442227  $ & $     0.999218750   $ & $      10 $ & $   1.6275451  $ & $     1,-2, 0, 2,-1,-1$ \\
        %\hline
        $ 0.60768664  $ & $     0.998984375   $ & $      10 $ & $   1.6455849  $ & $     1,-1,-1, 0, 0, 0$ \\
        %\hline
        $ 0.60335337  $ & $     0.999218750   $ & $      10 $ & $   1.6574035  $ & $     1, 1, 0,-2,-4,-5$ \\
        %\hline
        $ 0.59905359  $ & $     0.999296875   $ & $      10 $ & $   1.6692997  $ & $     1,-1, 0,-1, 0,-2$ \\
        %\hline
        $ 0.59769079  $ & $     0.998906250   $ & $      10 $ & $   1.6731059  $ & $     1,-2, 1,-1, 0, 1$ \\
        %\hline
        $ 0.59165468  $ & $     0.998984375   $ & $      10 $ & $   1.6901750  $ & $     1,-1,-2, 1, 1,-1$ \\
        %\hline
        $ 0.58364838  $ & $     0.998828125   $ & $      10 $ & $   1.7133603  $ & $     1,-2, 1, 0,-2, 2$ \\
        %\hline
        $ 0.57572460  $ & $     0.998984375   $ & $      10 $ & $   1.7369415  $ & $     1,-1, 0,-1,-1,-1$ \\
        %\hline
        $ 0.57309152  $ & $     0.998750000   $ & $      10 $ & $   1.7449219  $ & $     1,-2, 2,-3, 2,-3$ \\
        %\hline
        $ 0.57273303  $ & $     0.998906250   $ & $      10 $ & $   1.7460141  $ & $     1,-1,-1,-1, 0, 2$ \\
        %\hline
        $ 0.57086342  $ & $     0.998515625   $ & $      10 $ & $   1.7517324  $ & $     1,-2, 1,-1, 1,-2$ \\
        %\hline
        $ 0.57041893  $ & $     0.998359375   $ & $      10 $ & $   1.7530974  $ & $     1, 0,-1,-1,-2,-3$ \\
        %\hline
        $ 0.56813176  $ & $     0.998515625   $ & $      10 $ & $   1.7601550  $ & $     1,-2, 1,-2, 2,-1$ \\
        %\hline
        $ 0.56689143  $ & $     0.998750000   $ & $      10 $ & $   1.7640061  $ & $     1,-2, 0, 1, 0,-1$ \\
        %\hline
        $ 0.56595979  $ & $     0.998515625   $ & $      10 $ & $   1.7669099  $ & $     1, 0,-2,-2, 0, 1$ \\
        %\hline
        $ 0.56479027  $ & $     0.998906250   $ & $      10 $ & $   1.7705687  $ & $     1,-3, 4,-5, 5,-5$ \\
        %\hline
        $ 0.55915695  $ & $     0.997500000   $ & $      10 $ & $   1.7884066  $ & $     1,-1, 0,-2, 0,-1$ \\
        %\hline
        $ 0.55562233  $ & $     0.997890625   $ & $      10 $ & $   1.7997836  $ & $     1, 0,-3,-3, 2, 5$ \\
        %\hline
        $ 0.55401132  $ & $     0.996015625   $ & $      10 $ & $   1.8050172  $ & $     1,-2, 0, 1,-1, 1$ \\
        %\hline
        $ 0.54829524  $ & $     0.995312500   $ & $      10 $ & $   1.8238348  $ & $     1, 0,-1,-2,-2,-2$ \\
        %\hline
        $ 0.54790081  $ & $     0.992734375   $ & $      10 $ & $   1.8251478  $ & $     1,-1,-2, 0, 1, 1$ \\
        %\hline
        $ 0.54102298  $ & $     0.994531250   $ & $      10 $ & $   1.8483503  $ & $     1,-1,-1, 0,-1,-1$ \\
        %\hline
        $ 0.53962808  $ & $     0.995625000   $ & $      10 $ & $   1.8531281  $ & $     1,-2, 1,-2, 2,-2$ \\
        %\hline
        $ 0.53846598  $ & $     0.997109375   $ & $      10 $ & $   1.8571275  $ & $     1,-2, 1,-1, 0,-1$ \\
        %\hline
        $ 0.53687119  $ & $     0.995156250   $ & $      10 $ & $   1.8626441  $ & $     1,-3, 3,-1,-3, 5$ \\
        %\hline
        $ 0.53511394  $ & $     0.996015625   $ & $      10 $ & $   1.8687608  $ & $     1,-1, 0,-2, 0,-3$ \\
        %\hline
        $ 0.53312307  $ & $     0.994765625   $ & $      10 $ & $   1.8757394  $ & $     1,-2,-1, 3, 0,-3$ \\
        %\hline
        $ 0.52911051  $ & $     0.993437500   $ & $      10 $ & $   1.8899643  $ & $     1,-1,-1, 0,-1,-2$ \\
        %\hline
        $ 0.52809402  $ & $     0.995234375   $ & $      10 $ & $   1.8936021  $ & $     1,-1,-2, 0, 1, 0$ \\
        %\hline
        $ 0.52724900  $ & $     0.994453125   $ & $      10 $ & $   1.8966370  $ & $     1,-2, 0, 1,-1, 0$ \\
        %\hline
        $ 0.52656371  $ & $     0.993671875   $ & $      10 $ & $   1.8991054  $ & $     1,-2, 0, 0, 1,-1$ \\
        %\hline
        $ 0.52476315  $ & $     0.989687500   $ & $      10 $ & $   1.9056215  $ & $     1, 0,-1,-2,-3,-3$ \\
        %\hline
        $ 0.52402623  $ & $     0.991171875   $ & $      10 $ & $   1.9083013  $ & $     1,-1,-1,-1, 0,-1$ \\
        %\hline
        $ 0.52325199  $ & $     0.988750000   $ & $      10 $ & $   1.9111250  $ & $     1, 0,-2,-2,-1,-1$ \\
        %\hline
        $ 0.52234319  $ & $     0.992578125   $ & $      10 $ & $   1.9144501  $ & $     1,-1, 0,-1,-3,-1$ \\
        %\hline
        $ 0.51919283  $ & $     0.987187500   $ & $      10 $ & $   1.9260666  $ & $     1,-1,-1,-1,-1, 1$ \\
        %\hline
        $ 0.51815962  $ & $     0.991093750   $ & $      10 $ & $   1.9299072  $ & $     1,-1,-1,-2, 1, 0$ \\
        %\hline
        $ 0.51751378  $ & $     0.993515625   $ & $      10 $ & $   1.9323156  $ & $     1,-2, 2,-4, 3,-5$ \\
        %\hline
        $ 0.51734313  $ & $     0.992109375   $ & $      10 $ & $   1.9329530  $ & $     1,-3, 3,-2, 0, 1$ \\
        %\hline
        $ 0.51642814  $ & $     0.993359375   $ & $      10 $ & $   1.9363778  $ & $     1,-2, 0, 0, 0, 1$ \\
        %\hline
        $ 0.51544868  $ & $     0.993125000   $ & $      10 $ & $   1.9400573  $ & $     1,-2, 1,-1,-1, 0$ \\
        %\hline
        $ 0.51282383  $ & $     0.993125000   $ & $      10 $ & $   1.9499873  $ & $     1,-1,-2,-1, 1, 3$ \\
        %\hline
        $ 0.50707588  $ & $     0.992890625   $ & $      10 $ & $   1.9720914  $ & $     1,-2, 0, 0,-1, 3$ \\
        %\hline
        $ 0.50074301  $ & $     0.997890625   $ & $      10 $ & $   1.9970323  $ & $     1,-1,-1,-1,-1,-1$ \\
        %\hline
        $ 0.50036910  $ & $     0.998828125   $ & $      10 $ & $   1.9985246  $ & $     1,-2, 1,-2, 1,-2$ \\
        \hline
        %\end{tabular}
        \end{longtable}
   % \end{landscape}

\subsection{Small Salem numbers}
\label{ap:smallsalem}
    A lower bound $\alpha < D_2(\mu)$ for the correlation dimension of the Bernoulli convolution
    measure corresponding Salem parameter values of degree up to~$10$. 
    Computed using~$300$ iterations of the diffusion operator and~$6\cdot 10^6$
    partition intervals.

%\begin{longtable}[h!]
%    \begin{center}
    \begin{longtable}{|c|c|c|c|}
\hline
$\lambda$ & $\alpha$ & degree & $\beta=\lambda^{-1}$        \\
\hline
\endhead
        \hline
        \endfoot
        \endlastfoot
$  0.841490073675    $   &    $   0.999453125000  $  & $ 18  $   &   $ 1.188368147508  $  \\
%\hline
$  0.833314914305    $   &    $   0.999453125000  $  & $ 14  $   &   $ 1.200026523987  $  \\
%\hline
$  0.831520104180    $   &    $   0.999453125000  $  & $ 14  $   &   $ 1.202616743688  $  \\
%\hline
$  0.819859718384    $   &    $   0.999531250000  $  & $ 18  $   &   $ 1.219720859040  $  \\
%\hline
$  0.811284283822    $   &    $   0.999531250000  $  & $ 20  $   &   $ 1.232613548593  $  \\
%\hline
$  0.809281107406    $   &    $   0.999531250000  $  & $ 22  $   &   $ 1.235664580389  $  \\
%\hline
$  0.808853430235    $   &    $   0.999531250000  $  & $ 16  $   &   $ 1.236317931803  $  \\
%\hline
$  0.808077659864    $   &    $   0.999531250000  $  & $ 26  $   &   $ 1.237504821217  $  \\
%\hline
$  0.805979449568    $   &    $   0.999531250000  $  & $ 12  $   &   $ 1.240726423652  $  \\
%\hline
$  0.798227336699    $   &    $   0.999531250000  $  & $ 18  $   &   $ 1.252775937410  $  \\
%\hline
$  0.797874048512    $   &    $   0.999531250000  $  & $ 20  $   &   $ 1.253330650201  $  \\
%\hline
$  0.796753378647    $   &    $   0.999609375000  $  & $ 14  $   &   $ 1.255093516763  $  \\
%\hline
$  0.796038178870    $   &    $   0.999609375000  $  & $ 18  $   &   $ 1.256221154391  $  \\
%\hline
$  0.793585580847    $   &    $   0.999609375000  $  & $ 24  $   &   $ 1.260103540354  $  \\
%\hline
$  0.793471798443    $   &    $   0.999609375000  $  & $ 22  $   &   $ 1.260284236896  $  \\
%\hline
$  0.791741728445    $   &    $   0.999609375000  $  & $ 26  $   &   $ 1.263038139930   $  \\
%\hline
$  0.789081359693    $   &    $   0.999609375000  $  & $ 14  $   &   $ 1.267296442523  $  \\
%\hline
$  0.783220488506    $   &    $   0.999609375000  $  & $ 22  $   &   $ 1.276779674019   $  \\
%\hline
$  0.780219031051    $   &    $   0.999609375000  $  & $ 26  $   &   $ 1.281691371528   $  \\
%\hline
$  0.779729794543    $   &    $   0.999609375000  $  & $ 20  $   &   $ 1.282495560639  $  \\
%\hline
$  0.778442407020    $   &    $   0.999609375000  $  & $ 18  $   &   $ 1.284616550925  $  \\
%\hline
$  0.778363474601    $   &    $   0.999609375000  $  & $ 26  $   &   $ 1.284746821544  $  \\
%\hline
$  0.778149945665    $   &    $   0.999609375000  $  & $ 30  $   &   $ 1.285099363651   $  \\
%\hline
$  0.778136529751    $   &    $   0.999609375000  $  & $ 30  $   &   $ 1.285121520153  $  \\
%\hline
$  0.778097688729    $   &    $   0.999609375000  $  & $ 30  $   &   $ 1.285185670752  $  \\
%\hline
$  0.778090995075    $   &    $   0.999609375000  $  & $ 26  $   &   $ 1.285196726769  $  \\
%\hline
$  0.778089510311    $   &    $   0.999609375000  $  & $ 44  $   &   $ 1.285199179205  $  \\
%\hline
$  0.778067560084    $   &    $   0.999609375000  $  & $ 30  $   &   $ 1.285235436228  $  \\
%\hline
$  0.777962461454    $   &    $   0.999609375000  $  & $ 34  $   &   $ 1.285409064765  $  \\
%\hline
$  0.777365621302    $   &    $   0.999609375000  $  & $ 18  $   &   $ 1.286395966836  $  \\
%\hline
$  0.777163708407    $   &    $   0.999609375000  $  & $ 26  $   &   $ 1.286730182048  $  \\
%\hline
$  0.774148742227    $   &    $   0.999609375000  $  & $ 24  $   &   $ 1.291741425714  $  \\
%\hline
$  0.773970381657    $   &    $   0.999609375000  $  & $ 20  $   &   $ 1.292039106017  $  \\
%\hline
$  0.773743085596    $   &    $   0.999609375000  $  & $ 40  $   &   $ 1.292418657582  $  \\
%\hline
$  0.773454591798    $   &    $   0.999609375000  $  & $ 46  $   &   $ 1.292900721780  $  \\
%\hline
$  0.771798261859    $   &    $   0.999609375000  $  & $ 18  $   &   $ 1.295675371944  $  \\
%\hline
$  0.771479537372    $   &    $   0.999609375000  $  & $ 34  $   &   $ 1.296210659593  $  \\
%\hline
$  0.771354149852    $   &    $   0.999609375000  $  & $ 22  $   &   $ 1.296421365194  $  \\
%\hline
$  0.771116223305    $   &    $   0.999609375000  $  & $ 28  $   &   $ 1.296821373714  $  \\
%\hline
$  0.770160984197    $   &    $   0.999609375000  $  & $ 36  $   &   $ 1.298429835475  $  \\
%\hline
$  0.769381763673    $   &    $   0.999609375000  $  & $ 26  $   &   $ 1.299744869472  $  \\
\hline
\end{longtable} 

\pagebreak

  {\footnotesize
  \noindent
  \textsc{V. Kleptsyn, Univ Rennes, CNRS, IRMAR - UMR 6625, F-35000 Rennes,
  France.  } \\
  \noindent
  \textit{E-mail address}: \texttt{victor.kleptsyn@univ-rennes1.fr}
  \par
  \addvspace{\medskipamount}
  \noindent
  \textsc{M. Pollicott, Department of Mathematics, Warwick University, Coventry,
  CV4 7AL, UK.} \\
  \noindent
  \textit{E-mail address}: \texttt{masdbl@warwick.ac.uk}
  \par
  \addvspace{\medskipamount}
  \noindent
  \textsc{P. Vytnova, Department of Mathematics, Warwick University, Coventry,
  CV4 7AL, UK} \\
  \noindent
  \textit{E-mail address}: \texttt{P.Vytnova@warwick.ac.uk}
}

\end{document}